\documentclass[11pt]{amsart}
\usepackage[latin1]{inputenc}
\usepackage{amsthm}
\usepackage{amssymb}

\makeatletter

\numberwithin{equation}{section}
\numberwithin{figure}{section}

\usepackage{amsthm}\usepackage{amsfonts}
\usepackage[dvips]{graphics}\usepackage{color}

\usepackage{mathrsfs,cancel,comment,enumitem,bm}\usepackage[all,knot]{xy}
\allowdisplaybreaks
\newtheorem{theorem}{Theorem}

\newtheorem{lemma}[theorem]{Lemma}

\newtheorem{proposition}[theorem]{Proposition}

\theoremstyle{definition}

\theoremstyle{remark}

\newtheorem*{remark}{Remark}

\numberwithin{theorem}{section}

\newcommand{\ord}{\text {\rm ord}}

\newcommand{\C}{\mathbb{C}}

\newcommand{\Q}{\mathbb{Q}}

\newcommand{\Z}{\mathbb{Z}}
\newcommand{\N}{\mathbb{N}}

\newcommand{\SL}{{\text {\rm SL}}}

\newcommand{\lcm}{{\text {\rm lcm}}}

\def\H{\mathbb{H}}

\newcommand{\gen}{\text{gen}}
\newcommand{\pgen}{\text{gen}^+}

\newcommand{\pspn}{\text{spn}^+}
\newcommand{\cls}{\text{cls}}
\newcommand{\pcls}{\text{cls}^+}
\newcommand{\T}[1]{\Theta_{#1}}
\newcommand{\legendre}[2]{\left( \frac{#1}{#2} \right)}

\newcommand{\spnnorm}[1]{\theta\left(O^+({#1})\right)}

\newcommand{\erfc}{\operatorname{erfc}}

\renewcommand{\sf}{\operatorname{sf}}
\renewcommand{\phi}{\varphi}

\newif \ifdetails 

\detailsfalse

\newif\ifdiscs 

\discsfalse

\newenvironment{extradetails}{\ifdetails \noindent  \newline \noindent \bf
	************************** Begin Extra Details **************************
	\rm \\ \bf\fi } {\ifdetails \noindent \newline
	\noindent \bf  ************************** End Extra Details ***************************\rm\\ \fi }

\ifdiscs
\else
\excludecomment{discussion}
\fi

\ifdetails
\else
\excludecomment{extradetails}
\fi

\makeatother
\title{Sums of generalized polygonal numbers of almost prime ``length''}
\author{Soumyarup Banerjee}
\address{Department of Mathematics, Indian Institute of technology Kharagpur, Kharagpur, West Bengal - 721302, India}
\email{soumya.tatan@gmail.com}
\author{Ben Kane}
\address{Mathematics Department, The University of Hong Kong, Pokfulam, HongKong}
\email{bkane@hku.hk}
\author{Daejun Kim}
\address{Department of Mathematics Education, Korea University, Seoul 02841, Republic of Korea}
\email{daejunkim@korea.ac.kr}
\date{\today}
\thanks{The research of the first author was supported by INSPIRE Faculty Research Grant by DST India, Grant no:  DST/INSPIRE/04/2021/002753. The research of the second author was supported by grants from the Research Grants Council of the Hong Kong SAR, China (project numbers HKU 17314122 and HKU 17305923). This work of the third author was supported by the National Research Foundation of Korea(NRF) grant funded by the Korea government(MSIT) (RS-2024-00455692).}
\begin{document}
\begin{abstract}
In this paper, we consider sums of three generalized $m$-gonal numbers whose parameters are restricted to integers with a bounded number of prime divisors. With some restrictions on $m$ modulo $30$, we show that a density one set of integers is represented as such a sum, where the parameters are restricted to have at most 6361 prime factors. Moreover, if the squarefree part of $f_m(n)$ is sufficiently large, then $n$ is represented as such a sum, where $f_m(n)$ is a natural linear function in $n$.
\end{abstract}
\maketitle

\section{Introduction}
For $m\in\N_{\geq 3}$ and $x\in\Z$, set 
\[
p_m(x):=\frac{(m-2)x^2}{2} -\frac{(m-4)x}{2}.
\]
If $x\in\N$, then the number $p_m(x)$ counts the number of dots in a regular $m$-gon of length $x$ when $x\in\N$. More generally, for $x\in\Z$ we call these \begin{it}generalized $m$-gonal numbers\end{it}. For $n\in\N$, consider the equation 
\begin{equation}\label{eqn:sum3mgonal}
p_m(x)+p_m(y)+p_m(z)=n.
\end{equation}
For $4\nmid m$ and $m\not \equiv 2\pmod{3}$, it was shown in \cite{HaenschKane} that for $n$ sufficiently large \eqref{eqn:sum3mgonal} is solvable with $x,y,z\in\Z$. If we further restrict to $m$ odd, one can show that the genus and spinor genus of a related shifted lattice coincide, and then the result in \cite{KaneKim} may be applied to show that the restriction $m\not\equiv 2\pmod{3}$ is not necessary (see Theorem \ref{thm:3mgonaloddalmostuniversal}). We hence restrict ourselves to the case when $m$ is odd. 

To understand the relationship with shifted lattices, we complete the square so that \eqref{eqn:sum3mgonal} becomes
\[
\left( 2(m-2)x +4-m\right)^2+\left( 2(m-2)y +4-m\right)^2+\left( 2(m-2)z +4-m\right)^2=8(m-2)n+3(m-4)^2.
\]
So \eqref{eqn:sum3mgonal} is equivalent to 
\begin{equation}\label{eqn:sum3squaresCongruence}
X^2+Y^2+Z^2=8(m-2)n+3(m-4)^2,
\end{equation}
where $X=2(m-2)x+4-m$, $Y=2(m-2)y+4-m$, and $Z=2(m-2)z+4-m$. Taking $L=2(m-2)\Z^3$ and $\nu=(4-m,4-m,4-m)\in\Z^3$, \eqref{eqn:sum3squaresCongruence} is related to counting points of a given Euclidean distance on the shifted lattice $L+\nu$.

In addition to our assumption that $m$ is odd, for convenience we assume throughout that $m\not\equiv 1\pmod{3}$ and $m\not\equiv 4\pmod{5}$, due to the following technical issue.  Suppose that $m-4 \equiv 0 \pmod{3}$ or $m-4 \equiv 0 \pmod{5}$. Then $X=2(m-2)x+4-m \not \equiv 0 \pmod{p}$ is equivalent to saying that $x\not\equiv0\pmod{p}$. Hence the integers $n$ which can be represented as \eqref{eqn:sum3mgonal} with $x,y,z\not\equiv 0 \pmod{p}$ is restricted: when $p=3$, only those $n$ with $n\equiv 0\pmod{p}$, when $p=5$ only those $n$ with $n\not\equiv 0\pmod{p}$. 

In this paper, we consider the restriction of solutions to \eqref{eqn:sum3mgonal} to a thinner subset of the integers; namely, we fix $N\in\N$ and restrict $x$, $y$, and $z$ to be almost primes of order $N$ (i.e, the number of primes dividing each of $x$, $y$, and $z$ is bounded by $N$, counting multiplicities). For example, allowing $|x|,|y|,|z|$ to be $0$, $1$, or a prime, computational evidence indicates that for $m=3$ the equation \eqref{eqn:sum3mgonal} is solvable for sufficiently large $n$ (up to $8\cdot 10^6$, the equation is solvable for more than $99.65\%$ of the integers). Further computations lead to a more precise version of Gauss's Eureka Theorem, which states that \eqref{eqn:sum3mgonal} is solvable with $x,y,z\in\Z$ for all $n\in\N$; namely, it seems as though 
\[
p_3(x)+p_3(y)+p_3(z)=n
\]
is solvable for all $n\in\N$ with $|x|$, $|y|$ and $|z|$ having (counting multiplicities) at most $2$ prime factors (or equalling zero), while it appears to be solvable for $n\geq 466594$ (with only $171$ exceptional $n$) with $x,y,z\geq 0$ having at most $2$ prime factors (or equalling zero). Although these claims seem to be out of reach with current methods, under the above restrictions on $m$, we are able to show that the set of $n\in\N$ for which \eqref{eqn:sum3mgonal} is solvable with almost prime $|x|$, $|y|$, and $|z|$ is a density one set.
\begin{theorem}\label{thm:main}
Suppose that $m$ is odd, $m-4\not\equiv 0\pmod{3}$ and $m-4\not\equiv 0\pmod{5}$. 
\begin{enumerate}[leftmargin=*,label=\rm(\arabic*)]
\item There is a set $S\subseteq \N$ of density one, such that for $n\in S$ sufficiently large, the equation \eqref{eqn:sum3mgonal} is solvable with each of $x$, $y$, and $z$ containing at most $6361$ prime factors (counting multiplicities).

 In particular, letting $\sf(h)$ denote the squarefree part of $h$, we may take 
\[
S=\left\{n\in\N: \sf\left(8(m-2)n+3(m-4)^2\right)\geq \log\left(8(m-2)n+3(m-4)^2\right)^7\right\}.
\]
\item Moreover, there exists a constant $C$ such that if, for some $\varepsilon>0$, we have
\[
2^a\sf\left(8(m-2)n+3(m-4)^2\right)^{2-\varepsilon}\geq C\log\left(8(m-2)n+3(m-4)^2\right)^{13},
\]
then \eqref{eqn:sum3mgonal} is solvable with each of $x$, $y$, and $z$ containing at most $6359$ odd prime factors (counting multiplicity) and $\ord_2(x),\ord_2(y),\ord_2(z)\leq a$.
\end{enumerate}

\end{theorem}
\begin{remark}
Due to the ineffective bounds for class numbers of imaginary quadratic fields, the results are ineffective, so one cannot obtain explicit results, even for fixed $m$.
\end{remark}
Results resembling Theorem \ref{thm:main} exist for sums of four and three squares (see Br\"udern--Fouvry \cite{BrudernFouvry} and Blomer--Br\"udern \cite{BlomerBrudern}, respectively). The proof of Theorem \ref{thm:main} differs from the proofs in \cite{BrudernFouvry} and \cite{BlomerBrudern} because the congruence conditions on $X$, $Y$, and $Z$ in \eqref{eqn:sum3squaresCongruence} lead to possible issues caused by obstructions in what is known as the spinor genus of a quadratic form with congruence conditions (from an alternative perspective, the issue arises from the fact that coefficients of weight $\frac{3}{2}$ unary theta functions, defined in \eqref{eqn:unarydef}, do not satisfy the Ramanujan--Petersson bound). In the four variable case considered by Br\"udern and Fouvry, the spinor genus never causes complication. In the case considered by Blomer and Br\"udern, the spinor genus always coincides with the genus of the corresponding quadratic form (and hence the cuspidal part of the theta function is orthogonal to the subspace of unary theta functions). However, in the case that we consider, we cannot exclude the contribution from the spinor genus in such a direct way, and the complications that arise lead to technical difficulties; this is the reason why we cannot bound the number of primes dividing $x$, $y$ and $z$ in \eqref{eqn:sum3mgonal} for all $n\in\N$, only obtaining our main result for a density one set.

The paper is organized as follows. In Section \ref{sec:prelim}, we recall the theory of quadratic lattices and shifted lattices and the interrelation between modular forms and theta series. In Section \ref{sec:Eisenstein}, we compute local densities for the Eisenstein series component of the theta series and obtain lower bounds for the Fourier coefficients of the Eisenstein series component. In Section \ref{sec:unary}, we determine a bound on the Fourier coefficients which come from the contribution of unary theta functions; on the algebraic side, these are the contribution from the spinor genus. The Fourier coefficients coming from the contribution of the cuspidal part of the theta function are bounded in Section \ref{sec:cuspform}. In Section \ref{sec:mainerror}, we obtain bounds to set up a sieve needed in the proof of Theorem \ref{thm:main}. We apply sieving techniques and prove Theorem \ref{thm:main} in Section \ref{sec:final}.

\section{Preliminaries}\label{sec:prelim}

\subsection{Setup and notation}
Throughout this paper, we suppose that $m\equiv 1\pmod{2}$. Let $V$ be a positive definite quadratic space over $\Q$ of rank $k$ with an associated non-degenerate symmetric bilinear form 
\[
B:V\times V \rightarrow \Q \quad \text{with} \quad Q(x)=B(x,x)\text{ for any } x\in V.
\]
A \begin{it}quadratic $\Z$-lattice\end{it} is a free $\Z$-module  
\[
L=\Z e_1 + \Z e_2 + \cdots + \Z e_k
\]
of rank $k$, where $e_1,\dots,e_k\in V$ is a basis of $V$. The \begin{it}dual lattice of $L$\end{it} is defined by
\[
L^{\sharp}:=\{v\in V: B(v,x)\in \Z\ \forall x\in L\}.
\]
The $k\times k$ symmetric matrix $A_L:=(B(e_i,e_j))$ is called the \begin{it}Gram matrix\end{it} of $L$ with respect to a basis $v_1,\ldots, v_k$, and we write $L \cong (B(e_i,e_j))$. The discriminant $d_L$ is defined to be the determinant of $(B(e_i,e_j))$. If $B(e_i,e_j)=0$ for any $i\neq j$ (equivalently, if the Gram matrix is diagonal), then we simply write $L\cong \left< Q(e_1), \ldots , Q(e_k) \right>$.
For a rational prime $p$, we define the {\em localization} of $V$ and $L$ by $V_p:=V\otimes_\Q \Q_p$ and $L_p:=L\otimes_\Z Z_p$, respectively.

By a {\em lattice coset} (or simply {\em coset}), we mean a coset $L+\nu$ in $V/L$ where $\nu$ is a vector in $V$. If $\nu \in L$, then the lattice coset $L+\nu$ is nothing but the lattice $L$. Any elements in $L+\nu$ are of the form $x+\nu$ with $x\in L$. The smallest integer $a$ such that $a\nu\in L$ is called the {\em conductor} of the coset $L+\nu$. As in \eqref{eqn:sum3mgonal}, let 
\[
f(\bm{x}):=\sum_{j=1}^3 p_m(x_j)
\]
 be a sum of three generalized $m$-gonal numbers. In order to study the non-negative integers $n$ represented by $f$, for each such $n$ we write
\[
h=h(n):=8(m-2)n+3(m-4)^2.
\]
For $\bm{1}=(1,1,1)\in\N^3$, let $X^{\bm{1}}:=L^{\bm{1}}+\nu$ where 
\[
L^{\bm{1}}=\Z e_1 +\Z e_2 +\Z e_3 \cong \left<4(m-2)^2,4(m-2)^2,4(m-2)^2\right>,
\]
and
\[
\nu = \frac{4-m}{2(m-2)}(e_1+e_2+e_3)\in \Q L^{\bm{1}}.
\]
Then $n$ is represented by $f$ if and only if $h$ is represented by $X^{\bm{1}}$.
More generally, for $\bm{d}=(d_1,d_2,d_3)\in\N^3$, define the coset $X^{\bm{d}}:=L^{\bm{d}}+\nu$ where $L^{\bm{d}}$ is the sublattice of $L^{\bm{1}}$ defined by
\[
L^{\bm{d}}=\Z (d_1e_1)+\Z (d_2e_2)+ \Z (d_3e_3)\cong \left<4(m-2)^2d_1^2,4(m-2)^2d_2^2,4(m-2)^2d_3^2\right>.
\] 
Note that if $\bm{d'}\mid \bm{d}$, that is, $d_i'\mid d_i$ for any $i=1,2,3$, then
\[
X^{\bm{d}} = L^{\bm{d}} +\nu \subseteq L^{\bm{d'}} +\nu = X^{\bm{d'}}.
\]
We further note that $n$ is represented by $f$ with the restriction $d_j\mid x_j$ if and only if $h$ is represented by $X^{\bm{d}}$.
We are interested in the number $r(h,X^{\bm{d}})$ of solutions $\bm{y}\in\Z^3$ to 
\[
h=8(m-2)n+3(4-m)^2=\sum_{j=1}^3(2(m-2)d_jy_j+4-m)^2.
\]

Let $O(V)$ and $O^+(V)$ be the {\em orthogonal group} and {\em proper orthogonal group} of $V$, respectively.
The {\em class} $\cls(X)$ and the {\em proper class} $\pcls(X)$ of a coset $X$ are defined as the orbits of $X$ under the action of $O(V)$ and $O^+(V)$, respectively.
Let $O_A(V)$ and $O_A^+(V)$ be the {\em ad{\'e}lizations} of $O(V)$ and $O^+(V)$, respectively.
By \cite[Lemma 4.2]{ChanOh13}, $O_A(V)$ acts on $X$, and the orbit of $X$ under the action of $O_A(V)$ is called the {\em genus} of $X$, denoted by $\gen(X)$, while the orbit of $X$ under the action of $O_A^+(V)$ is called {\em proper genus} of $X$, denoted by $\pgen(X)$.
Let $\theta:O^+(V)\rightarrow \Q^\times/(\Q^\times)^2$ be the {\em spinor norm map} (cf. \cite[$\S 55$]{OMBook}) and denote its kernel by $O'(V)$. Let $O_A'(V)$ be the ad{\'e}lization of $O'(V)$.
The {\em proper spinor genus} $\pspn(X)$ of $X$ is defined to be the orbit of $X$ under the action of $O^+(V)O_A'(V)$. Clearly,
$$
\pcls(X)\subseteq \pspn(X) \subseteq \pgen(X).
$$
Set
$$
O^+(L+\nu)=\{\sigma\in O^+(V) : \sigma(L+\nu)=L+\nu\} \quad \text{and} \quad o^+(L+\nu)=|O^+(L+\nu)|.
$$
The groups $O^+(L_p+\nu)$ for any prime $p$ and $O_A^+(L+\nu)$ may analogously be defined.

It was proven in \cite[Proposition 2.5]{Xu} that the number of proper spinor genera in $\pgen(L+\nu)$ is given by
\begin{equation}\label{eqn:numberofproperspinorgenera}
\left[I_\Q : \Q^\times \prod\limits_{p\in \Omega} \spnnorm{L_p+\nu}\right],
\end{equation}
where $I_\Q$ is the id{\`e}le group and $\Omega$ is the set of all places of $\Q$ including the infinite place $\infty$.
Hence, in order to verify the number of proper spinor genera in the genus, we need to compute the local spinor norm group $\spnnorm{L_p+\nu}$ of $L+\nu$ at each prime $p$.

\subsection{Modular forms}

We briefly introduce modular forms of half-integral weight below. We refer readers to \cite{OnoBook} for an introduction to modular forms of integral weight or for more details.

For a positive integer $N$, we require natural congruence subgroups of $\SL_2(\Z)$ defined by 
$$
\begin{aligned}
	&\Gamma_0(N)=\left\{\left(\begin{smallmatrix} a&b\\c&d \end{smallmatrix}\right) \in \SL_2(\Z) : c \equiv 0 \pmod{N} \right\},\\
	&\Gamma_1(N)=\left\{\left(\begin{smallmatrix} a&b\\c&d \end{smallmatrix}\right) \in \Gamma_0(N) : a,d \equiv 1 \pmod{N} \right\}.
\end{aligned}
$$
For $\gamma=\left(\begin{smallmatrix} a&b\\c&d \end{smallmatrix}\right)\in \Gamma_0(4)$ and $\kappa\in \frac{1}{2}+\Z$, define the {\em slash operator} on a function $f:\H \rightarrow \C$ by
$$
f\mid_{\kappa}\gamma (z)= \legendre{c}{d}\varepsilon_d^{2\kappa}(cz+d)^{-\kappa}f(\gamma z),
$$
where $\varepsilon_d=1$ if $d\equiv 1 \pmod{4}$, $\varepsilon_d=i$ if $ d\equiv 3 \pmod{4}$, $\legendre{\cdot}{\cdot}$ is the Kronecker--Jacobi--Legendre symbol, and $\gamma$ acts on $\H$ by fractional linear transformations $\gamma z := \frac{az+b}{cz+d}$. We call $f$ a {\em (holomorphic) modular form} of weight $\kappa$ on $\Gamma\subseteq \Gamma_0(4)$ ($\Gamma$ a congruence subgroup containing $\left(\begin{smallmatrix} 1&1\\0&1 \end{smallmatrix}\right)$) with character $\chi$ if
\begin{enumerate}[label={\rm (\arabic*)}]
	\item $f|_\kappa\gamma = \chi(d)f$ for any  $\gamma=\left(\begin{smallmatrix} a&b\\c&d \end{smallmatrix}\right)\in \Gamma$,
	\item $f$ is holomorphic on $\H$,
	\item $f(z)$ grows at most polynomially in $y$ as $z=x+iy\rightarrow \Q\cup \{i\infty\}$.
\end{enumerate}
We moreover call $f$ a {\em cusp form} if $f(z)\rightarrow 0$ as $z\rightarrow \Q\cup \{i\infty\}$.
The space of modular forms (resp. cusp forms) of weight $\kappa$, character $\chi$ and congruence subgroup $\Gamma$, will be denoted by $M_{\kappa}(\Gamma,\chi)$ (resp. $S_{\kappa}(\Gamma,\chi)$).
The space of {\em Eisenstein series}, denoted by $E_{\kappa}(\Gamma,\chi)$, is the orthogonal complement of $S_{\kappa}(\Gamma,\chi)$ in $M_{\kappa}(\Gamma,\chi)$ with respect to the Petersson inner product (for an introduction and properties of the inner product, see \cite[Chapter III]{Lang}). If $f$ is a modular form for a congruence group $\Gamma$ containing $\left(\begin{smallmatrix} 1&1\\0&1\end{smallmatrix}\right)$, then $f$ has a Fourier series expansion 
$$
f(z)=\sum\limits_{n=0}^\infty a(n)q^n,
$$
where $q=e^{2\pi iz}$. In particular, if $f$ is a cusp form, then $a(0)=0$. For $\kappa=\frac{3}{2}$, we further require the subspace 
\begin{equation}\label{eqn:unarydef}
U_t(\Gamma,\chi):=S_{3/2}(\Gamma,\chi) \cap \left\{f(z)=\sum\limits_{n=1}^\infty a(n)nq^{tn^2}\right\}
\end{equation}
spanned by unary theta functions.  For a Dirichlet character $\psi$ modulo $m_\psi$, consider
\[
 h(z,\psi)=\sum_{n=1}^\infty \psi(n)nq^{n^2}.
\]
Note that for $L\mid N$ the space $U_t(\Gamma_0(N)\cap\Gamma_1(L),\chi)$ is spanned by 
\begin{equation}\label{UtNchi-is-spanned-by}
	\left\{h(tu^2z,\psi) : u\in\Z, \ 4tm_\psi^2u^2\mid N, \ \psi = \chi \legendre{-4t}{\cdot}\eta,\ \eta\pmod{L}\right\}.
\end{equation}
Here $\eta$ runs through all characters modulo $L$ and $m_{\psi}$ is the conductor of $\psi$. This follows from the fact that the spaces $h(tu^2z,\psi)$ for different $t$ or $\psi$ are orthogonal to each other with respect to Petersson inner product and the modularity given in \cite[Proposition 2.2]{Shimura}.

\subsection{Elementary theta functions} Let $k$ be a positive integer, $A$ a positive definite $k\times k$ symmetric matrix, $h$ an element in $\Z^k$, and $N$ a positive integer satisfying the following conditions:
\begin{equation}\label{condition-thetafunctionofShimura}
	\text{Both $A$ and $NA^{-1}$ have coefficients in $\Z$; } Ah\in N\Z^k.
\end{equation}
In \cite{Shimura}, Shimura defined the theta function 
\begin{equation}\label{defn-thetafunctionofShimura}
	\vartheta(z;h,A,N,P)=\sum\limits_{x\in\Z^k,\, x\equiv h\pmod{N}} P(x) \cdot q^{ x^tAx/2N^2},
\end{equation}
where $P(x)$ is a spherical function of order $\nu\in\Z_{\ge 0}$ with respect to $A$. In this article, we only concern the case when $P(x)=1$ where $\nu=0$, or $P(x)=x$ with $k=1$ where $\nu=1$. 
Indeed, the function $h(z,\psi)$ defined in \eqref{UtNchi-is-spanned-by} is given by a linear combination of the theta functions corresponding to the latter case (see \cite[Proposition 2.2]{Shimura}).
Moreover, Shimura \cite{Shimura} proved the following transformation formula of the theta functions.
\begin{proposition}[Proposition 2.1 of \cite{Shimura}]\label{prop-Shimura-transformation-theta}
	Let $\vartheta(z;h,A,N,P)$ be defined by \eqref{defn-thetafunctionofShimura} under the assumption \eqref{condition-thetafunctionofShimura}, and let $\gamma=\left(\begin{smallmatrix}
		a&b\\c&d
	\end{smallmatrix}\right)\in \SL_2(\Z)$ with $b\equiv 0\pmod{2}$ and $c\equiv 0 \pmod{2N}$. Then
\begin{multline*}
\vartheta(\gamma(z);h,A,N,P)=e(ab\cdot h^tAh/2N^2)\legendre{\det(A)}{d}\legendre{2c}{d}^k\varepsilon_d^{-k}\\
\times(cz+d)^{(k+2\nu)/2}\vartheta(z;ah,A,N,P),
\end{multline*}
where $e(z):=e^{2\pi i z}$.
\end{proposition}

\subsection{Theta series for cosets}\label{subsection-prelim-thetaseriesofcosets} 
Let $X=L+\nu$ be a coset on a quadratic space $V$ of rank $k$ and we assume throughout that $B(X,X)\subseteq \Z$. For a positive integer $n$, we define
$$
R_X(n):=\{x\in X : Q(x)=n\} \quad \text{and} \quad r_X(n):= |R_X(n)|,
$$
and the theta series $\T{X}(z)$ of the coset $X$ is defined as
$$
\T{X}(z):=\sum\limits_{x\in X} q^{Q(x)} = \sum\limits_{n=0}^\infty r_X(n) q^n.
$$
We define the theta series $\T{\pgen(X)}(z)$ of $\pgen(X)$ and $r_{\pgen(X)}(n)$ by
\begin{equation}\label{defn-thetaofpropergenus}
	\T{\pgen(X)}(z)=\sum\limits_{n=0}^\infty r_{\pgen(X)}(n)q^n:=\left(\sum\limits_{Y\in\pgen(X)}\frac{1}{o^+(Y)}\right)^{-1}\cdot \left(\sum\limits_{Y\in\pgen(X)}\frac{\T{Y}(z)}{o^+(Y)}\right)
\end{equation}
and the theta series $\T{\pspn(X)}(z)$ of $\pspn(X)$ and $r_{\pspn(X)}(n)$ by
\begin{equation}\label{defn-thetaofproperspinorgenus}
	\T{\pspn(X)}(z)=\sum\limits_{n=0}^\infty r_{\pspn(X)}(n) q^n:= \left(\sum\limits_{Y\in\pspn(X)}\frac{1}{o^+(Y)}\right)^{-1}\cdot \left(\sum\limits_{Y\in\pspn(X)}\frac{\T{Y}(z)}{o^+(Y)}\right).
\end{equation}
The summation runs over a system of representatives of proper classes in the proper genus or in the proper spinor genus of $X$.

For any non-zero integer $d$, let $\chi_d$ denote the character $\chi_d(\cdot)=\legendre{d}{\cdot}$ obtained from the Kronecker symbol. Using Proposition \ref{prop-Shimura-transformation-theta}, a straightforward calculation (see \cite[Proposition 2.3]{KaneKim}) shows that the theta series of cosets of rank $k$ are modular forms of weight $k/2$.
\begin{proposition}\label{prop-thetaofcoset-is-a-modularform}
	Let $X=L+\nu$ be a coset on a quadratic space $V$ of odd rank $k$ and conductor $a$. Let $N_L$ be the level of $L$ and $d_L$ the discriminant of $L$. Then 
	$$
	\T{X}(z)\in  M_{k/2}(\Gamma_0(4N_L)\cap \Gamma_1(a), \chi_{4d_L}). 
	$$
\end{proposition}

\section{The Eisenstein series components}\label{sec:Eisenstein}
We recall the work of Shimura \cite{ShimuraCongruence}.
Let $X=L+\nu$ be a lattice coset of rank $k$, and let $h\in\Q$.
For any prime $p$ and $z\in\Q_p$, we define $\bm{e}_p(z)=e^{2\pi i (-y)}$ with $y\in \cup_{t=1}^\infty p^{-t}\Z$ such that $z-y\in\Z_p$. 
Let $\lambda(v)$ and $\lambda_p$ be the characteristic functions of $X$ and $X_p$ respectively.
Normalizing the measures $dv$ and $d\sigma$ on $L_p$ and $\Z_p$ so that $\int_{L_p}dv=\int_{\Z_p} d\sigma =1$, the local density at $p$ is defined by 
\begin{equation}\label{eqn:defn-localdensity}
	\begin{aligned}
		b_p(h,\lambda,0)&:=\int_{\Q_p} \int_{V_p} \bm{e}_p(\sigma (Q(v)-h))\lambda_p(v) dv d\sigma\\
		&=\int_{\Q_p} \int_{L_p} \bm{e}_p(\sigma (Q(v+\nu)-h)) dv d\sigma.
	\end{aligned}
\end{equation}
Shimura \cite[Theorem 1.5]{ShimuraCongruence} showed that the $h$-th Fourier coefficient of the Eisenstein series part of $\T{X}$ can be expressed as the product of local densities. We provide the statement for the special case when $k=3$.
\begin{theorem}\label{thm:ShimuraThm1.5}
	Let $X=L+\nu$ be a ternary lattice coset, and let $h\in\Q$. Then
	\[
	r_{\pgen(X)}(h)=\frac{(2\pi)^{3/2}h^{1/2}L(1,\psi_h)}{[\frac{1}{2}L^\sharp : L]^{1/2}\Gamma(3/2)\zeta(2)} \cdot \prod_{p\mid e_1} \frac{b_p(h,\lambda,0)(1-\psi_h(p)p^{-1})}{1-p^{-2}} \cdot \prod_{p\mid he', p\nmid e_1} \gamma_p(3/2).
	\]
	Here $\psi_h(\cdot)$ is the primitive character associated to the real character $\legendre{-h}{\cdot}$, $e_1$ is the product of all primes $p$ at which $h\notin \Z_p$ or $L_p$ is not maximal or $X_p\neq L_p$, $e'$ is the product of all primes such that $L_p^\sharp \neq 2L_p$, and the numbers $\gamma_p(s)$ are given in \cite[Section 1.6]{ShimuraCongruence}.
\end{theorem}
Let $\lambda_{\bm{d}}$ be the characteristic function of $X^{\bm{d}}$.
In our case, $e_1$ is the product of all primes dividing $2(m-2)d_1d_2d_3$, and $e_1=e'$ since $2\mid e_1$. Therefore, we have
\begin{multline}\label{eqn:Siegelformula}
	r_{\pgen(X^{\bm{d}})}(h)=\frac{(2\pi)^{3/2}h^{1/2}L(1,\psi_h)}{(8d_{L^{\bm{d}}})^{1/2}(\sqrt{\pi}/2)(\pi^2/6)} \cdot \prod_{p\mid e_1} \frac{b_p(h,\lambda_{\bm{d}},0)(1-\psi_h(p)p^{-1})}{1-p^{-2}}\cdot \prod_{p\mid h, p\nmid e_1} \gamma_p(3/2)\\
	=\frac{12}{\pi} \left(\frac{h}{d_{L^{\bm{d}}}}\right)^{1/2} L(1,\psi_h)\cdot \prod_{p\mid 2(m-2)d_1d_2d_3} \frac{b_p(h,\lambda_{\bm{d}},0)(1-\psi_h(p)p^{-1})}{1-p^{-2}}\cdot \prod_{\substack{p\mid h \\ p\nmid 2(m-2)d_1d_2d_3}} \gamma_p(3/2).
\end{multline}
Furthermore, $t_p$ defined in \cite[Section 1.6]{ShimuraCongruence} is equal to $1$ in our case, hence for any odd prime number $p$, we have
\begin{equation}\label{eqn:gamma3over2}
	\gamma_p(3/2) = \frac{1-p^{-\left(\lfloor\ord_p(h)/2\rfloor+1\right)}-\psi_h(p)p^{-1}(1-p^{-\lfloor\ord_p(h)/2\rfloor})}{1-p^{-1}}.
\end{equation}

\subsection{Computation of local densities of $X^{\bm{d}}$}
In this subsection, we compute the local densities $b_p(h,\lambda_{\bm{d}},0)$ of $X^{\bm{d}}$ in order to evaluate $r_{\pgen(X^{\bm{d}})}(h)$ using \eqref{eqn:Siegelformula}. For $\sigma\in\Q_p$, we first require a formula for 
\[
\tau_{p,s}(\sigma):=\int_{\Z_p} e_p\left( 4(m-2)s\sigma\cdot ((m-2)s x^2-(m-4)x)\right) dx,
\]
which follows from a similar argument as in the proof of \cite[Lemmas 3.3, 3.7, 3.9]{KaneLiu} and relies on the fact that $x\mapsto (m-2)sx^2-(m-4)x$ is either a bijection or a two-to-one map from $\Z/p^t\Z$ to its image and orthogonality of root of unity.
\begin{lemma}\label{lem:taupdjeval}
If $s\in\Z_p$ satisfies $p\mid 2(m-2)s$, then we have 
	\[
	\tau_{p,s}(\sigma) = 
	\begin{cases}
		1& \text{if } 4(m-2)s \sigma \in \Z_p,\\
		0& \text{otherwise}.
	\end{cases}
	\]
\end{lemma}
\begin{extradetails}
First assume that $p=2$. ($\because$ First note that if $4(m-2)s\sigma\in\Z_2$, then $\tau_{2,s}(\sigma) = \int_{\Z_2} dx = 1$. 
		
		Now we write $4(m-2)s\sigma = 2^{-t}\alpha$ with $t\in\N$, $\alpha\in\Z_2^\times$.
		Note that, since $p\mid (m-2)$, the map sending $x\mapsto (m-2)sx^2-(m-4)x$ is a bijection from $\Z/2^t\Z$ onto itself when $s$ is even, and is a $2$-to-$1$ map from $\Z/2^t\Z$ onto $2\Z/2^t\Z$ if $s$ is odd (the proof is the same as in Lemma 3.7 and Lemma 3.9 in \cite{KaneLiu}, respectively).
		Then
		\[
		\begin{aligned}
			\tau_{2,s}(\sigma)&=\int_{\Z_2} e_2\left( 2^{-t}\alpha ((m-2)sx^2-(m-4)x)\right) dx\\
			&=	\sum_{x\in \Z_2/2^t\Z_2} \int_{2^t\Z_2} e_2\left( 2^{-t}\alpha ((m-2)s(x+y)^2-(m-4)(x+y))\right) dy\\
			&=	\sum_{x\in \Z/2^t\Z} e_2\left( 2^{-t}\alpha ((m-2)sx^2-(m-4)x)\right) \int_{2^t\Z_2} dy\\
			&=\begin{cases}
				2^{-t} \sum_{z\in \Z/2^t\Z} e_2(2^{-t}\alpha z)=0 & \text{if } 2\mid s,\\
				2\cdot 2^{-t} \sum_{z\in \Z/2^{t-1}\Z} e_2(2^{1-t}\alpha z)=0 & \text{if } 2\nmid s.\\
			\end{cases}
		\end{aligned}
		\]
		This completes the computation.)

Now assume that $p\neq 2$. 	($\because$ First note that if $4(m-2)s\sigma\in\Z_p$, then $\tau_{p,s}(\sigma) = \int_{\Z_p} dx = 1$. 

Now we write $4(m-2)s\sigma = p^{-t}\alpha$ with $t\in\N$, $\alpha\in\Z_p^\times$.
Note that (since $p\mid (m-2)s$) the map sending $x\mapsto (m-2)sx^2-(m-4)x$ is a bijection from $\Z/p^t\Z$ onto itself.
Then
\[
\begin{aligned}
	\tau_{p,s}(\sigma)&=\int_{\Z_p} e_p\left( p^{-t}\alpha ((m-2)sx^2-(m-4)x)\right) dx\\
	&=	\sum_{x\in \Z_p/p^t\Z_2} \int_{p^t\Z_p} e_p\left( p^{-t}\alpha ((m-2)s(x+y)^2-(m-4)(x+y))\right) dy\\
	&=	\sum_{x\in \Z/p^t\Z} e_p\left( p^{-t}\alpha ((m-2)sx^2-(m-4)x)\right) \int_{p^t\Z_2} dy\\
	&=p^{-t} \sum_{z\in \Z/p^t\Z} e_p(p^{-t}\alpha z)=0.
\end{aligned}
\]
This completes the computation.)		
\end{extradetails}
We next evaluate the local density for the prime $p=2$.
\begin{lemma}\label{lemma:localdensity:at2}
	The local density of $X^{\bm{d}}$ at the prime $2$ is given by 
	\[
	b_2(h,\lambda_{\bm{d}},0)=
	\begin{cases}
		2^{2+\min_j \{\ord_2(d_j)\}} & \text{if } 8(m-2)n\in 4\gcd(d_1,d_2,d_3)\Z_2,\\
		0 & \text{otherwise.}
	\end{cases}
	\]
\end{lemma}
\begin{proof} With $v=x_1(d_1e_1)+x_2(d_2e_2)+x_3(d_3e_3)\in L^{\bm{d}}$, note that
	\begin{equation}\label{eqn:Qshift}
	Q(v+\nu)=\sum_{j=1}^3 4(m-2)^2 \left(d_jx_j-\frac{m-4}{2(m-2)}\right)^2.
	\end{equation}
Using Lemma \ref{lem:taupdjeval}, \eqref{eqn:defn-localdensity} yields (using orthogonality of roots of unity in the final step)
	\begin{align*}
		b_2(h,\lambda_{\bm{d}},0)&=\int_{\Q_p} e_p\left(\sigma (3(m-4)^2-h)\right) \prod_{j=1}^3 \tau_{2,d_j}(\sigma) d\sigma=\int_{\frac{1}{4\gcd(d_1,d_2,d_3)}\Z_p} e_p\left( -8(m-2)n\sigma\right) d\sigma\\
&=\begin{cases}
2^{2+\min_j\{\ord_2(d_j)\}}&\text{if }8(m-2)n\in4\gcd(d_1,d_2,d_3),\\
0&\text{otherwise}.
\end{cases}	
\end{align*}
\begin{extradetails}
	If $8(m-2)n\in4\gcd(d_1,d_2,d_3)$, we have 
	\[
	b_2(h,\lambda_{\bm{d}},0)=\int_{\frac{1}{4\gcd(d_1,d_2,d_3)}\Z_p} d\sigma=2^{2+\min_j\{\ord_2(d_j)\}}.
	\]
	Otherwise, let $\frac{-8(m-2)n}{4\gcd(d_1,d_2,d_3)}=2^{-u}\beta$ with $u\in\N$ and $\beta\in\Z_2^\times$
	\[
	\begin{aligned}
		b_2(h,\lambda_{\bm{d}},0)&
		=\sum_{\sigma\in \frac{1}{4\gcd(d_1,d_2,d_3)}\Z_2/2^{u}\Z_2} \int_{2^{u}\Z_2} e_2\left( 4\gcd(d_1,d_2,d_3)2^{-u}\beta (\sigma+\sigma')\right) d\sigma'\\
		&=\sum_{\sigma\in \frac{1}{4\gcd(d_1,d_2,d_3)}\Z_2/2^{u}\Z_2} e_2\left( 4\gcd(d_1,d_2,d_3)2^{-u}\beta \sigma\right) \int_{2^{u}\Z_2}d\sigma'\\
		&=\sum_{\sigma\in \Z_2/4\gcd(d_1,d_2,d_3)2^{u}\Z_2} e_2\left( 2^{-u}\beta \sigma\right)\cdot 2^{-u}=0.
	\end{aligned}
	\]
\end{extradetails}
	This completes the proof of the lemma.
\end{proof}

For odd primes $p\mid (m-2)$, we have the following evaluation.
\begin{lemma}\label{lemma:localdensity:podddividingm-2} 
	Let $p$ be an odd prime such that $p\mid (m-2)$. Then we have
	\[
	b_p(h,\lambda_{\bm{d}},0)=\begin{cases}
		p^{\ord_p(m-2)+\min_j\{\ord_p(d_j)\}} & \text{if } n\in \gcd(d_1,d_2,d_3)\Z_p,\\
		0 &\text{otherwise}.
	\end{cases}
	\]
\end{lemma}
\begin{proof} 
	By Lemma \ref{lem:taupdjeval}, \eqref{eqn:defn-localdensity}, and \eqref{eqn:Qshift}, we have 
\[
		b_p(h,\lambda_{\bm{d}},0)=\int_{\Q_p}\hspace{-.05cm} e_p\left(\sigma (3(m-4)^2-h)\right) \prod_{j=1}^3 \tau_{p,d_j}(\sigma) d\sigma=\int_{\frac{1}{(m-2)\gcd(d_1,d_2,d_3)}\Z_p}\hspace{-0.13cm} e_p\left(-8(m-2)n\sigma\right)\hspace{-.05cm} d\sigma.
\]
The lemma now follows as before.\qedhere
\begin{extradetails}
	If $n\in \gcd(d_1,d_2,d_3)\Z_p$, then 
	\[
	b_p(h,\lambda_{\bm{d}},0)= \int_{\frac{1}{(m-2)\gcd(d_1,d_2,d_3)}\Z_p} d\sigma =p^{\ord_p(m-2)+\min_j\{\ord_p(d_j)\}}.
	\]
	Otherwise, let $\frac{-8(m-2)n}{(m-2)\gcd(d_1,d_2,d_3)}=p^{-u}\beta$ with $u\in\N$ and $\beta\in\Z_p^\times$
	\[
	\begin{aligned}
		b_p(h,\lambda_{\bm{d}},0)&
		=\sum_{\sigma\in \frac{1}{(m-2)\gcd(d_1,d_2,d_3)}\Z_p/p^{u}\Z_p} \int_{p^{u}\Z_2} e_p\left( (m-2)\gcd(d_1,d_2,d_3)p^{-u}\beta (\sigma+\sigma')\right) d\sigma'\\
		&=\sum_{\sigma\in \frac{1}{(m-2)\gcd(d_1,d_2,d_3)}\Z_p/p^{u}\Z_p} e_p\left( (m-2)\gcd(d_1,d_2,d_3)p^{-u}\beta \sigma\right) \int_{p^{u}\Z_p}d\sigma'\\
		&=\sum_{\sigma\in \Z_p/(m-2)\gcd(d_1,d_2,d_3)p^{u}\Z_p} e_p\left( p^{-u}\beta \sigma\right)\cdot p^{-u}=0.
	\end{aligned}
	\]
\end{extradetails}
\end{proof}

We  next consider odd primes $p\nmid (m-2)(m-4)$ which divide at least one $d_j$.
\begin{lemma} \label{lemma:localdensity:poddnotdividingm-2m-4}
Let $p$ be an odd prime such that $p\nmid (m-2)$, $p\mid d_1d_2d_3$, and $p\nmid (m-4)$. Write 
\[
N=8(m-2)n+\#\{j:p\nmid d_j\}(m-4)^2. 
\]
	When $\#\{j:p\nmid d_j\}=2$, we have \[b_p(h,\lambda_{\bm{d}},0)=1+\legendre{-1}{p}\frac{1}{p}\cdot \begin{cases}
		p-1 & \text{if } N\in p\Z_p,\\
		-1 &\text{if } N\notin p\Z_p.
	\end{cases}
	\]
	When $\#\{j:p\nmid d_j\}=1$, we have 
	\[
	b_p(h,\lambda_{\bm{d}},0)=1+\legendre{N}{p}.
	\]
	When $\#\{j:p\nmid d_j\}=0$, we have \[b_p(h,\lambda_{\bm{d}},0)= \begin{cases}
		p & \text{if } N\in p\Z_p,\\
		0 &\text{if } N\notin p\Z_p.
	\end{cases}
	\]
\end{lemma}
\begin{proof} 
	By \eqref{eqn:defn-localdensity}, we have 
	\begin{equation}\label{eqn:localdensityatp:2}
		b_p(h,\lambda_{\bm{d}},0)=\int_{\Q_p} e_p\left(\sigma (3(m-4)^2-h)\right) \prod_{j=1}^3 \tau_{p,d_j}(\sigma) d\sigma.	
	\end{equation}
If $p\mid d_j$, then Lemma \ref{lem:taupdjeval} gives (note that $p\nmid 4(m-2)$)
	\begin{equation}\label{eqn:taupmiddjsimp}
	\tau_{p,d_j}(\sigma) = 
	\begin{cases}
		1& \text{if } d_j \sigma \in \Z_p,\\
		0& \text{otherwise}.
	\end{cases}
	\end{equation}
	Now assume that $p\nmid d_j$. Since the map sending $x\mapsto (m-2)d_jx^2-(m-4)x$ is no longer a bijection or $2$-to-$1$ mapping, we cannot use the same argument as in Lemma \ref{lem:taupdjeval}. If $\sigma\in\Z_p$ then $\tau_{p,d_j}(\sigma)=1$. Otherwise, we write $\sigma = p^{-t}\alpha$ with $t\in\N$, $\alpha\in\Z_p^\times$ and use Gauss's evaluation of the quadratic Gauss sum to obtain
	\[
	\begin{aligned}
		\tau_{p,d_j}(\sigma)&=\int_{\Z_p} e_p\left( \sigma  (2(m-2)d_jx-(m-4))^2-\sigma(m-4)^2\right) dx\\
		&=e_p\left(-\sigma(m-4)^2\right) \cdot  \int_{\Z_p} e_p\left( \sigma x^2\right) dx=e_p\left(-\sigma(m-4)^2\right)	\sum_{x\in \Z/p^t\Z} e_p\left( p^{-t}\alpha x^2\right) \int_{p^t\Z_p} dy\\
		&=e_p\left(-\sigma(m-4)^2\right) p^{-t} \cdot \sum_{0\le x<p^t} e^{2\pi i \cdot  \frac{ax^2}{p^t}}=e_p\left(-\sigma(m-4)^2\right) p^{-t} \cdot \varepsilon_{p^t} \sqrt{p^t} \legendre{-p^{t}\sigma}{p^t},
	\end{aligned}
	\]
	where $a\in\N$ with $a\equiv -\alpha\pmod{p^t}$, and $\varepsilon_c=1$ if $c\equiv 1\pmod{4}$ and $i$ if $c\equiv 3\pmod{4}$. Since $p\mid d_1d_2d_3$ and $d_j$ are squarefree, \eqref{eqn:taupmiddjsimp} implies that the product $\prod_{j=1}^3 \tau_{p,d_j}(\sigma)=0$ unless $\sigma\in\frac{1}{p}\Z_p$.
	Plugging these in \eqref{eqn:localdensityatp:2} we have
\begin{align*}
		b_p&(h,\lambda_{\bm{d}},0)= \int_{\frac{1}{p}\Z_p} e_p\left(\sigma (3(m-4)^2-h)\right) \prod_{j=1}^3 \tau_{p,d_j}(\sigma) d\sigma\\
		&=\int_{\Z_p} d\sigma + \int_{\frac{1}{p}\Z_p\setminus\Z_p} e_p\left(\sigma (3(m-4)^2-h)\right) \left[e_p(-\sigma(m-4)^2)\varepsilon_p p^{-\frac{1}{2}}\legendre{-p\sigma}{p}\right]^{\#\{j:p\nmid d_j\}} d\sigma\\
		&=1+ \int_{\frac{1}{p}\Z_p\setminus\Z_p} e_p\left(-\sigma N
\right)\left[\varepsilon_p p^{-\frac{1}{2}}\legendre{-p\sigma}{p}\right]^{\#\{j:p\nmid d_j\}}   d\sigma\\
		&=1+\sum_{\sigma\in\Z/p\Z-\{0\}}\int_{\Z_p} e_p(-(p^{-1}\sigma+\sigma')\cdot N)\left[\varepsilon_p p^{-\frac{1}{2}}\legendre{-(\sigma+p\sigma')}{p}\right]^{\#\{j:p\nmid d_j\}} d\sigma'\\
&=1+	\sum_{\sigma\in\Z/p\Z-\{0\}}e_p(-p^{-1}\sigma\cdot N)   \left[\varepsilon_p p^{-\frac{1}{2}}\legendre{-\sigma}{p}\right]^{\#\{j:p\nmid d_j\}} \int_{\Z_p}d\sigma'\\
				&=1+\sum_{\sigma\in\Z/p\Z-\{0\}}e_p\left(-\frac{1}{p}\sigma\cdot N\right)\left[\varepsilon_p p^{-\frac{1}{2}}\legendre{-\sigma}{p}\right]^{\#\{j:p\nmid d_j\}}.
	\end{align*}
	If $\#\{j:p\nmid d_j\}=0$ or $2$, then \[b_p(h,\lambda_{\bm{d}},0)=1+\left[\legendre{-1}{p}\frac{1}{p}\right]^{\frac{1}{2}\#\{j:p\nmid d_j\}}\cdot \begin{cases}
		p-1 & \text{if } N\in p\Z_p,\\
		-1 &\text{if } N\notin p\Z_p.
	\end{cases}
	\]
	If $k:=\#\{j:p\nmid d_j\}=1$, then
	\[
	\begin{aligned}
		b_p(h,\lambda_{\bm{d}},0)&=1+\sum_{\sigma\in\Z/p\Z-\{0\}}e_p\left(-\frac{1}{p}\sigma\cdot N\right)\varepsilon_p p^{-\frac{1}{2}}\legendre{-\sigma}{p}\\
		&=1+\varepsilon_p p^{-\frac{1}{2}}\legendre{-1}{p}\cdot \tau\left(\legendre{\cdot}{p},\psi_{N,p}\right),
	\end{aligned}
	\]
	where $\tau(\chi,\psi):=\sum_{x\pmod{c}} \chi(x)\psi(x)$ is the Gauss sum for a multiplicative character $\chi$ and an additive character $\psi$, both of modulo $c$, and $\psi_{m,p^k}(x):=e^{\frac{2\pi i mx}{p^k}}$.
	Note that if $N\in p\Z_p$, then
	\[
	\tau\left(\legendre{\cdot}{p},\psi_{N,p}\right)
	=\sum_{\sigma\in\Z/p\Z} \legendre{\sigma}{p}=0.
	\]
	Otherwise, $N\in\Z_p^\times$ and it is well-known (cf. \cite[(4.20)]{BanerjeeKane}) that
	\[
	\tau\left(\legendre{\cdot}{p},\psi_{N,p}\right)
	=\varepsilon_p p^{\frac{1}{2}} \legendre{N}{p}.\qedhere
	\]
\end{proof}
For primes dividing $m-4$, we may  use the following lemma.

\begin{lemma} \label{lemma:localdensity:poddnotdividingm-2butm-4}
 Assume that $d_1,d_2,d_3\in \N$ are square-free. Let $p$ be an odd prime such that $p\nmid (m-2)$, $p\mid d_1d_2d_3$ and $p\mid (m-4)$. Write $h=up^a$ with $u\in\Z_p^\times$, $a\in\N_0$.	Then the local coset $X^{\bm{d}}_p$ is nothing but the local lattice $L^{\bm{d}}_p$.\\
\noindent	When $\#\{j:p\mid d_j\}=1$, we have
	\[
	b_p(h,\lambda_{\bm{d}},0)= 
	\begin{cases}
		2+\left(1-\frac{1}{p}\right)\legendre{-1}{p} - \frac{1}{p^{\lfloor a/2\rfloor}} + (-1)^{a}\legendre{-u}{p}^{a+1}\frac{1}{p^{\lceil a/2\rceil}}& \text{if } a\ge 2,\\
		\left(1-\frac{1}{p}\right) \left(1+\legendre{-1}{p}\right) &\text{if } a=1,\\
		1-\legendre{-1}{p}\frac{1}{p} & \text{if }a=0.
	\end{cases}
	\]	
	When $\#\{j:p\mid d_j\}=2$, we have
	\[
	b_p(h,\lambda_{\bm{d}},0)= 
	\begin{cases}
		1+p- \frac{1}{p^{\lfloor a/2\rfloor-1}} + (-1)^{a}\legendre{-u}{p}^{a+1}\frac{1}{p^{\lceil a/2\rceil-1}}& \text{if } a\ge 2,\\
		0 &\text{if } a=1,\\
		1+\legendre{u}{p} & \text{if }a=0.
	\end{cases}
	\]	
	When $\#\{j:p\mid d_j\}=3$, we have
	\[
	b_p(h,\lambda_{\bm{d}},0)= 
	\begin{cases}
		p^2+p-\frac{1}{p^{\lfloor a/2\rfloor-2}} + (-1)^{a}\legendre{-u}{p}^{a+1}\frac{1}{p^{\lceil a/2\rceil-2}}& \text{if } a\ge 2,\\
		0 &\text{if } a\le 1.
	\end{cases}
	\]	
\end{lemma}
\begin{proof}
This follows from the local densities for lattices given in \cite[Theorem 3.1]{Yang}.
\end{proof}

\subsection{Lower bound for the $\bm{d}=\bm{1}$ case}
We are now ready to give a lower bound for $r_{\pgen(X^{\bm{1}})}(h)$.

\begin{lemma}\label{lem:X1lower}
We have 
\[
r_{\pgen(X^{\bm{1}})}(h)\geq \frac{12}{\pi (m-2)^2}\prod_{p\mid(m-2)}\frac{1}{1+p^{-1}} h^{\frac{1}{2}}L(1,\psi_h)\gg_{m,\varepsilon} h^{\frac{1}{2}-\varepsilon}.
\]

\end{lemma}
\begin{proof}
We start by plugging $\bm{d}=\bm{1}$ into \eqref{eqn:Siegelformula}, yielding 
\begin{equation}\label{eqn:SiegelformulaX1}
	r_{\pgen(X^{\bm{1}})}(h)=\frac{12}{\pi} \left(\frac{h}{d_{L^{\bm{1}}}}\right)^{1/2} L(1,\psi_h)\cdot \prod_{p\mid 2(m-2)} \frac{b_p(h,\lambda_{\bm{1}},0)(1-\psi_h(p)p^{-1})}{1-p^{-2}}\cdot \prod_{\substack{p\mid h \\ p\nmid 2(m-2)}} \gamma_p(3/2).
\end{equation}
Using  Lemma \ref{lemma:localdensity:at2} and Lemma \ref{lemma:localdensity:podddividingm-2}, we conclude that 
\[
\prod_{p\mid 2(m-2)} \frac{b_p(h,\lambda_{\bm{1}},0)(1-\psi_h(p)p^{-1})}{1-p^{-2}}=8(m-2)\prod_{p\mid (m-2)} \frac{1-\frac{1}{p}\psi_{h}(p)}{1-p^{-2}}\geq \frac{8(m-2)}{\prod_{p\mid (m-2)}\left(1+p^{-1}\right)}.
\]
\begin{extradetails}
By Lemma \ref{lemma:localdensity:at2}, we have 
\[
b_2(h,\lambda_{\bm{1}},0)=4
\]
and, since $h=8(m-2)n+3(m-4)^2\equiv 3\pmod{8}$,
\[
\frac{1-\frac{\psi_h(2)}{2}}{1-\frac{1}{4}} = \frac{1-\frac{1}{2}\legendre{-3}{2}}{\frac{3}{4}} = 2.
\]
For $p\mid (m-2)$, we use Lemma \ref{lemma:localdensity:podddividingm-2} to compute 
\[
\prod_{p\mid (m-2)} b_p(h,\lambda_{\bm{1}},0) = m-2.
\]
Finally, we have 
\[
\prod_{p\mid (m-2)} \frac{1-\frac{1}{p}\psi_{h}(p)}{1-p^{-2}}\geq \prod_{p\mid (m-2)} \frac{1-\frac{1}{p}}{1-p^{-2}}=\prod_{p\mid (m-2)} \frac{1}{1+p^{-1}}.
\]
\end{extradetails}
For $p\mid h$ with $p\nmid 2(m-2)$, one can use \eqref{eqn:gamma3over2} to bound 
\[
\gamma_p(3/2) \geq 1.
\]
\begin{extradetails}
\begin{align*}
\gamma_p(3/2) 
&= \frac{1-p^{-\left(\lfloor\ord_p(h)/2\rfloor+1\right)}-\psi_h(p)p^{-1}(1-p^{-\lfloor\ord_p(h)/2\rfloor})}{1-p^{-1}}\\
&=\frac{1-\psi_h(p)p^{-1} -p^{-\left(\lfloor\ord_p(h)/2\rfloor+1\right)}\left(1-\psi_{h}(p)\right)}{1-p^{-1}}\\
&=\begin{cases} 1&\text{if }\psi_h(p)=1,\\
1 +\frac{2p^{-1} - 2p^{-\left(\lfloor\ord_p(h)/2\rfloor+1\right)}}{1-p^{-1}}\geq 1&\text{if }\psi_h(p)=-1,\\
1 +\frac{p - p^{-\left(\lfloor\ord_p(h)/2\rfloor+1\right)}}{1-p^{-1}}\geq 1&\text{if }\psi_h(p)=0.
\end{cases}
\end{align*}
\end{extradetails}
Plugging in $d_{L^1}=64(m-2)^6$ yields the first bound. Writing for $h=|\Delta| t^2$, where $\Delta<0$ is a fundamental discriminant,
\[
\left|L(1,\psi_h)\right|=\left|L(1,\psi_{|\Delta|})\right|\prod_{p\mid t}\left(1-\psi_{|\Delta|}(p)p^{-1}\right)\geq \left|L(1,\psi_{|\Delta|})\right|2^{-\omega(t)}\gg_{\varepsilon} \left|L(1,\psi_{|\Delta|})\right|t^{-2\varepsilon},
\]
Dirichlet's class number formula gives (where $w_{\Delta}$ is the size of the automorphism group of all binary quadratic forms of discriminant $\Delta$)
\[
h^{\frac{1}{2}}\left|L(1,\psi_h)\right|\gg_{\varepsilon} t^{1-2\varepsilon} |\Delta|^{\frac{1}{2}}\left|L(1,\psi_{|\Delta|})\right|=\frac{2\pi}{w_{\Delta}} t^{1-2\varepsilon}  h(\Delta)
\]
 and the lower bound  $h(\Delta)\gg |\Delta|^{\frac{1}{2}-\varepsilon}$ by Siegel \cite{Siegel}.
\end{proof}

\subsection{Ratios of representations by the genus}
Let $\mathcal{S}_{2^j}$ be the set of $d\in \N$ with either $d$ odd and squarefree or $d=2^jd'$ with $d'$ odd and squarefree. From \eqref{eqn:Siegelformula}, we have for $\bm{d}\in\mathcal{S}_1^3$ and $\bm{\ell}\in \mathcal{S}_1^3$ with $\gcd(d_1d_2d_3,\ell_1\ell_2\ell_3)=1$ that 
\begin{multline}\label{eqn:ratiorgendrgen1}
		\frac{r_{\pgen(X^{\bm{d}\bm{\ell}})}(h)}{r_{\pgen(X^{\bm{\ell}})}(h)}\\
		= \frac{1}{d_1d_2d_3}\cdot \frac{\prod_{p\mid 2(m-2)d_1d_2d_3\ell_1\ell_2\ell_3}\frac{b_p(h,\lambda_{\bm{d\ell}},0)(1-\psi_h(p)p^{-1})}{1-p^{-2}}}
		{\prod_{p\mid 2(m-2)\ell_1\ell_2\ell_3}\frac{b_p(h,\lambda_{\bm{\ell}},0)(1-\psi_h(p)p^{-1})}{1-p^{-2}}}\cdot \frac{\prod_{p\mid (h,2(m-2)d_1d_2d_3\ell_1\ell_2\ell_3) } \gamma_p^{-1}(3/2)}{\prod_{p\mid (h,2(m-2)\ell_1\ell_2\ell_3) } \gamma_p^{-1}(3/2)}\\
		= \frac{1}{d_1d_2d_3}\cdot 
		\prod_{p\mid (m-2)}\frac{b_p(h,\lambda_{\bm{d}\bm{\ell}},0)}{b_p(h,\lambda_{\bm{\ell}},0)}
		\cdot \prod_{\substack{p\mid d_1d_2d_3\\ p\nmid 2(m-2)}}b_p(h,\lambda_{\bm{d}},0)\frac{(1-\psi_h(p)p^{-1})}{1-p^{-2}} \cdot \prod_{\substack{p\mid (h,d_1d_2d_3)\\ p\nmid 2(m-2)} } \gamma_p^{-1}(3/2),
\end{multline}
where in the last step we used $b_p(h,\lambda_{\bm{d\ell}},0)=b_p(h,\lambda_{\bm{\ell}},0)$ for $p\mid \ell_1\ell_2\ell_3$, which follows from the fact that  $b_p(h,\lambda_{\bm{d}},0)=b_p(h,\lambda_{\bm{p}^{\ord_p\bm{d}}},0)$ for general $\bm{d}\in \N^3$. Define
\begin{multline}\label{eqn:defn_beta_p}
	\beta_{X^{\bm{p^{c}}},p}(h):= \frac{1}{p^{c_1+c_2+c_3}}
	\left(\frac{b_p(h,\lambda_{\bm{p^c}},0)}{b_p(h,\lambda_{\bm{1}},0)}\right)^{\delta_{p\mid (m-2)}}\\
	\cdot\left(b_p(h,\lambda_{\bm{p^c}},0)\frac{(1-\psi_h(p)p^{-1})}{1-p^{-2}} \gamma_p(3/2)^{-1}\right)^{\delta_{p\nmid 2(m-2), \, \bm{c}\neq \bm{0}}}.
\end{multline}
Noting that $b_p(h,\lambda_{\bm{d}},0)=b_p(h,\lambda_{\bm{p}^{\ord_p\bm{d}}},0)$ and that $\gamma_p(3/2)=1$ if $p\nmid h$; and combining \eqref{eqn:ratiorgendrgen1} and \eqref{eqn:defn_beta_p}, we have for $\bm{d}\in\mathcal{S}_1^3$ that
\begin{equation}\label{eqn:ratiorgendrgen1intoprod}
	\frac{r_{\pgen(X^{\bm{d}\bm{\ell}})}(h)}{r_{\pgen(X^{\bm{\ell}})}(h)}
	=\prod_{p\text{ odd}} \beta_{X^{\bm{p}^{\ord_p(\bm{d})}},p}(h),
\end{equation}
where the product runs over all odd prime numbers. We next compute the $\beta_{X^{\bm{p}^{\ord_p(\bm{d})}},p}(h)$ so that the dependence on $\bm{d}$ can be dealt with in the individual components.

\begin{lemma}\label{lem:beta} 
Let $m$ be an odd integer such that $m\not\equiv 4\pmod{3}$, $m\not\equiv 4\pmod{5}$, $p$ be an odd prime, $h=8(m-2)n+3(m-4)^2$ and $\psi_h(\cdot)$ be the primitive character associated to $\legendre{-h}{\cdot}$.

	\begin{enumerate}[leftmargin=*,label=\rm(\arabic*)]
		\item If $p\mid (m-2)$, then we have 
		\[
			\beta_{X^{\bm{p^{c}}},p}(h)=\begin{cases}
				p^{-(c_1+c_2+c_3)+\min_j\{c_j \}} & \text{if } n \in p^{\min_j\{c_j \}}\Z_p,\\
				0	&	\text{otherwise.}
			\end{cases}
		\]
		\item If  $p\nmid (m-2)$, then we have $\beta_{X^{(1,1,1)},p}(h)=1$; and for $\bm{c}\neq\bm{0}$ we have
		\[
		\beta_{X^{\bm{p^{c}}},p}(h)=\frac{1}{p^{c_1+c_2+c_3}} b_p(h,\lambda_{\bm{p^c}},0) w_p(h),
		\]
		where
		\begin{equation}\label{eqn:wphdef}
		w_p(h):=
		\begin{cases}
			\frac{1}{1+p^{-1}} & \text{if } \psi_h(p)=1,\\
			\frac{1}{1+p^{-1}-2p^{-\lfloor a/2\rfloor-1}} & \text{if } \psi_h(p)=-1,\\
			\frac{1}{(1+p^{-1})(1-p^{-\lfloor a/2\rfloor-1})} & \text{if } \psi_h(p)=0.
		\end{cases}	
		\end{equation}
		\item Assume that $p\nmid (m-2)(m-4)$. Then for $p\nmid h-(m-4)^2$, we have $\frac{1}{p}\cdot\frac{p-1}{p+1}\leq \beta_{X^{(p,1,1)},p}(h)\le \frac{1}{p}\cdot\frac{p+1}{p-1}$; for $p\mid h-(m-4)^2$, we have
		\[
			\beta_{X^{(p,1,1)},p}(h)= \frac{1}{p}\cdot 
			\begin{cases}
				\frac{2p-1}{p+1} & \text{if } p\equiv 1 \pmod{4},\\
				\frac{1}{p-1} & \text{if } p\equiv -1 \pmod{4}.
			\end{cases}
		\]
		\item Let $\bm{d}=(d_1,d_2,d_3)$ be a tuple of squarefree integers such that $p\mid d_1d_2d_3$ implies that $p\nmid (m-2)(m-4)$. We have
		\[
			\prod_{p\mid\bm{d}} \beta_{X^{\bm{d}},p}(h) \le \frac{\widetilde{\omega}(d_1)\widetilde{\omega}(d_2)\widetilde{\omega}(d_3)}{d_1d_2d_3}
		\]
		where $\widetilde{\omega}$ is defined multiplicatively on squarefree integers with
		\[
		\frac{\widetilde{\omega}(p)}{p}:=\begin{cases}\frac{2}{p}&\text{if }p\nmid n,\\
			\max\left\{\left(\frac{1}{p(p-1)}\right)^{\frac{1}{3}}, \frac{2}{p}\right\}&\text{if }p\mid n.\\
		\end{cases}
		\]
\item The ratio
\[
g(\bm{d}):=\prod_{p\mid \bm{d}}  \frac{\beta_{X^{\bm{d}},p}(h)}{\prod_{j=1}^3\beta_{X^{(d_j,1,1)},p}(h)},
\]
only depends on $u_{i,j}:=\gcd(d_i,d_j)$ and we have $g(\bm{d})\ll \max_{i,j}(u_{i,j})^{12}$.
	\end{enumerate}
\end{lemma}
\begin{proof}
	
	%
	%
	(1) If $p\mid (m-2)$, then \eqref{eqn:defn_beta_p} and Lemma \ref{lemma:localdensity:podddividingm-2} imply the lemma immediately.
	\vspace{.05in}
	
	\noindent (2) If $p\nmid (m-2)$, then we first have $\beta_{X^{(1,1,1)},p}(h)=1$. Plugging \eqref{eqn:gamma3over2} into \eqref{eqn:defn_beta_p}, we have 
	\[
	\frac{(1-\psi_h(p)p^{-1})}{1-p^{-2}} \gamma_p(3/2)^{-1}=
	\begin{cases}
		\frac{1}{1+p^{-1}} & \text{if } \psi_h(p)=1,\\
		\frac{1}{1+p^{-1}-2p^{-\lfloor a/2\rfloor-1}} & \text{if } \psi_h(p)=-1,\\
		\frac{1}{(1+p^{-1})(1-p^{-\lfloor a/2\rfloor-1})} & \text{if } \psi_h(p)=0,
	\end{cases}	
	\]
	which equals to $w_p(h)$.
	\vspace{.05in}
	
	\noindent (3) Assume that $p\nmid (m-2)(m-4)$. From part (2) and Lemma \ref{lemma:localdensity:poddnotdividingm-2m-4}, we have
	\begin{multline}
		\beta_{X^{(p,1,1)},p}(h)=\frac{1}{p} b_p(h,\lambda_{(p,1,1)},0) w_p(h)\\
		=\frac{w_p(h)}{p} \left(1+\legendre{-1}{p}\frac{1}{p}\begin{cases}p-1 &\text{if } h-(m-4)^2\in p\Z_p \\-1 &\text{if } h-(m-4)^2 \not\in p\Z_p\end{cases}\right).
	\end{multline}
	If $p \mid h-(m-4)^2$, then $\psi_h(p)=\legendre{-h}{p}=\legendre{-1}{p}$. Hence we have 
	\[
		\beta_{X^{(p,1,1)},p}(h)=\frac{w_p(h)}{p}\left(1+\legendre{-1}{p}\frac{p-1}{p}\right)
		=\frac{1}{p}\cdot\begin{cases}
				\left(1+\frac{p-1}{p}\right) \frac{1}{1+p^{-1}}  & \text{if } p\equiv 1 \pmod{4},\\
				\left(1-\frac{p-1}{p}\right) \frac{1}{1-p^{-1}} & \text{if } p\equiv -1 \pmod{4},
		\end{cases}
	\]
	which yields the lemma. Let us now assume that $p \nmid h-(m-4)^2$. Note that
	\begin{equation}\label{eqn:w_p(h)bound}
	0<\frac{1}{1+p^{-1}}\le w_p(h) \le\frac{1}{1-p^{-1}}.
	\end{equation}
	Thus we have 
	\[
		\beta_{X^{(p,1,1)},p}(h)=\frac{1}{p}\left(1-\legendre{-1}{p}\frac{1}{p}\right)w_p(h)\le \frac{1}{p}\left(1+\frac{1}{p}\right)\frac{1}{1-p^{-1}} = \frac{1}{p}\cdot\frac{p+1}{p-1}.
	\]
and
	\[
		\beta_{X^{(p,1,1)},p}(h)=\frac{1}{p}\left(1-\legendre{-1}{p}\frac{1}{p}\right)w_p(h)\ge \frac{1}{p}\left(1-\frac{1}{p}\right)\frac{1}{1+p^{-1}} = \frac{1}{p}\cdot\frac{p-1}{p+1}.
	\]
	\vspace{.05in}
	
	\noindent (4) Since $\widetilde{\omega}$ is multiplicative, we need only to show for any $p\nmid 2(m-2)(m-4)$ that 
	\[
		\beta_{X^{\bm{p^c}},p}(h)\le \left(\frac{\widetilde{\omega}(p)}{p}\right)^{c_1+c_2+c_3}.
	\]
	By symmetry and from the definition of $\widetilde{\omega}(p)$, it suffice to show that
	\begin{equation}\label{eqn:beta4bound}
		\beta_{X^{(p,1,1)},p}(h)\le \frac{2}{p}, \quad \beta_{X^{(p,p,1)},p}(h)\le \frac{4}{p^2} \quad \text{and} \quad \beta_{X^{(p,p,p)},p}(h)\le \begin{cases}
			\frac{8}{p^3} & \text{if } p\nmid n,\\
			\frac{1}{p(p-1)} & \text{if } p\mid n.
		\end{cases}
	\end{equation}
	The first inequality follows immediately from part (3). The second inequality also follows using part (2), Lemma \ref{lemma:localdensity:poddnotdividingm-2m-4}, and \eqref{eqn:w_p(h)bound} as
	\[
		\beta_{X^{(p,p,1)},p}(h)=\frac{w_p(h)}{p^2}\left(1+\legendre{8(m-2)n+(m-4)^2}{p}\right)\le \frac{1}{1-p^{-1}}\frac{2}{p^2}\le \frac{4}{p^2}.
	\]
	To show the third inequality, note from part (2) and Lemma \ref{lemma:localdensity:poddnotdividingm-2m-4} that 
	\[
		\beta_{X^{(p,p,p)},p}(h)=\frac{w_p(h)}{p^3}\begin{cases}p &\text{if } n\in p\Z_p, \\0&\text{if } n\not\in p\Z_p. \end{cases}
	\]
	Hence we need only to show the inequality when $p\mid n$, and this follows using \eqref{eqn:w_p(h)bound} as
	\[
		\frac{w_p(h)}{p^2}\le \frac{1}{1-p^{-1}}\frac{1}{p^2} = \frac{1}{p(p-1)}.
	\]
\noindent

\noindent
(5) From part (4), we have an upper bound for $p\nmid (m-2)(m-4)$ sufficiently large
\[
\beta_{X^{\bm{d}},p}(h)\leq 
\begin{cases}
\left(\frac{2}{p}\right)^{\ord_p(d_1)+\ord_p(d_2)+\ord_p(d_3)}&\text{if }p\nmid n,\\
\left(\frac{1}{p(p-1)}\right)^{\frac{\ord_p(d_1)+\ord_p(d_2)+\ord_p(d_3)}{3}} &\text{if }p\mid n.
\end{cases}
\]
 and part (3) gives a lower bound 
\[
\beta_{X^{(p,1,1)},p}(h)\geq \frac{1}{p(p-1)}.
\]
The ratio is furthermore $1$ for $p\nmid u_{1,2}u_{1,3}u_{2,3}$. Hence 
\begin{align*}
g(\bm{d})&=\prod_{p\mid u_{1,2}u_{1,3}u_{2,3}}\frac{\beta_{X^{\bm{d}},p}(h)}{\prod_{j=1}^3 \beta_{X^{(d_j,1,1)},p}(h)}\ll \prod_{p\mid u_{1,2}u_{1,3}u_{2,3}} \left(\frac{p(p-1)}{(p(p-1))^{\frac{1}{3}}}\right)^{\ord_p(d_1)+\ord_p(d_2)+\ord_p(d_3)}\\
&\leq \prod_{p\mid u_{1,2}u_{1,3}u_{2,3}} p^2(p-1)^2\leq u_{1,2}^4u_{1,3}^4u_{2,3}^4\leq \max\left(u_{1,2},u_{1,3},u_{2,3}\right)^{12}.\qedhere
\end{align*}
\end{proof}
We later require bounds on 
\[
1-\beta_{X^{(p,1,1)},p}(h)\qquad\text{ and }\qquad \sum_{\bm{c}\in \{0,1\}^3}\beta_{X^{p^{\bm{c}}},p}(h).
\]

\begin{lemma}\label{lem:betabound}
Suppose that $m\not\equiv 4\pmod{3}$, $m\not\equiv 4\pmod{5}$, and $m$ is odd. 
\begin{enumerate}[leftmargin=*,label=\rm(\arabic*)]
\item For any odd prime $p\mid (m-2)$, we have 
\[
1-\beta_{X^{(p,1,1)},p}(h)=	1-\frac{1}{p}.
\]
and 
\[
\sum_{\bm{c}\in \{0,1\}^3} \beta_{X^{\bm{p^{c}}},p}(h)= \left(1+\frac{1}{p}\right)^3.
\]
\item For any odd prime $p\mid (m-4)$ (note that $p\neq 3$ and $p\neq 5$ by assumption), we have 
\[
1-\beta_{X^{(p,1,1)},p}(h)\geq 1-\frac{3}{p}
\]
and 
\[
\sum_{\bm{c}\in \{0,1\}^3} \beta_{X^{\bm{p^{c}}},p}(h)\leq \left(1+\frac{1}{p}\right)^{13}.
\]

\item For any odd prime $p\nmid (m-2)(m-4)$, we have 
\[
1-\beta_{X^{(p,1,1)},p}(h)\geq 1-\frac{2}{p}
\]
and 
\[
\sum_{\bm{c}\in \{0,1\}^3} \beta_{X^{\bm{p^{c}}},p}(h)\leq  \left(1+\frac{1}{p}\right)^6
\]
\end{enumerate}
\end{lemma}
\begin{proof}
Recall the evaluation \eqref{eqn:gamma3over2} of $\gamma_p(3/2)$, the evaluations in Lemmas \ref{lemma:localdensity:podddividingm-2} and \ref{lemma:localdensity:poddnotdividingm-2butm-4}  for $b_p(h,\lambda_{\bm{d}},0)$, and the definition \eqref{eqn:wphdef} and bound \eqref{eqn:w_p(h)bound} for $w_p(h)$, which are frequently used.
\noindent

\noindent
%
%
(1) If $p\mid (m-2)$, then \eqref{eqn:defn_beta_p} and Lemma \ref{lemma:localdensity:podddividingm-2} imply that
\[
\beta_{X^{\bm{p^{c}}},p}(h)=p^{-(c_1+c_2+c_3)+\min_j\{c_j \}}.
\]
Hence
\begin{equation}\label{eqn:mainpmidm-2}
1-\beta_{X^{(p,1,1)},p}(h)= 1-\frac{1}{p}
\end{equation}
and
\begin{equation}\label{eqn:errorpmidm-2}
\sum_{\bm{c}\in \{0,1\}^3} \beta_{X^{\bm{p^{c}}},p}(h)
= 1+\frac{3}{p}+\frac{3}{p^2} +\frac{1}{p^3}=\left(1+\frac{1}{p}\right)^3,
\end{equation}
 yielding (1).
\vspace{.05in}

\noindent (2) If $p\mid (m-4)$, then write $h=u p^a$ with $u\in\Z_p^\times$. Then, by Lemma \ref{lem:beta} (2), we have 
\[
\beta_{X^{\bm{p^{c}}},p}(h)=\frac{1}{p^{c_1+c_2+c_3}}b_p(h,\lambda_{\bm{p^c}},0)^{\delta_{\bm{c}\neq \bm{0}}} \cdot w_p(h)^{\delta_{\bm{c}\neq \bm{0}}}.
\]
Using Lemma \ref{lemma:localdensity:poddnotdividingm-2butm-4} that
\begin{multline}\label{eqn:mainpmidm-4}
1-\beta_{X^{(p,1,1)},p}(h)
	= 1-\frac{1}{p}b_p(h,\lambda_{\bm{p^c}},0)w_p(h)\\
	= 1 -\frac{w_p(h)}{p}  \begin{cases}
		2+\left(1-\frac{1}{p}\right)\legendre{-1}{p}
		- \frac{1}{p^{\lfloor a/2\rfloor}} + (-1)^{a}\legendre{-u}{p}^{a+1}\frac{1}{p^{\lceil a/2\rceil}}& \text{if } a\ge 2,\\
		\left(1-\frac{1}{p}\right) \left(1+\legendre{-1}{p}\right) &\text{if } a=1,\\
		1-\legendre{-1}{p}\frac{1}{p} & \text{if }a=0.
	\end{cases} \\
	\geq 1 - \frac{w_p(h)}{p}\cdot \begin{cases}
		2+\left(1-\frac{1}{p}\right)\legendre{-1}{p}&\text{if } a\ge 2,\\
		\left(1-\frac{1}{p}\right) \left(1+\legendre{-1}{p}\right)& \text{if } a=1,\\
		1-\legendre{-1}{p}\frac{1}{p}&\text{if }a=0.
	\end{cases}
\end{multline}
We then use \eqref{eqn:w_p(h)bound} to obtain $0<w_p(h)\le \frac{1}{1-p^{-1}}$ if $a\le 1$ and $0<w_p(h)\le \frac{1}{1-p^{-2}}$ if $a\ge 2$, yielding that one may bound the last line of \eqref{eqn:mainpmidm-4} from below by
\[
	\geq 1 - \frac{w_p(h)}{p}\cdot \begin{cases}
	3-\frac{1}{p} &\text{if } a\ge 2\\
	2-\frac{2}{p} &\text{if } a=1\\
	1+\frac{1}{p} &\text{if }a=0
	\end{cases}\quad 
	\geq 1-\frac{3}{p},
\]
yielding the bound as stated. Finally, in the case when $p=7$, noting that $\legendre{-1}{7}=-1$ and $0<w_p(h)\le \frac{1}{1-7^{-1}}=\frac{7}{6}$, one may easily bound the last line of \eqref{eqn:mainpmidm-4} from below by $\frac{3}{7}$.

On the other hand, Lemma \ref{lemma:localdensity:poddnotdividingm-2butm-4} again implies that
\begin{align}\label{eqn:errorpmidm-4}	
\sum_{\bm{c}\in \{0,1\}^3} &\beta_{X^{\bm{p^{c}}},p}(h)=1+\sum_{{\bm{c}\in \{0,1\}^3}\setminus\{\bm{0}\}}\frac{1}{p^{c_1+c_2+c_3}}b_p(h,\lambda_{\bm{p^c}},0)w_p(h)\\
\nonumber	&= 1+w_p(h)\cdot\frac{3}{p} \begin{cases}
		2+\left(1-\frac{1}{p}\right)\legendre{-1}{p} - \frac{1}{p^{\lfloor a/2\rfloor}} + (-1)^{a}\legendre{-u}{p}^{a+1}\frac{1}{p^{\lceil a/2\rceil}}& \text{if } a\ge 2,\\
		\left(1-\frac{1}{p}\right) \left(1+\legendre{-1}{p}\right) &\text{if } a=1,\\
		1-\legendre{-1}{p}\frac{1}{p} & \text{if }a=0.
	\end{cases}\\
\nonumber &\qquad	+w_p(h)\cdot\frac{3}{p^2}\begin{cases}
		1+p- \frac{1}{p^{\lfloor a/2\rfloor-1}} + (-1)^{a}\legendre{-u}{p}^{a+1}\frac{1}{p^{\lceil a/2\rceil-1}}& \text{if } a\ge 2,\\
		0 &\text{if } a=1,\\
		1+\legendre{u}{p} & \text{if }a=0.
	\end{cases}\\
\nonumber &\qquad	+w_p(h)\cdot\frac{1}{p^3}\begin{cases}
		p^2+p-\frac{1}{p^{\lfloor a/2\rfloor-2}} + (-1)^{a}\legendre{-u}{p}^{a+1}\frac{1}{p^{\lceil a/2\rceil-2}}& \text{if } a\ge 2,\\
		0 &\text{if } a\le 1.
	\end{cases}\\
\nonumber 	&\leq 1+ w_p(h)\cdot\begin{cases}
		\frac{3}{p}\left(2+\left(1-\frac{1}{p}\right)\legendre{-1}{p}\right)+\frac{3}{p^2}\left(1+p\right)+\frac{1}{p^3}\left(p^2+p\right)& \text{if } a\ge 2,\\
		\frac{3}{p}\left(1-\frac{1}{p}\right) \left(1+\legendre{-1}{p}\right)& \text{if } a=1,\\
		\frac{3}{p}\left(1-\legendre{-1}{p}\frac{1}{p}\right)+\frac{3}{p^2}\left(1+\legendre{u}{p}\right)&  \text{if }a=0.
	\end{cases}\\
\nonumber &\leq 1+ w_p(h)\cdot\begin{cases}
	\frac{13}{p}+\frac{4}{p^2} & \text{if } a\ge 2,\\
	\frac{6}{p}-\frac{6}{p^2} &  \text{if } a=1,\\
	\frac{3}{p}+\frac{9}{p^2} &  \text{if }a=0.
	\end{cases}
\end{align}
For \eqref{eqn:errorpmidm-4} the worst case is $a\geq 2$, which can be easily bounded by 
\[
 1+w_p(h)\left(\frac{13}{p}+\frac{4}{p^2}\right)\leq  1+\frac{1}{1-p^{-1}}\left(\frac{13}{p}+\frac{4}{p^2}\right) \leq  \left(1+\frac{1}{p}\right)^{13},
\]
 yielding the claim in (2).
\vspace{.05in}

\noindent (3) Now suppose that $p\nmid 2(m-2)(m-4)$. Then Lemma \ref{lem:beta} directly gives 
\[
1-\beta_{X^{(p,1,1)},p}(h)\geq 1-\frac{2}{p}.
\]
On the other hand, Lemma \ref{lemma:localdensity:poddnotdividingm-2m-4} again implies that
\begin{multline}\label{eqn:errorpnmidm-4}
	\sum_{\bm{c}\in \{0,1\}^3} \beta_{X^{\bm{p^{c}}},p}(h)=1+\sum_{{\bm{c}\in \{0,1\}^3}\setminus\{\bm{0}\}}\frac{1}{p^{c_1+c_2+c_3}}b_p(h,\lambda_{\bm{p^c}},0)w_p(h)\\
	= 1+w_p(h)\cdot\frac{3}{p} \left(1+\legendre{-1}{p}\frac{1}{p}\begin{cases}p-1 &\text{if } 8(m-2)n+2(m-4)^2\in p\Z_p \\-1 &\text{if } 8(m-2)n+2(m-4)^2\notin p\Z_p\end{cases}\right)\\
	+w_p(h)\cdot\frac{3}{p^2}\left(1+\legendre{8(m-2)n+(m-4)^2}{p}\right)
	+w_p(h)\cdot\frac{1}{p^3}\begin{cases}p &\text{if } n\in p\Z_p \\0&\text{if } n\not\in p\Z_p \end{cases}\\
	=1+w_p(h)\cdot
	\begin{cases}
		\frac{3}{p}\left(1-\legendre{-1}{p}\frac{1}{p}\right)+\frac{3}{p^2}\left(1+1\right)+\frac{1}{p^2} &\text{if } n\in p\Z_p, \\
		\frac{3}{p}\left(1+\legendre{-1}{p}\left(1-\frac{1}{p}\right)\right)+\frac{3}{p^2}\left(1+\legendre{-1}{p}\right)&\text{if } n\not\in p\Z_p\text{ and }h-(m-4)^2\in p\Z_p,\\
 		\frac{3}{p}\left(1-\legendre{-1}{p}\frac{1}{p}\right)+\frac{3}{p^2}\left(1+\legendre{8(m-2)n+(m-4)^2}{p}\right)&\text{if } n\not\in p\Z_p\text{ and }h-(m-4)^2\notin p\Z_p, 
	\end{cases}\\
	\leq 1+ w_p(h)\cdot \begin{cases}
		\frac{3}{p}+\frac{1}{p^2}\left(7-3\legendre{-1}{p}\right) &\text{if } n\in p\Z_p \\
		\frac{6}{p}+\frac{3}{p^2} \text{ or } 1+\frac{3}{p^2} &\text{if } n\not\in p\Z_p\text{ and }h-(m-4)^2\in p\Z_p,\\
		\frac{3}{p}+\frac{9}{p^2}&\text{if } n\not\in p\Z_p\text{ and }h-(m-4)^2\notin p\Z_p. 
	\end{cases}
\end{multline}
The worst case in \eqref{eqn:errorpnmidm-4} is the case $n\notin p\Z_p$ and $h-(m-4)^2\in p\Z_p$, yielding an upper bound 
\[
1+w_p(h)\cdot\left(\frac{6}{p}+\frac{3}{p^2}\right)\leq 1+\frac{1}{1-p^{-1}}\cdot\left(\frac{6}{p}+\frac{3}{p^2}\right)\leq  \left(1+\frac{1}{p}\right)^6.\qedhere
\]
\end{proof}
It is useful to collect the following bounds for the local densities for $p\mid\sf(h)$.
\begin{lemma} \label{lem:betaboundspecialcase}
	Assume that $p\mid \sf(h)$ and $m\not\equiv 4\pmod{3}$.
	\begin{enumerate}[leftmargin=*,label=\rm(\arabic*)]
		\item If $p=3$ and $3\mid (m-2)$, then we have
		\[
			\sum_{\bm{c}\in \{0,1\}^3} \beta_{X^{\bm{3^{c}}},3}(h) \le \left(1+\frac{1}{3}\right)^3 = \frac{64}{27}.
		\]
		\item If $p=3$ and $3\nmid (m-2)$, then we have
		\[
			\sum_{\bm{c}=(1,1,1)} \beta_{X^{\bm{3^{c}}},3}(h)\leq \frac{1}{6}.
		\]
		\item If $p\nmid 3(m-2)(m-4)$, then we have
		\[
			\sum_{\#\{j:c_j=1\}=2} \beta_{X^{\bm{p^{c}}},p}(h)\leq \frac{1+p^{-1}}{1-p^{-2}} \cdot \frac{6}{p^2}
		\]
	\end{enumerate}
\end{lemma}

\begin{proof}
\noindent (1) This follows from Lemma \ref{lem:betabound}.

\noindent (2) Since $3\nmid (m-4)$ by assumption, Lemma \ref{lemma:localdensity:poddnotdividingm-2m-4} implies that
\[
\sum_{\bm{c}=(1,1,1)} \beta_{X^{\bm{3^{c}}},3}(h)= w_3(h)\cdot \frac{1}{3^3}\cdot \begin{cases}
	3 & \text{if } n\in 3\Z_3,\\
	0 & \text{if } n\notin 3\Z_3
\end{cases}
\quad \leq \frac{1}{1-3^{-1}} \cdot \frac{1}{3^3} \cdot 3 =\frac{1}{6}.
\]

\noindent (3) By Lemma \ref{lemma:localdensity:poddnotdividingm-2m-4}, we have
\[
\sum_{\#\{j:c_j=1\}=2} \beta_{X^{\bm{p^{c}}},p}(h)=w_p(h)\cdot\frac{3}{p^2}\left(1+\legendre{8(m-2)n+(m-4)^2}{p}\right)\leq \frac{1}{1-p^{-1}} \cdot \frac{6}{p^2}
\]
which implies the lemma.
\end{proof}

\section{Unary theta functions occuring in $\Theta_{X^{\bm{d}}}$}\label{sec:unary}

\subsection{Computation of local spinor norm groups of $X^{\bm{d}}$}

Recall that  $\mathcal{S}_{2^j}$ is the set of $d\in \N$ with either $d$ odd and squarefree or $d=2^jd'$ with $d'$ odd and squarefree.
\begin{proposition}\label{localspinornorm:oddprime}
	Let $\bm{d}\in (\cup_{j=0}^{\infty}\mathcal{S}_{2^j})^3$, and let $p$ be an odd prime.
	Then the local spinor norm group $\spnnorm{X^{\bm{d}}_p}$ of $X^{\bm{d}}$ at $p$ is given as follows.
	\begin{enumerate}[leftmargin=*,label={\rm (\arabic*)}]
		\item If $p\nmid (m-2)d_1d_2d_3$, then $\spnnorm{X^{\bm{d}}_p}=\Z_p^\times(\Q_p^\times)^2$.
		\item Assume that $p\mid (m-2)d_1d_2d_3$ and let $i,j,k$ be indices such that $\{i,j,k\}=\{1,2,3\}$.
		\begin{enumerate}
			\item If $p\nmid d_1d_2d_3$ or $p\mid d_i$ for all $i=1,2,3$, then $\spnnorm{X^{\bm{d}}_p}=
			\begin{cases}
			(\Q_3^\times)^2& \text{if } p=3,\\
			\Z_p^\times(\Q_p^\times)^2& \text{if } p\neq 3.
			\end{cases}
			$
			\item If $p\mid d_i$ and $p\nmid d_jd_k$, then 
					$\spnnorm{X^{\bm{d}}_p}=
					\begin{cases}
						\Z_p^\times(\Q_p^\times)^2 & \text{if } p\nmid (m-2),\\ 
						(\Q_p^\times)^2 & \text{if } p\mid (m-2).
					\end{cases}$
			\item If $p\mid d_i$, $p\mid d_j$, $p\nmid d_k$, and $p\nmid(m-2)$, then
					\[
					\spnnorm{X^{\bm{d}}_p}=
					\begin{cases}
						\Z_p^\times(\Q_p^\times)^2 & \text{if } p\mid (m-4) \text{ or } p\nmid(m-4) \text{ with } \legendre{2}{p}=-1,\\ 
						(\Q_p^\times)^2 & \text{if } p\nmid (m-4)\text{ with } \legendre{2}{p}=1.
					\end{cases}
					\]
			\item If $p\mid d_i$, $p\mid d_j$, $p\nmid d_k$, and $p\mid (m-2)$, then $\spnnorm{X^{\bm{d}}_p}=(\Q_p^\times)^2$.
		\end{enumerate}
	\end{enumerate}
\end{proposition}

\begin{proof}
 For the ease of notation, we denote $X=X^{\bm{d}}$ and $L=L^{\bm{d}}$ throughout the proof of the proposition. 
	We first note that since $O^+(X_p)\subseteq O^+(L_p)$ and $\spnnorm{L_p}= \spnnorm{\left<d_1^2, d_2^2,d_3^2\right>}$, 
	\[
	(\Q_p^{\times})^2\subseteq \spnnorm{X_p}\subseteq \spnnorm{L_p}\subseteq\Z_p^\times(\Q_p^\times)^2.
	\]
	Hence $\spnnorm{X_p}$ could be either $(\Q_p^\times)^2$ or $\Z_p^\times(\Q_p^\times)^2$.

\noindent
(1) In this case, we have $X_p=L_p\cong\left<1,1,1\right>$ and it is well known that 
\[
\spnnorm{\left<1,1,1\right>}=\Z_p^{\times}(\Q_p^{\times})^2.
\]

\noindent
(2) Assume that $p\nmid d_1d_2d_3$ and $p\neq 3$. Let $w=\frac{2(m-2)}{(m-4)}\nu=e_1+e_2+e_3$ and note that
	\[
	L_p=\Z_p\left[w,d_1e_1-\frac{d_1}{3}w,d_2e_2-\frac{d_2}{3}w\right]\cong \Z_p[w]\perp K_p,
	\]
	where $K_p$ is a binary lattice with $K_p\cong \frac{4}{3}(m-2)^2
	\left[\begin{matrix}
		2d_1^2& -d_1d_2\\
		-d_1d_2& 2d_2^2 
	\end{matrix}\right]$.
	Any isometry in $O^+(K_p)$ can be extended to an isometry $\sigma\in O^+(L_p)$ satisfying $\sigma(w)=w$, equivalently, $\sigma(\nu)=\nu$, and therefore $\sigma \in O^+(L_p+\nu)$. Thus, since $\spnnorm{K_p}=\Z_p^{\times}(\Q_p^{\times})^2$, we may conclude that 
\[
\Z_p^\times(\Q_p^{\times})^2 =\spnnorm{K_p}\subseteq \spnnorm{L_p+\nu},
\]
so  $\spnnorm{L_p+\nu}=\Z_p^\times(\Q_p^{\times})^2$.
	
	Assume that $p\mid d_i$ for all $i=1,2,3$ and $p\neq 3$. Then with the same notation, we have 
	\[
	L_p=\Z_p\left[pw,d_1e_1-\frac{d_1}{3}w,d_2e_2-\frac{d_2}{3}w\right]\cong \Z_p[pw]\perp K_p,
	\]
	Hence we may similarly conclude that $\spnnorm{L_p+\nu}=\Z_p^\times(\Q_p^{\times})^2$.
	We have proved the proposition for the case (2)-(a) with $p\neq 3$ so far.
	
	In proving the remaining cases, we make use of \cite[Theorem 2]{TeterinSpinornorm} to compute $\spnnorm{L_p+\nu}$.
	Let $M_p=L_p+\Z_p\nu$ and as defined in \cite{TeterinSpinornorm}, let
	\[
	O(M_p/L_p)=\{\sigma\in O(V_p):\sigma(x)\in x+L_p \text{ for all } x\in M_p\}.
	\]
	Note that $\sigma\in O(M_p/L_p)$ if and only if $\sigma(L_p)=L_p$ and $\sigma(\nu)\equiv \nu \pmod{L_p}$. Thus, $O(M_p/L_p)=O(L_p+\nu)$ and $O^+(M_p/L_p)=O^+(L_p+\nu)$. From \cite[Theorem 2]{TeterinSpinornorm}, we know that $\spnnorm{M_p/L_p}$ and hence $\spnnorm{L_p+\nu}=\spnnorm{X_p}$ is generated by pairs of spinor norms of symmetries coming from $O(M_p/L_p)$.
	
	Let $\tau$ be a symmetry in $O(V_p)$. Then there is a primitive vector $x=x_1(d_1e_1)+x_2(d_2e_2)+x_3(d_3e_3)\in L_p$ such that
	\[
	\tau(y)=\tau_x(y)=y-\frac{2B(x,y)}{Q(x)}x
	\]
	for any $y\in L_p$. Note that
	\[
	\begin{aligned}
	\tau_x(d_ie_i)&=d_ie_i-\frac{2\cdot 4(m-2)^2d_i^2x_i}{4(m-2)^2(d_1^2x_1^2+d_2^2x_2^2+d_3^2x_3^2)}x
	=d_ie_i-\frac{2d_i^2x_i}{d_1^2x_1^2+d_2^2x_2^2+d_3^2x_3^2}x,\\
	\tau_x(\nu)&=\nu-\frac{2\cdot4(m-2)^2\frac{4-m}{2(m-2)}(d_1x_1+d_2x_2+d_3x_3)}{4(m-2)^2(d_1^2x_1^2+d_2^2x_2^2+d_3^2x_3^2)}x
	=\nu+\frac{(m-4)(d_1x_1+d_2x_2+d_3x_3)}{(m-2)(d_1^2x_1^2+d_2^2x_2^2+d_3^2x_3^2)}x.
	\end{aligned}
	\]
	Hence $\tau_x\in O(M_p/L_p)$ if and only if 
	\begin{gather}
		\frac{2d_i^2x_i}{d_1^2x_1^2+d_2^2x_2^2+d_3^2x_3^2}\in \Z_p \text{ for all } i=1,2,3, \quad \text{and} \label{symmetrycond:1} \\ 
		\frac{(m-4)(d_1x_1+d_2x_2+d_3x_3)}{(m-2)(d_1^2x_1^2+d_2^2x_2^2+d_3^2x_3^2)}\in\Z_p.\label{symmetrycond:2}
	\end{gather}
Since 
\[
\spnnorm{X_p}=\left\{ Q(x_1)Q(x_2)(Q_p^{\times})^2: \tau_{x_1},\tau_{x_2}\in O(M_p/L_p)\right\},
\]
we next use \eqref{symmetrycond:1} and \eqref{symmetrycond:2} to compute the set $\{Q(x) (\Q_p^\times)^2: \tau_x \in O(M_p/L_p)\}$ case-by-case according to the different cases that occur in the statement of the proposition.
	Assume that $\tau_x\in O(M_p/L_p)$.\\ 

	\noindent{\bf Case (2)-(a) with $p=3$}.
	
	Let $(d_1,d_2,d_3,3)=g$. Assume without loss of generality that $x_1\in\Z_3^\times$. Then \eqref{symmetrycond:1} for $i=1$ implies that $(d_1/g)^2x_1^2+(d_2/g)^2x_2^2+(d_3/g)^2x_3^2 \in \Z_3^\times$, from which we have $x_1^2+x_2^2+x_3^2\not\equiv 0\pmod{3}$ since $(d_i/g)^2\equiv 1 \pmod{3}$. Hence we may assume without loss of generality that $x_3\in 3\Z_3$.
	On the other hand, \eqref{symmetrycond:2} implies that $d_1x_1+d_2x_2+d_3x_3\equiv 0\pmod{3g}$.
	Hence we have $(d_1/g)x_1+(d_2/g)x_2 \equiv 0\pmod{3}$ so that $(d_1/g)^2x_1^2+(d_2/g)^2x_2^2+(d_3/g)^2x_3^2 \equiv 2 \pmod{3}$. Thus,
		\[
	\{Q(x)(\Q_3^\times)^2: \tau_x \in O(M_3/L_3)\} \subseteq \{2(\Q_3^\times)^2 \},
	\]
	so we may conclude that $\spnnorm{L_3+\nu}\subseteq (\Q_3^\times)^2$.\\
	
	\noindent{\bf Case (2)-(b)} Assume without loss of generality that $p\mid d_1$, $p \nmid d_2d_3$. 
	
	Assume that $p \nmid (m-2)$. If either $x_2\in \Z_p^\times$ or $x_3\in \Z_p^\times$, then by \eqref{symmetrycond:1} we necessarily have $d_1^2x_1^2+d_2^2x_2^2+d_3^2x_3^2\in \Z_p^\times$, and then \eqref{symmetrycond:2} is automatically satisfied. On the other hand, since $\left<d_2^2,d_3^2\right>\cong \left<1,1\right>$, we have 
	\[
	\{(d_1^2x_1^2+d_2^2x_2^2+d_3^2x_3^2) (\Q_p^\times)^2 : x_2 \in \Z_p^\times \text{ or } x_3\in \Z_p^\times\} \supseteq \{(\Q_p^\times)^2, \Delta_p(\Q_p^\times)^2 \}.
	\]
	Thus, $\{Q(x)(\Q_p^\times)^2: \tau_x \in O(M_p/L_p)\} \supseteq \{(\Q_p^\times)^2, \Delta_p(\Q_p^\times)^2 \}$ so that we may conclude that $\spnnorm{L_p+\nu}\supseteq \Z_p^\times (\Q_p^\times)^2$.
	
	Assume that $p \mid (m-2)$. If either $x_2\in \Z_p^\times$ or $x_3\in \Z_p^\times$, then by \eqref{symmetrycond:1} we necessarily have $d_1^2x_1^2+d_2^2x_2^2+d_3^2x_3^2\in \Z_p^\times$. Since $p\mid (m-2)$, we have $p\nmid (m-4)$, and hence \eqref{symmetrycond:2} implies that $d_1x_1+d_2x_2+d_3x_3\equiv 0 \pmod{p}$. Thus, $d_2x_2+d_3x_3\equiv 0 \pmod{p}$, so we have $d_1^2x_1^2+d_2^2x_2^2+d_3^2x_3^2\in 2(d_2x_2)^2(\Q_p^\times)^2=2(\Q_p^\times)^2$.
		
	Otherwise, if $x_2,x_3\in p\Z_p$, then $x_1\in \Z_p^\times$ since $x$ is a primitive vector of $L_p$. Then by \eqref{symmetrycond:1} for $i=1$ we necessarily have $d_1^2x_1^2+d_2^2x_2^2+d_3^2x_3^2 \in p^2 \Z_p^\times$. Again by \eqref{symmetrycond:1} for $i=2,3$ we have $x_2,x_3\in p^2\Z_p$.
	Since $p\nmid (m-4)$, \eqref{symmetrycond:2} implies that $d_1x_1+d_2x_2+d_3x_3\in p^3\Z_p$. This implies $x_1\in p\Z_p$, which is a contradiction. Hence we have proved that 
	\[
	\{Q(x)(\Q_p^\times)^2: \tau_x \in O(M_p/L_p)\} \subseteq \{2(\Q_p^\times)^2 \},
	\]
	so we may conclude that $\spnnorm{L_p+\nu}\subseteq (\Q_p^\times)^2$.\\

	\noindent{\bf Case (2)-(c)} Assume without loss of generality that $p\mid d_1$, $p \mid d_2$, $p\nmid d_3$, and $p\nmid (m-2)$. 
	
	If $x_3\in \Z_p^\times$, then $d_1^2x_1^2+d_2^2x_2^2+d_3^2x_3^2\in\Z_p^\times$, so that both \eqref{symmetrycond:1} and \eqref{symmetrycond:2} always hold. In this case, $Q(x)\in d_3^2x_3^2(\Q_p^\times)^2=(\Q_p^\times)^2$.
	Now we assume $x_3\in p\Z_p$. Then since $d_1^2x_1^2+d_2^2x_2^2+d_3^2x_3^2\in p^2\Z_p$, we have $x_3\in p^2\Z_p$. Since either $x_1\in\Z_p^\times$ or $x_2\in \Z_p^\times$, \eqref{symmetrycond:1} implies that $d_1^2x_1^2+d_2^2x_2^2+d_3^2x_3^2\in p^2\Z_p^\times$.
	Now \eqref{symmetrycond:2} implies that $(m-4)(d_1x_1+d_2x_2+d_3x_3)\in p^2\Z_p$, equivalently, $(m-4)((d_1/p)x_1+(d_2/p)x_2)\in p\Z_p$.
	If $p\mid (m-4)$, then this condition always hold, and since $Q(x)\in ((d_1/p)^2x_1^2+(d_2/p)^2x_2^2) (\Q_p^\times)^2$, we have 
	\[
	\{Q(x) (\Q_p^\times)^2 : x_3 \in p^2\Z_p, \text{ and either } x_1 \in \Z_p^\times \text{ or } x_2\in \Z_p^\times\} \supseteq \{(\Q_p^\times)^2, \Delta_p(\Q_p^\times)^2 \}.
	\]
	Hence $\spnnorm{L_p+\nu}\supseteq \Z_p^\times (\Q_p^\times)^2$ in this case. 
	If $p\nmid (m-4)$, then we necessarily have $(d_1/p)x_1+(d_2/p)x_2\in p\Z_p$ so that $x_1,x_2\in\Z_p^\times$ and $Q(x)/p^2 \in 2(d_1/p)^2x_1^2 (\Q_p^\times)^2=2(\Q_p^\times)^2$. Hence,
	\[
	\{Q(x) (\Q_p^\times)^2: \tau_x \in O(M_p/L_p)\} = \{(\Q_p^\times)^2, 2(\Q_p^\times)^2 \},
	\]
	from which $\spnnorm{L_p+\nu}$ is determined as in the statement according to $\legendre{2}{p}=\pm 1$.\\
	
	\noindent{\bf Case (2)-(d)} Assume without loss of generality that $p\mid d_1$, $p \mid d_2$, $p\nmid d_3$, and $p \mid (m-2)$.
	
	If $x_3\in\Z_p^\times$, then \eqref{symmetrycond:1} for $i=3$ implies that $d_1^2x_1^2+d_2^2x_2^2+d_3^2x_3^2\in \Z_p^\times$, hence  \eqref{symmetrycond:2} implies that $d_1x_1+d_2x_2+d_3x_3\in p\Z_p$.
	So we have $d_3x_3\in p\Z_p$, which is a contradiction.
	Thus, $x_3\in p\Z_p$. This implies that $d_1^2x_1^2+d_2^2x_2^2+d_3^2x_3^2\in p^2\Z_p$ and \eqref{symmetrycond:1} for $i=3$ implies that $x_3\in p^2\Z_p$.
	Since either $x_1\in\Z_p^\times$ or $x_2\in \Z_p^\times$, \eqref{symmetrycond:1} implies that $d_1^2x_1^2+d_2^2x_2^2+d_3^2x_3^2\in p^2\Z_p^\times$.
	Now, \eqref{symmetrycond:2} implies that $d_1x_1+d_2x_2+d_3x_3\in p^3\Z_p$, from which we have $(d_1/p)x_1+(d_2/p)x_2\equiv 0\pmod{p}$.
	Thus, $(d_1^2x_1^2+d_2^2x_2^2+d_3^2x_3^2)/p^2 \in 2(d_1/p)^2x_1^2 (\Q_p^\times)^2=2(\Q_p^\times)^2$, so $\spnnorm{L_p+\nu}\subseteq (\Q_p^\times)^2$.
\end{proof}

\begin{proposition}\label{localspinornorm:prime2}
	Let $\bm{d}\in \mathcal{S}_1^3$ and $m\geq 3$ be odd. Then
	\[
	\spnnorm{X^{\bm{d}}_2} = \Z_2(\Q_2^\times)^2.
	\]
\end{proposition}
\begin{proof}
	For the ease of notations, we denote $X=X^{\bm{d}}$ and $L=L^{\bm{d}}$ throughout the proof of the proposition. 
	Changing the basis for $L_2$ with $f_i = ((m-2)d_i)^{-1}e_i$ ($i=1,2,3$), we have $L_2\cong \left<4,4,4\right>$ and
	\[
	L_2+\nu = L_2 + \frac{4-m}{2}\left(f_1+f_2+f_3\right) = L_2 + \frac{1}{2}\left(f_1+f_2+f_3\right),
	\]
	hence one may show that $O^+(L_2+\nu) = O^+(L_2)$ (see the proof of \cite[Proposition 2.3]{Kim}).
	Thus
	\[
	\spnnorm{L_2+\nu}=\spnnorm{L_2}=\Z_2(\Q_2^\times)^2.
	\]
	This completes the proof of the proposition.
\end{proof}
For $\bm{d}=\bm{1}$, the genus and spinor genus coincide, yielding that \eqref{eqn:sum3mgonal} is solvable for sufficiently large $n$.
\begin{theorem}\label{thm:3mgonaloddalmostuniversal}
	Let $m\ge3$ be an odd integer. Then $\pgen(X^{\bm{1}})=\pspn(X^{\bm{1}})$. In particular, for every sufficiently large integer $n$, the equation
	\[
		p_m(x)+p_m(y)+p_m(z)=n
	\]
	is solvable with $x,y,z\in\Z$.
\end{theorem}
\begin{proof}
	To prove $\pgen(X^{\bm{1}})=\pspn(X^{\bm{1}})$, it suffice to show that 
	\[
	\left[I_\Q : \Q^\times \prod\limits_{p\in \Omega} \spnnorm{X_p^{\bm{1}}}\right]=1
	\]
	as is described in \eqref{eqn:numberofproperspinorgenera}. 	By Propositions \ref{localspinornorm:oddprime} and \ref{localspinornorm:prime2}, we have
	\[
		\spnnorm{X_p^{\bm{1}}} =
		\begin{cases} 
			\Z_2 (\Q_2^\times)^2 & \text{if } p=2,\\
			(\Q_3^\times)^2& \text{if } p=3 \text{ and } 3\mid m-2,\\
			\Z_p^\times (\Q_p^\times)^2 & \text{otherwise}.
		\end{cases}	
	\]
	For any $\iota\in I_\Q$, one may write $\iota=a\cdot j$ for some $a\in \Q^\times$ and $j=(j_p)\in I_\Q$ with $j_p\in\Z_p^\times$, $j_\infty >0$. 
	If $3\nmid m-2$, then $j\in\prod_{p\in \Omega} \spnnorm{X_p^{\bm{1}}}$ and hence we are done. Assume that $3\mid m-2$. If $j_3\in (\Q_3^\times)^2$, then $j_3\in \spnnorm{X_p^{\bm{1}}}$ and hence we are done. If $j_3\in 2(\Q_3^\times)^2$, then $2j\in \prod_{p\in \Omega} \spnnorm{X_p^{\bm{1}}}$ so that $\iota=(a/2)(2j)\in \Q^\times \prod_{p\in \Omega} \spnnorm{X_p^{\bm{1}}}$.
	
To conclude the second statement, note that since $\pgen(X^{\bm{1}})=\pspn(X^{\bm{1}})$, we may write 
\[
\Theta_X= \Theta_{\pgen(X^{\bm{1}})}+\Theta_X-\Theta_{\pgen(X^{\bm{1}})}= \Theta_{\pgen(X^{\bm{1}})}+\Theta_X-\Theta_{\pspn(X^{\bm{1}})}.
\]
By \cite[Theorem 6.2]{KaneKim}, $\Theta_X-\Theta_{\pspn(X^{\bm{1}})}$ is in the subspace of cusp forms which are orthogonal to unary theta functions. To obtain the claim, one then compares the lower bound $r_{\pgen(X^{\bm{1}})}(h)\gg_{m,\varepsilon} h^{\frac{1}{2}-\varepsilon}$ from Lemma \ref{lem:X1lower} with the upper bound of the absolute value of the Fourier coefficients of $\Theta_X-\Theta_{\pspn(X^{\bm{1}})}$ given by Duke \cite{Duke}.
\end{proof}

\subsection{Unary theta function part}
Although we are unsure how many spinor genera lie in the genus for $\bm{d}\in\mathcal{S}_1^3$, we are able to show that the theta functions for the genus and spinor genus coincide.
\begin{proposition}\label{proposition:unarypartforodd:d}
Let $\bm{d}\in\mathcal{S}_1^3$ and $m\geq 3$ be an odd integer. Then
\[
\Theta_{\pgen(X^{\bm{d}})}(z)=\Theta_{\pspn(X^{\bm{d}})}(z).
\]
\end{proposition}
\begin{proof}
For the ease of notation, we denote $X=X^{\bm{d}}$ and $L=L^{\bm{d}}$ throughout the proof of the proposition. Assume to the contrary that $\Theta_{\pspn(L+\nu)}(z)-\Theta_{\pgen(L+\nu)}(z)\neq 0$. Then by \cite[Theorem 5.6]{KaneKim}, its Fourier coefficients are supported on finitely many square classes $t\Z^2$ ($t$ : squarefree) such that some integers in $t\Z^2$ is locally represented by $X$ and satisfying
\begin{equation}\label{condition-splittingintegers}
\spnnorm{X_p} \subseteq N_{E_\mathfrak{p}/\Q_p}(E_\mathfrak{p}^\times) \quad  \forall \ p \ (\mathfrak{p}\mid p),	
\end{equation}
where $E=\Q\left(\sqrt{-td_L}\right)$. Let $t\Z^2$ be a square class satisfying \eqref{condition-splittingintegers}.
Noting that $\spnnorm{X^{\bm{d}}_2}=\Z_2 (\Q_2^\times)^2$ by Proposition \ref{localspinornorm:prime2}, we necessarily have
\[
\left( \frac{\alpha,-t}{2}\right) =1 \quad \text{for all } \alpha\in\Z_2,
\]
where $\left(\frac{-,-}{2}\right)$ denotes the Hilbert symbol at $2$. This implies in our case that $-t\in (Q_2^\times)^2$. Thus we have $t\equiv 7 \pmod{8}$.
However, no integers of the form $tx^2$ is $2$-adically represented by $X$ since for any $x_j\in\Z_2$,
\[
Q(x_1(d_1e_1)+x_2(d_2e_2)+x_3(d_3e_3)+\nu)=\sum_{j=1}^3 (2(m-2)d_jx_j+4-m)^2\equiv 3 \pmod{8}.
\]
This is a contradiction, and hence $\Theta_{\pgen(X)}(z)=\Theta_{\pspn(X)}(z)$. 
\end{proof}
We next consider those $\bm{d}$ with $d_j$ not necessarily odd.
%
%
%


\begin{proposition}\label{proposition:unarypartford:general}
	Let $\bm{d'}\in (\cup_{j=0}^{\infty}\mathcal{S}_{2^j})^3$ and write $\bm{d'}=\bm{d}\cdot (2^{a_1},2^{a_2},2^{a_3})$ with $\bm{d}\in\mathcal{S}_1^3$ and $\bm{a}=(a_1,a_2,a_3)\in\Z_{\ge 0}^3$. Assuming that $r_{\pgen(X^{\bm{d}})}(h)>0$, define the numbers $c_{\bm{a},\bm{d}}(h)$ by
	\[
	r_{\pspn(X^{\bm{d'}})}(h)-r_{\pgen(X^{\bm{d'}})}(h)=c_{\bm{a},\bm{d}}(h)\cdot r_{\pgen(X^{\bm{d}})}(h).
	\]
	Then we have a bound $|c_{\bm{a},\bm{d}}(h)| \le 2^{-a_1-a_2-a_3+\min_j\{a_j\}}$. Moreover, $c_{\bm{a},\bm{d}}(h)\neq 0$ only if the following are satisfied for each prime $p$ dividing the squrefree part $\text{sf}(h)$ of $h$.
	\begin{enumerate}[label={\rm (\arabic*)}]
		\item $p\nmid 3(m-2)(m-4)$, $\#\{j : p \mid d_j\} = 2$, and $\legendre{2}{p}=1$.
		\item Otherwise, $p=3$, and either $3\mid (m-2)$ or $\#\{j : 3 \mid d_j\} = 3$.
	\end{enumerate}
\end{proposition}
\begin{proof}
	Note that in general one has
	\[
	0\le r_{\pspn(X)}(h) \le 2\cdot r_{\pgen(X)}(h)
	\]
	for any ternary coset $X$ so that we have
	\[
	|c_{\bm{a},\bm{d}}(h)|\le \frac{r_{\pgen(X^{\bm{d'}})}(h)}{r_{\pgen(X^{\bm{d}})}(h)}.
	\]
	On the other hand, \eqref{eqn:Siegelformula} and Lemma \ref{lemma:localdensity:at2} implies
	\begin{equation}\label{eqn:rgenratio}
		\frac{r_{\pgen(X^{\bm{d'}})}(h)}{r_{\pgen(X^{\bm{d}})}(h)} 
		=\frac{d_{L^{\bm{d'}}}^{-\frac{1}{2}}\cdot b_2(h,\lambda_{\bm{d'}},0)}{d_{L^{\bm{d}}}^{-\frac{1}{2}}\cdot b_2(h,\lambda_{\bm{d}},0)}
		=\frac{1}{2^{a_1+a_2+a_3}}\cdot \frac{2^{2+\min_j\{a_j\}}}{2^{2+0}}=2^{-(a_1+a_2+a_3)+\min_j \{a_j\}}.
	\end{equation}
	This proves the first part of the proposition.
	
	To prove the second part, assume that $c_{\bm{a},\bm{d}}(h) \neq 0$. We note by \cite[Theorem 5.6]{KaneKim} that $c_{\bm{a},\bm{d}}(h) \neq 0$ only if some integer in $\text{sf}(h)\Z^2$ is locally represented by $X^{\bm{d'}}$ and satisfying \eqref{condition-splittingintegers} with $E=\Q\left(\sqrt{-hd_{L^{\bm{d'}}}}\right)=\Q\left(\sqrt{-\text{sf}(h)}\right)$.
	Noting that $h$ is an odd integer and 
	\begin{center}
		$N_{E_\mathfrak{p}/\Q_p}(E_\mathfrak{p}^\times)=\Q_p^\times$ for any $p\nmid \text{sf}(h)$, and $\Z_p^\times \not\subseteq N_{E_\mathfrak{p}/\Q_p}(E_\mathfrak{p}^\times)$ for any $p \mid \text{sf}(h)$
	\end{center} 
	for $p\neq 2$, we have $c_{\bm{a},\bm{d}}(h)\neq 0$ only if $\spnnorm{X^{\bm{d'}}_p}=(\Q_p^\times)^2$ for all $p\mid \text{sf}(h)$.
	Noting further that $\spnnorm{X^{\bm{d'}}_p}=\spnnorm{X^{\bm{d}}_p}$ for any odd prime numbers $p$, it follows from Proposition \ref{localspinornorm:oddprime} that $\bm{d}$ and $m$ should satisfy the following for each prime divisor $p\mid \text{sf}(h)$:
	\begin{enumerate}
		\item[(a)] $\#\{j : p \mid d_j\} = 2$, $p\nmid (m-4)$, $\legendre{2}{p}=1$ if $p\neq3$ and $p\nmid (m-2)$.
		\item[(b)] $\#\{j : p \mid d_j\} = 1 \text{ or } 2$ if $p\neq3$ and $p\mid (m-2)$.
		\item[(c)] If $3\mid \text{sf}(h)$, then either
		\begin{itemize}
			\item $3\mid (m-2)$ and $\#\{j : 3 \mid d_j\} = 0,1,2, \text{ or }3$, or
			\item $3\nmid (m-2)$, and either $\#\{j : 3 \mid d_j\} = 3$ or $\#\{j : 3 \mid d_j\} = 2$ with $3\nmid m-4$, $\legendre{2}{3}=1$.
		\end{itemize}
	\end{enumerate}
	The condition (b) cannot happen since $p\mid h$ and $p\mid (m-2)$ implies that $p \mid 3(m-4)^2(=h-8(m-2)n)$, which yields $p=3$ since $\gcd(m-2,m-4)=1$. The condition (c) may be summarized as the condition (2) in the statement of the proposition.
\end{proof}

\section{Error from the cuspidal part}\label{sec:cuspform}
In this section, we obtain a bound on the contribution from the cuspidal part that is orthogonal to unary theta functions. For $k\in \Z+\frac{1}{2}$, let $S_{k}^{\perp}(\Gamma,\chi)$ denote the subspace of $S_{k}(\Gamma,\chi)$ which is orthogonal to unary theta functions. 

We follow the proof from \cite{BanerjeeKane}, with the results from Schulze-Pillot and Yenirce \cite{SPY} replaced with bounds from the appendix of \cite{BHM}.
\begin{lemma}\label{lem:cf(n)<norm}
 Suppose that $k\in\Z+\frac{1}{2}$ and $f\in S_{k}^{\perp}(\Gamma_0(N)\cap\Gamma_1(L),\psi)$ with $L\mid N$ and $\psi$ a character modulo $N$. If $f$ has the Fourier expansion $f(\tau)=\sum_{n\geq 1} c_f(n) q^n,$ then 
\begin{equation}\label{eqn:afbnd}
\left|c_f(n)\right|\ll_{k,\varepsilon} \|f\| N^{\frac{25}{16}+\varepsilon} n^{\frac{k-1}{2} + \frac{103}{512}+\varepsilon}\varphi(L)^{\frac{1}{2}}.
\end{equation}

In particular, for $k=\frac{3}{2}$,we have 
\begin{equation}\label{eqn:k=3/2spec}
\left|c_{f}(n)\right|\ll_{\varepsilon} \|f\| N^{\frac{25}{16}+\varepsilon} n^{\frac{231}{512}+\varepsilon}\varphi(L)^{\frac{1}{2}}.
\end{equation}
\end{lemma}
\begin{proof}
By \cite[Theorem 6]{BHM}, if $g\in S_k^{\perp}\left(\Gamma_0(N),\chi\right)$ is normalized with respect to the Petersson inner product, then the $n$-th coefficient $c_g(n)$ of its Fourier expansion satisfies
\[
\left|c_{g}(n)\right|\ll_{k} N^{\frac{25}{16}+\varepsilon} n^{\frac{k-1}{2}+\frac{1}{4}-\frac{1}{16}+\frac{\vartheta}{8}+\varepsilon},
\]
where $\vartheta\geq 0$ is any admissible constant, with $\vartheta=\frac{7}{64}$ known to hold. Hence 
\begin{equation}\label{eqn:cgnbound}
\left|c_{g}(n)\right|\ll_{k,\varepsilon} N^{\frac{25}{16}+\varepsilon} n^{\frac{k-1}{2}+\frac{103}{512}+\varepsilon},
\end{equation}
Since (see \cite[Theorem 2.5]{Cho} for the integral-weight case, which generalizes directly) 
\begin{equation}\label{eqn:chisplitting}
S_{k}^{\perp}(\Gamma_0(N)\cap\Gamma_1(L),\psi)=\bigoplus_{\chi\pmod{L}} S_{k}^{\perp}(\Gamma_0(N),\psi\chi),
\end{equation}
 we have
\[
f=\sum_{\chi\pmod{L}} f_{\chi}
\]
with $f_{\chi}\in S_{k}^{\perp}(\Gamma_0(N),\psi\chi)$. Note that since the splitting in \eqref{eqn:chisplitting} is orthogonal with respect to the Petersson inner product, we have
\begin{equation}\label{eqn:normfsplit}
\left\|f\right\|^2=\sum_{\chi\pmod{L}} \left\|f_{\chi}\right\|^2.
\end{equation}
Since $\frac{f_{\chi}}{\|f_{\chi}\|}$ has norm $1$ under the Petersson inner product, \eqref{eqn:cgnbound} with $g=\frac{f_{\chi}}{\|f_{\chi}\|}$ implies that 
\begin{equation}\label{eqn:cfchibound}
\left|c_{\frac{f_{\chi}}{\|f_{\chi}\|}}(n)\right|\ll_{k,\varepsilon} N^{\frac{25}{16}+\varepsilon} n^{\frac{k-1}{2} + \frac{103}{512}+\varepsilon}.
\end{equation}
Using the triangle inequality and Cauchy--Schwartz inequality and combining with \eqref{eqn:normfsplit} and \eqref{eqn:cfchibound} yields the claim.\qedhere
\begin{extradetails}
We conclude that 
\begin{align*}
\left|c_f(n)\right|&\leq \sum_{\chi\pmod{L}} |c_{f_{\chi}}(n)|=\sum_{\chi\pmod{L}}\left\|f_{\chi}\right\|  \left|c_{\frac{f_{\chi}}{\|f_{\chi}\|}}(n)\right|\\
&\leq \left(\sum_{\chi\pmod{L}}\left\|f_{\chi}\right\|^2\right)^{\frac{1}{2}}\left(\sum_{\chi\pmod{L}}\left|c_{\frac{f_{\chi}}{\|f_{\chi}\|}}(n)\right|^2\right)^{\frac{1}{2}}\\
&\overset{\eqref{eqn:normfsplit}}{=} \|f\|\left(\sum_{\chi\pmod{L}}\left|c_{\frac{f_{\chi}}{\|f_{\chi}\|}}(n)\right|^2\right)^{\frac{1}{2}}\overset{\eqref{eqn:cfchibound}}{\ll_{k,\varepsilon}} \|f\| N^{\frac{25}{16}+\varepsilon} n^{\frac{k-1}{2} + \frac{103}{512}+\varepsilon}\varphi(L)^{\frac{1}{2}}.
\end{align*}
\end{extradetails}
\end{proof}

We next bound $\|f\|$ for $f$ equal to the projection of $\Theta_{X^{\bm{d}}}$ to $S_{\frac{3}{2}}^{\perp}\left(\Gamma_{0}(N_{m,\bm{d}})\cap\Gamma_1(M_{m,\bm{d}})\right)$, where 
\begin{align*}
N_{m,\bm{d}}&:= 16(m-2)^2\lcm(d_1,d_2,d_3)^2,\\
M_{m,\bm{d}}&:=\frac{2(m-2)\lcm\left(d_1,d_2,d_3\right)}{\gcd\left(m-4,\lcm(d_1,d_2,d_3)\right)}. 
\end{align*}
\begin{lemma}\label{lem:OrthogonalBound}
Suppose that $f$ is the projection of $\Theta_{X^{\bm{d}}}$ to 
\[
S_{\frac{3}{2}}^{\perp}\left(\Gamma_{0}(N_{m,\bm{d}})\cap\Gamma_1(M_{m,\bm{d}})\right).
\]
Then 
\[
\|f\|^2\ll N_{m,\bm{d}}^2 M_{m,\bm{d}}^{4+\varepsilon}.
\]
and 
\[
\left|c_{f}(n)\right|\ll_{\varepsilon}  N_{m,\bm{d}}^{\frac{41}{16}+\varepsilon} M_{m,\bm{d}}^{\frac{5}{2}+\varepsilon} n^{\frac{231}{512}+\varepsilon}.
\]
\end{lemma}
\begin{proof}
Using an argument of Blomer \cite{Blomer}, we follow the proof of \cite[Lemma 3.2]{KKT}, modifying for $\ell=3$ odd and choosing $N=N_{m,\bm{d}}$ and $M=M_{m,\bm{d}}$.

We may follow the proof mutatis mutandis until \cite[(3.18)]{KKT}. Combining \cite[(3.12)]{KKT} with \cite[(3.18)]{KKT}, we have (note that, in the notation from \cite{KKT}, $\Delta_Q=8$ is fixed for $Q\cong \left<1,1,1\right>$) 
\begin{equation}\label{eqn:normf111}
\|f\|^2 \ll \frac{N}{\mu(N)}M^{-6}\sum_{j=1}^{\mu(N)/N}\prod_{d=1}^3\gcd\left(M^2b_d,c_j\right) \sum_{n=1}^{\infty} n^3 \left(\frac{N}{4\pi n}\right)^{\frac{1}{2}}\Gamma\left(\frac{1}{2},\frac{2\pi n}{N}\right).
\end{equation}
where $\mu(N):=\left[\SL_2(\Z):\Gamma(N)\right]$. Here $\prod_{d=1}^3 b_d=\Delta_Q$ and $c_j$ is the denominator of the $j$-th cusp with $\gcd(c_j,M^2N_Q)=\delta$.

 Since $\Gamma(\frac{1}{2},z^2)=\sqrt{\pi}\erfc(z)$ by \cite[8.4.6]{NIST} and \cite[7.12.1]{NIST} (together with $\erfc(z)\sim 1$ for $z\to 0$), the inner sum may be bounded by 
\begin{multline*}
\sum_{n=1}^{\infty} n^3 \left(\frac{N}{4\pi n}\right)^{\frac{1}{2}}\Gamma\left(\frac{1}{2},\frac{2\pi n}{N}\right)=\sqrt{\pi}\sum_{n=1}^{\infty} n^3 \left(\frac{N}{4\pi n}\right)^{\frac{1}{2}}\erfc\left(\sqrt{\frac{2\pi n}{N}}\right)\\
\ll\sum_{n=1}^{N} n^3 \left(\frac{N}{4\pi n}\right)^{\frac{1}{2}}+ \sum_{n=N}^{\infty} n^3 \left(\frac{N}{4\pi n}\right)^{\frac{1}{2}}\left(\left(\frac{N}{2\pi n}\right)^{\frac{1}{2}}e^{-\frac{2\pi n}{N}}+O\left(\left(\frac{N}{n}\right)^{\frac{3}{2}}e^{-\frac{2\pi n}{N}}\right)\right)\\
\ll N^4+  N\sum_{n=N}^{\infty} n^2e^{-\frac{2\pi n}{N}}+O\left( N^2\sum_{n=N}^{\infty} ne^{-\frac{2\pi n}{N}}\right).
\end{multline*}
Using \cite[Lemma 2.6]{BanerjeeKane}, we obtain
\[
\sum_{n=1}^{\infty} n^3 \left(\frac{N}{4\pi n}\right)^{\frac{1}{2}}\Gamma\left(\frac{1}{2},\frac{2\pi n}{N}\right)\ll N^4
\]
Plugging this into \eqref{eqn:normf111} yields 
\[
\|f\|^2 \ll \frac{N^5}{\mu(N)}M^{-6}\sum_{j=1}^{\mu(N)/N}\prod_{d=1}^3\gcd\left(M^2b_d,c_j\right).
\]
Using $\delta=\gcd(M^2N_Q,c_j)$, we then bound 
\[
\gcd\left(M^2 b_d,c_j\right)\leq b_d\gcd\left(M^2,c_j\right)= b_d \gcd\left(M^2,\delta\right). 
\]
Hence, using $\prod_{d=1}^3 b_d=\Delta_Q=8$, we have 
\begin{multline*}
\|f\|^2 \ll \frac{N^5}{\mu(N)}\frac{\Delta_Q}{M^6}  \sum_{\substack{1\leq j\leq \mu(N)/N\\ \gcd(M^2N_Q,c_j)=\delta}}\gcd\left(M^2,\delta\right)^3\ll \frac{N^5}{\mu(N)} \sum_{\delta\mid M^2N_Q} \sum_{\substack{1\leq j\leq \mu(N)/N\\ \gcd(M^2N_Q,c_j)=\delta}}\left(\frac{\gcd\left(M^2,\delta\right)}{M^2}\right)^3\\
=\frac{N^5}{\mu(N)} \sum_{\delta\mid M^2N_Q} \varphi\left(\frac{M^2N_Q}{\delta}\right)\varphi(\delta)\frac{M^2N_Q}{\delta}\left(\frac{\gcd\left(M^2,\delta\right)}{M^2}\right)^3
\end{multline*}
where in the last step we used the fact that the number of cusps with $\gcd(M^2N_q,c_j)=\delta$ is (for example, see \cite[p. 12, displayed formula after (4.4)]{Blomer})
\[
\varphi\left(\frac{M^2N_Q}{\delta}\right)\varphi(\delta)\frac{M^2N_Q}{\delta}.
\]
We now recall that (for example, see \cite[(2.2)]{Blomer})
\[
\mu(N)=N^3\prod_{p\mid N} \left(1-\frac{1}{p^2}\right),
\]
yielding 
\begin{multline*}
\|f\|^2\ll \frac{N^2}{\prod_{p\mid N}\left(1-\frac{1}{p^2}\right)} \sum_{\delta\mid M^2N_Q} \varphi\left(\frac{M^2N_Q}{\delta}\right)\varphi(\delta)\frac{M^2N_Q}{\delta}\left(\frac{\gcd\left(M^2,\delta\right)}{M^2}\right)^3\\
\ll N^2\sum_{\delta\mid M^2N_Q} \varphi\left(\frac{M^2N_Q}{\delta}\right)\varphi(\delta)\frac{M^2N_Q}{\delta}\left(\frac{\gcd\left(M^2,\delta\right)}{M^2}\right)^3,
\end{multline*}
bounding against $\zeta(2)$ in the last step. Since $\frac{\gcd(M^2,\delta)}{M^2}\leq 1$, we may then bound 
\begin{multline*}
\sum_{\delta\mid M^2N_Q} \varphi\left(\frac{M^2N_Q}{\delta}\right)\varphi(\delta)\frac{M^2N_Q}{\delta}\left(\frac{\gcd\left(M^2,\delta\right)}{M^2}\right)^3\leq \sum_{\delta\mid M^2N_Q} \varphi\left(\frac{M^2N_Q}{\delta}\right)\varphi(\delta)\frac{M^2N_Q}{\delta}\\
\leq M^2N_Q\varphi\left(M^2N_Q\right)\sigma_{-1}\left(M^2N_Q\right)\ll_{\varepsilon} M^{4+\varepsilon}
\end{multline*}
This gives the first statement. Plugging this into \eqref{eqn:k=3/2spec} with $L=M$ thus yields
\[
\left|c_{f}(n)\right|\ll_{\varepsilon}  N^{\frac{41}{16}+\varepsilon} n^{\frac{231}{512}+\varepsilon} M^{\frac{5}{2}+\varepsilon}
\]
This is the second statement.
\end{proof}

\section{Main term and error term from sieving}\label{sec:mainerror}

For $n\in\N$ and $\bm{\ell}\in\N^3$, let $\mathcal{A}_{\bm{\ell}}$ be the set of solutions $\bm{x}\in\Z^3$ with $\ell_j\mid x_j$ to 
\[
\sum_{j=1}^3p_m\left(x_j\right)=n.
\]
As noted in \eqref{eqn:sum3squaresCongruence}, the set $\mathcal{A}_{\bm{\ell}}$ is in one-to-one correspondence with solutions to 
\[
X_1^2+X_2^2+X_3^2=h,
\]
with $X_j\equiv 4-m\pmod{2(m-2)\ell_j}$ and $h=8(m-2)n+3(m-4)^2$. Let $\mathcal{P}$ be a finite set of odd primes and for $\bm{\ell}\in\N^3$ and $n\in\N$ let $S_{h}(\mathcal{A}_{\bm{\ell}},\mathcal{P})$ denote the size of the set 
\begin{equation}\label{eqn:ShPdef}
\mathscr{S}_{h}(\mathcal{A}_{\bm{\ell}},\mathcal{P}):=\left\{\bm{x}\in\Z^3: \sum_{j=1}^3 p_m(\ell_j x_j)=n,\ p\mid x_j\implies p\notin \mathcal{P}\right\}.
\end{equation}
 Recall that $\mathcal{S}_{2^a}$ is the set of $d\in \N$ with either $d$ odd and squarefree or $d=2^ad'$ with $d'$ odd and squarefree. We assume throughout that $a\geq 3$. We then define
\begin{equation}\label{eqn:SahPdef}
S_{a,h}(\mathcal{A}_{\bm{\ell}},\mathcal{P}):=\sum_{\bm{\alpha}\in\{0,1\}^3}\mu(2^{\alpha_1})\mu(2^{\alpha_2})\mu(2^{\alpha_3})S_{h}\left(\mathcal{A}_{2^{a\bm{\alpha}}\bm{\ell}},\mathcal{P}\right),
\end{equation}
which is the size of the set 
\[
\mathscr{S}_{a,h}\left(\mathcal{A}_{\bm{\ell}},\mathcal{P}\right):=\left\{\bm{x}\in\Z^3: \sum_{j=1}^3 p_m(\ell_j x_j)=n,\ p\mid x_j\implies p\notin \mathcal{P}\cup\{2\}\text{ or }(p=2\text{ and }\ord_p<a)\right\}.
\]
When $\mathcal{P}=\mathcal{P}_{h,z}:=\{p\leq z\}\cup \{p\mid \sf(h)\}$, we simply abbreviate by $\mathscr{S}_{h}(\mathcal{A}_{\bm{\ell}},z):=\mathscr{S}_{h}(\mathcal{A}_{\bm{\ell}},\mathcal{P}_{h,z})$ (and $S_{h}(\mathcal{A}_{\bm{\ell}},z):=S_h(\mathcal{A}_{\ell},\mathcal{P}_{h,z})$, etc.).

For $\beta,D>0$ and $d\in\mathcal{S}_{1}$ of the form $d=p_1p_2\cdots p_{r}$ with $p_1>p_2>\dots>p_r$ with $p_j$ an odd prime, we next define the \begin{it}Rosser weights\end{it} $\lambda_{d}^{\pm}=\lambda_{d,D}^{\pm}(\beta)$. Setting 
\[
y_m=y_{m}(D,\beta):=\left(\frac{D}{p_1\cdots p_m}\right)^{\frac{1}{\beta}},
\]
these are defined by
\begin{align*}
\lambda_d^+=\lambda_{d,D}^+(\beta)&:=\begin{cases} (-1)^r&\text{if }p_{2\ell+1}<y_{2\ell+1}(D,\beta)\quad \forall 0\leq \ell\leq \frac{r-1}{2},\\
0&\text{otherwise},
\end{cases}\\
\lambda_d^-=\lambda_{d,D}^-(\beta)&:=\begin{cases} (-1)^r&\text{if }p_{2\ell}<y_{2\ell}(D,\beta)\quad \forall 0\leq \ell\leq \frac{r}{2},\\
0&\text{otherwise}.
\end{cases}
\end{align*}
As is standard, we consider $D$ and $\beta$ to be fixed throughout and omit these in the notation. Treating $2^a$ as a prime, we extend the above definition for $d\in\mathcal{S}_{2^a}$ by modifying $\lambda_{2^{ab}d'}^{\pm} = \mu(2^{b})\lambda_{d'}^{\pm}$ for $2\nmid d'$. We furthermore define 
\[
\Lambda_{d}^-:=3\lambda_{d}^--2\lambda_{d}^+.
\]
For $z_0\geq 3$, write 
\[
P_{\sf(h)}(z_0):=\prod_{p\in \mathcal{P}_{h,z_0}} p=\prod_{3\leq p\leq z_0}p \prod_{\substack{p\mid \sf(h)\\ p>z_0}} p.
\]
Define $W_{a,h}(z_0):=\left(1-\frac{3}{2^a}+\frac{3}{2^{2a}}-\frac{1}{2^{2a}}\right)\prod_{p\mid P_{\sf(h)}(z_0)} \left(1-\beta_{X^{(1,1,p)},p}(h)\right)$, where $\beta_{X^{\bm{d}},p}(h)$ was defined in \eqref{eqn:defn_beta_p}, and 
\[
H(n):=\prod_{p\mid n}\left(1+ \left(\frac{1}{p(p-1)}\right)^{\frac{1}{3}}\right).
\]
We also require
\begin{multline}\label{eqn:SahETz0}
S_{a,h}^{\operatorname{ET}}\left(z_0\right):= \frac{3\cdot 2^{a}+4}{2^{2a}} \left(\sum_{\bm{c}\in \{0,1\}^3} \beta_{X^{\bm{3^{c}}},3}(h)\right)^{\delta_{\text{\tiny$\substack{3\mid\text{sf}(h),\\3\mid(m-2)}$}}}
\left(\sum_{\bm{c}=(1,1,1)} \beta_{X^{\bm{3^{c}}},3}(h)\right)^{\delta_{\text{\tiny$\substack{3\mid\text{sf}(h),\\3\nmid(m-2)}$}}}\\
\times \prod_{\substack{p\mid \text{sf}(h)\\ p\neq2,\, p\nmid 3(m-2)(m-4)}} \left(\sum_{\#\{j:c_j=1\}=2} \beta_{X^{\bm{p^{c}}},p}(h)\right)
 \prod_{\substack{p\leq z_0,\, p\nmid \text{sf}(h)\\ p\neq 2}} \left(\sum_{\bm{c}\in \{0,1\}^3} \beta_{X^{\bm{p^{c}}},p}(h)\right).
\end{multline}
\begin{proposition}\label{prop:SahMainz0}
Let $D_0$ and $z_0$ be given, and write $s_0:=\frac{\log(D_0)}{\log(z_0)}$. Suppose that all prime factors of $\bm{\ell}$ are greater than $z_0$ and do not divide $\sf(h)$. Then, for some $B,C>0$
\begin{multline*}
S_{a,h}\left(\mathcal{A}_{\bm{\ell}},z_0\right)=\bigg(W_{a,h}(z_0)+O_{\varepsilon}\Big(H(n)^{4}\Delta^{-\frac{2}{3}}\log(D_0)^6\log(\log(\sf(h)))^6 \\
+\Delta^{\varepsilon-B}\log(z_0)^2\log(\log(\sf(h)))^2+\Delta^Ce^{-s_0}\Big)\bigg)r_{\pgen(X^{\bm{\ell}})}(h)\\
+O\left(S_{a,h}^{\operatorname{ET}}\left(z_0\right)r_{\pgen(X^{\bm{\ell}})}(h)\right)+O\Bigg(\sum_{\substack{\bm{d}\in\mathcal{S}_{2^a}^3\\ p\mid d_j\implies p\leq z_0\\ \text{or }p\mid \sf(h)\\ d_j\leq D_0}}\left|r_{X^{\bm{d}\bm{\ell}}}(h)-r_{\pspn(X^{\bm{d}\bm{\ell}})}(h)\right|\Bigg).
\end{multline*}
\end{proposition}

For $\varepsilon\in\{\pm 1\}$, we set
\[
M_{h}^{\varepsilon}(z_0):=\begin{cases}
\sum_{d_1\mid P_{\sf(h)}(z_0)}\Lambda_{d_1}^-\sum_{d_2\mid P_{\sf(h)}(z_0)}\lambda_{d_2}^+\sum_{d_3\mid P_{\sf(h)}(z_0)}\lambda_{d_3}^+\prod_{p\mid \bm{d}} \beta_{X^{\bm{d}},p}(h)&\text{if }\varepsilon=-1,\\
\sum_{d_1\mid P_{\sf(h)}(z_0)}\lambda_{d_1}^+\sum_{d_2\mid P_{\sf(h)}(z_0)}\lambda_{d_2}^+\sum_{d_3\mid P_{\sf(h)}(z_0)}\lambda_{d_3}^+\prod_{p\mid \bm{d}} \beta_{X^{\bm{d}},p}(h)&\text{if }\varepsilon=1,
\end{cases}
\]
and next compute upper and lower bounds for $S_{a,h}(\mathcal{A}_{\bm{\ell}},z_0)$. 

\begin{proposition}\label{prop:SahBound}
Suppose that all prime divisors of $\bm{\ell}$ are greater than $z_0$ and do not divide $\sf(h)$.
\noindent

\noindent
\begin{enumerate}[leftmargin=*,label={\rm (\arabic*)}]
\item 
We have 
\begin{multline*}
S_{a,h}(\mathcal{A}_{\bm{\ell}},z_0)\geq  \left(1-\frac{3}{2^a}+\frac{3}{2^{2a}}-\frac{1}{2^{2a}}\right)M_{h}^{-}(z_0)r_{\pgen(X^{\bm{\ell}})}(h)-5 S_{a,h}^{\operatorname{ET}}\left(z_0\right)r_{\pgen(X^{\bm{\ell}})}(h)\\
-5\hspace{-.5cm}\sum_{\substack{\bm{d}\in\mathcal{S}_{2^a}^3\\ p\mid d_j\implies p\in\mathcal{P}_{h,z_0}\cup \{2\}\\ d_j\leq D_0}}\hspace{-.5cm}\left|r_{X^{\bm{d}\bm{\ell}}}(h)-r_{\pspn(X^{\bm{d}\bm{\ell}})}(h)\right|.
\end{multline*}
\item 
We have 
\begin{multline*}
S_{a,h}(\mathcal{A}_{\bm{\ell}},z_0)\leq  \left(1-\frac{3}{2^a}+\frac{3}{2^{2a}}-\frac{1}{2^{2a}}\right)M_{h}^{+}(z_0)r_{\pgen(X^{\bm{\ell}})}(h)+ 5 S_{a,h}^{\operatorname{ET}}\left(z_0\right)r_{\pgen(X^{\bm{\ell}})}(h)\\
+5\hspace{-.5cm}\sum_{\substack{\bm{d}\in\mathcal{S}_{2^a}^3\\ p\mid d_j\implies p\in\mathcal{P}_{h,z_0}\cup \{2\}\\ d_j\leq D_0}}\hspace{-.5cm}\left|r_{X^{\bm{d}\bm{\ell}}}(h)-r_{\pspn(X^{\bm{d}\bm{\ell}})}(h)\right|.
\end{multline*}
\end{enumerate}
\end{proposition}
\begin{remark}
The last sum is the coefficient of a cusp form orthogonal to unary theta functions; this can be bounded using usual techniques and is an error term for $z_0$ and $\ell_j$ sufficiently small (after modifying $\mu$ to $\widetilde{\mu}$).
\end{remark}
\begin{proof}
(1) Since (see \cite[Lemma 10]{BrudernFouvry}) 
\begin{equation}\label{eqn:rosserweights}
\sum_{d\mid c}\lambda_{d}^-\leq \sum_{d\mid c} \mu(d)\leq \sum_{d\mid c} \lambda_{d}^+
\end{equation}
and $\sum_{d\mid c}\mu(d)\in \{0,1\}$, \cite[Lemma 4.2]{BlomerBrudern} implies that 
\begin{equation}\label{eqn:RosserWeightsProduct}
\prod_{j=1}^3 \sum_{d_j\mid c_j}\mu\left(d_j\right)\geq \sum_{k=1}^3 \sum_{d_k\mid c_k}\lambda_{d_k}^-\prod_{\substack{1\leq j\leq 3\\ j\neq k}} \sum_{d_j\mid c_j}\lambda_{d_j}^+ -2\prod_{j=1}^3\sum_{d_j\mid c_j}\lambda_{d_j}^+.
\end{equation}
Suppose that $\mathcal{P}_1$ and $\mathcal{P}_2$ are two disjoint sets of primes with $p\nmid \ell_j$ for every $p\in\mathcal{P}_2$ and $1\leq j\leq 3$. Writing $\mathcal{R}:=\prod_{p\in \mathcal{P}_2}p$, we compute
\begin{align}\label{eqn:Sahlower}
S_{a,h}\left(\mathcal{A}_{\bm{\ell}},\mathcal{P}_1\cup\mathcal{P}_2\right)&= \sum_{\bm{x}\in \mathscr{S}_{a,h}\left(\mathcal{A}_{\bm{\ell}},\mathcal{P}_1\right)}\prod_{j=1}^3\sum_{d_j\mid \gcd(x_j,\mathcal{R})} \mu\left(d_j\right)\\
\nonumber &= \sum_{\bm{x}\in \mathscr{S}_{a,h}\left(\mathcal{A}_{\bm{\ell}},\mathcal{P}_1\right)}\sum_{\substack{d_1\mid \mathcal{R}\\ d_1\mid x_1}}\mu(d_1)\sum_{\substack{d_2\mid \mathcal{R}\\ d_2\mid x_2}}\mu(d_2)\sum_{\substack{d_3\mid \mathcal{R}\\ d_3\mid x_3}}\mu(d_3)\\
\nonumber &\overset{\eqref{eqn:RosserWeightsProduct}}\geq  \sum_{\bm{x}\in \mathscr{S}_{a,h}\left(\mathcal{A}_{\bm{\ell}},\mathcal{P}_1\right)}\sum_{\substack{d_1\mid \mathcal{R}\\ d_1\mid x_1}}\Lambda_{d_1}^-\sum_{\substack{d_2\mid \mathcal{R}\\ d_2\mid x_2}}\lambda_{d_2}^+\sum_{\substack{d_3\mid \mathcal{R}\\ d_3\mid x_3}}\lambda_{d_3}^+\\
\nonumber &=\sum_{d_1\mid \mathcal{R}}\Lambda_{d_1}^-\sum_{d_2\mid \mathcal{R}}\lambda_{d_2}^+\sum_{d_3\mid \mathcal{R}}\lambda_{d_3}^+S_{a,h}\left(\mathcal{A}_{\bm{d}\cdot \bm{\ell}},\mathcal{P}_1\right).
\end{align}
We assume that $p\mid \ell_j\implies p>z_0$ and $p\nmid \sf(h)$ and take $\mathcal{P}_1=\emptyset$ and $\mathcal{P}_2=\mathcal{P}_{h,z_0}$, so that by \eqref{eqn:SahPdef} we have 
\[
S_{a,h}\left(\mathcal{A}_{\bm{d}\cdot \bm{\ell}},\mathcal{P}_1\right)= \sum_{\bm{\alpha}\in\{0,1\}^3} \mu(2^{\alpha_1})\mu(2^{\alpha_2})\mu(2^{\alpha_3})r_{X^{2^{a\bm{\alpha}}\bm{d}\bm{\ell}}}(h).
\]
 We decompose 
\begin{equation}\label{eqn:rshiftdecomp}
r_{X^{\bm{d}\bm{\ell}}}(h)=r_{\pgen(X^{\bm{d}\bm{\ell}})}(h)+\left(r_{\pspn(X^{\bm{d}\bm{\ell}})}(h)-r_{\pgen(X^{\bm{d}\bm{\ell}})}(h)\right)+\left(r_{X^{\bm{d}\bm{\ell}}}(h)-r_{\pspn(X^{\bm{d}\bm{\ell}})}(h)\right).
\end{equation}
Plugging in the lower bound \eqref{eqn:Sahlower} and then decomposing with \eqref{eqn:rshiftdecomp}  gives for $\bm{\ell}$ with $p\mid \ell_j\implies p>z_0$ and $p\nmid \sf(h)$

\vspace{-1cm}
\begin{equation}\label{eqn:sievedecomp}
S_{a,h}\left(\mathcal{A}_{\bm{\ell}},z_0\right)\geq \hspace{-.3cm}\sum_{\substack{\bm{d}\in\mathcal{S}_{2^a}^3\\ p\mid d_j\implies p\in \mathcal{P}_{h,z_0}}}\hspace{-.3cm}\Lambda_{d_1}^-\lambda_{d_2}^+\lambda_{d_3}^+\begin{array}{ll}\\  \Big(r_{\pgen(X^{\bm{d}\bm{\ell}})}(h) &+ \left(r_{\pspn(X^{\bm{d}\bm{\ell}})}(h)-r_{\pgen(X^{\bm{d}\bm{\ell}})}(h)\right)\\
& \quad\ + \left(r_{X^{\bm{d}\bm{\ell}}}(h)-r_{\pspn(X^{\bm{d}\bm{\ell}})}(h)\right)\Big),\end{array}
\end{equation}
which we split as three sums. Since $|\Lambda_{d_1}^-|\leq 5$, pulling the absolute value inside the third sum in \eqref{eqn:sievedecomp} yields the last term on the right-hand side of the proposition, and we henceforth ignore it.

We bring the absolute value inside the second sum and then further split it depending on whether $\bm{d}\in \mathcal{S}_1^3$ or not.  The second sum is then bounded by $5$ times
\begin{multline}\label{eqn:sievesplit}
\hspace{-.5cm}\sum_{\substack{\bm{d}\in\mathcal{S}_1^3\\ p\mid d_j\implies p\in \mathcal{P}_{h,z_0}}}\hspace{-.5cm}\left|r_{\pspn(X^{\bm{d}\bm{\ell}})}(h)-r_{\pgen(X^{\bm{d}\bm{\ell}})}(h)\right|\ + \hspace{-.6cm}\sum_{\substack{\bm{d'}\in\mathcal{S}_{2^a}^3\setminus \mathcal{S}_1^3\\ p\mid d_j'\implies p\in\mathcal{P}_{h,z_0}\cup\{2\}}}\hspace{-.6cm}\left|r_{\pspn(X^{\bm{d'}\bm{\ell}})}(h)-r_{\pgen(X^{\bm{d'}\bm{\ell}})}(h)\right|.
\end{multline}
For $\bm{d}\in \mathcal{S}_1^3$, we have by Proposition \ref{proposition:unarypartforodd:d} that
\[
r_{\pspn(X^{\bm{d}\bm{\ell}})}(h)-r_{\pgen(X^{\bm{d}\bm{\ell}})}(h)=0,
\]
so the first sum on the right-hand side vanishes. For $\bm{d'}\in \mathcal{S}_{2^a}^3\setminus\mathcal{S}_1^3$, write $\bm{d'}=\bm{d}\cdot (2^{a_1},2^{a_2},2^{a_3})$ with $\bm{d}\in \mathcal{S}_1^3$, and recall the notation in Proposition \ref{proposition:unarypartford:general}
\[
r_{\pspn(X^{\bm{d'}\bm{\ell}})}(h)-r_{\pgen(X^{\bm{d'}\bm{\ell}})}(h)= c_{\bm{a},\bm{d}\bm{\ell}}(h)r_{\pgen(X^{\bm{d}\bm{\ell}})}(h)
\]
with the properties that
\[
-2^{-a_1-a_2-a_3+\min_j\{a_j\}}\leq  c_{\bm{a},\bm{d}\bm{\ell}}(h)\leq 2^{-a_1-a_2-a_3+\min_j\{a_j\}}
\]
and $c_{\bm{0},\bm{d}\bm{\ell}}(h)=0$. Therefore 
\[
\sum_{\substack{\bm{d'}\in\mathcal{S}_{2^a}^3\setminus \mathcal{S}_1^3\\ p\mid d_j'\implies p\in\mathcal{P}_{h,z_0}\cup\{2\}}}\hspace{-.5cm}\left|r_{\pspn(X^{\bm{d'}\bm{\ell}})}(h)-r_{\pgen(X^{\bm{d'}\bm{\ell}})}(h)\right|
=\hspace{-.2cm}\sum_{\bm{a}\in \{0,a\}^3\setminus\{\bm{0}\}} \hspace{-.2cm}\sum_{\substack{\bm{d}\in \mathcal{S}_1^3\\ p\mid d_j\implies p\in \mathcal{P}_{h,z_0}}}\hspace{-.3cm} \left|c_{\bm{a},\bm{d}\bm{\ell}}(h)\right|r_{\pgen(X^{\bm{d}\bm{\ell}})}(h).
\]
Plugging this back into the second term in \eqref{eqn:sievesplit} and then plugging back into \eqref{eqn:sievedecomp} yields that the first two terms in \eqref{eqn:sievedecomp} are bounded below by 
\begin{equation}\label{eqn:firsttwoterms}
\hspace{-.4cm}\sum_{\substack{\bm{d}\in\mathcal{S}_{2^a}^3\\ p\mid d_j\implies p\in\mathcal{P}_{h,z_0}\cup\{2\}}}\hspace{-0.8cm}\Lambda_{d_1}^-\lambda_{d_2}^+\lambda_{d_3}^+r_{\pgen(X^{\bm{d}\bm{\ell}})}(h)-5\hspace{-.3cm}\sum_{\substack{\bm{d}\in\mathcal{S}_1^3\\ p\mid d_j\implies p\in\mathcal{P}_{h,z_0}}}\hspace{-.3cm}r_{\pgen(X^{\bm{d}\bm{\ell}})}(h)\left(\sum_{\bm{a}\in \{0,a\}^3\setminus\{\bm{0}\}}\left|c_{\bm{a},\bm{d}\bm{\ell}}(h)\right|\right).
\end{equation}

Using \eqref{eqn:rgenratio} and noting that 
\begin{equation}\label{eqn:local2density}
\sum_{\substack{\bm{\alpha}\in \{0,1\}^3}} \mu\left(2^{\alpha_1}\right)\mu\left(2^{\alpha_2}\right)\mu\left(2^{\alpha_3}\right) 2^{a\left(-\alpha_1-\alpha_2-\alpha_3+\min_j\{\alpha_j\}\right)}=1-\frac{3}{2^a}+\frac{3}{2^{2a}}-\frac{1}{2^{2a}},
\end{equation}
we conclude that 
\begin{equation}\label{eqn:firsttwoterms-dodd}
\hspace{-1cm}\sum_{\substack{\bm{d}\in\mathcal{S}_{2^a}^3\\ p\mid d_j\implies p\in\mathcal{P}_{h,z_0}\cup\{2\}}}\hspace{-0.8cm}\Lambda_{d_1}^-\lambda_{d_2}^+\lambda_{d_3}^+r_{\pgen(X^{\bm{d}\bm{\ell}})}(h)=\hspace{-.5cm}\sum_{\substack{\bm{d}\in\mathcal{S}_1^3\\ p\mid d_j\implies p\in\mathcal{P}_{h,z_0}}}\hspace{-.7cm}
\Lambda_{d_1}^-\lambda_{d_2}^+\lambda_{d_3}^+r_{\pgen(X^{\bm{d}\bm{\ell}})}(h)\left(1-\frac{3}{2^a}+\frac{3}{2^{2a}}-\frac{1}{2^{2a}}\right).
\end{equation}
We note by Proposition \ref{proposition:unarypartford:general} that 
\begin{equation}\label{eqn:boundabsval2factors}
\sum_{\bm{a}\in \{0,a\}^3\setminus\{\bm{0}\}}\left|c_{\bm{a},\bm{d}\bm{\ell}}(h)\right|\leq \sum_{\bm{a}\in \{0,a\}^3\setminus\{\bm{0}\}} 2^{-a_1-a_2-a_3+\min_j\{a_j\}}
=\frac{3\cdot 2^{a}+4}{2^{2a}}=:s_a(h).
\end{equation}
Plugging in \eqref{eqn:firsttwoterms-dodd}, we next rewrite $r_{\pgen(X^{\bm{d}\bm{\ell}})}(h)$ appearing in \eqref{eqn:firsttwoterms} in terms of $r_{\pgen(X^{\bm{d}})}(h)$. 

Combining \eqref{eqn:boundabsval2factors} and \eqref{eqn:ratiorgendrgen1intoprod} with the condition for $c_{\bm{a},\bm{d}}(h)\neq 0$ in Proposition \ref{proposition:unarypartford:general}, we may bound \eqref{eqn:firsttwoterms} from below by
\begin{equation}\label{eqn:lowerboundfirsttwoterms}
\left(1-\frac{3}{2^a}+\frac{3}{2^{2a}}-\frac{1}{2^{2a}}\right)M_{h}^{-}(z_0)r_{\pgen(X^{\bm{\ell}})}(h)-5S_{a,h}^{\operatorname{ET}}\left(z_0\right)r_{\pgen(X^{\bm{\ell}})}(h).
\end{equation}
\noindent

\noindent
(2) Plugging the other inequality from \eqref{eqn:rosserweights} into the calculation of \eqref{eqn:Sahlower} (where \eqref{eqn:RosserWeightsProduct} is used) yields the claim by the same steps as in (1).
\end{proof}

In order to prove Proposition \ref{prop:SahMainz0}, we need to get bounds for the main term.
\begin{proof}[Proof of Proposition \ref{prop:SahMainz0}]
The error terms are of the same size in Proposition \ref{prop:SahBound} (1) and (2), so it remains to bound the main terms.

We first consider the main term from Proposition \ref{prop:SahBound} (2). We separate the terms off which have $\gcd(d_j,d_k)>\Delta$ for some $j\neq k$. For those terms, we use Lemma \ref{lem:beta} (4) to bound (WLOG, assume that $\gcd(d_2,d_3)>\Delta$ by symmetry) 
\begin{multline*}
\hspace{-.32cm}\sum_{d_1\mid P_{\sf(h)}(z_0)} \hspace{-.2cm}\lambda_{d_1}^+\hspace{-.2cm}\sum_{d_2\mid P_{\sf(h)}(z_0)}\hspace{-.2cm} \lambda_{d_2}^+\hspace{-.2cm}\sum_{\substack{d_3\mid P_{\sf(h)}(z_0)\\ \gcd(d_2,d_3)>\Delta}}\hspace{-.2cm} \lambda_{d_3}^+ \prod_{p\mid \bm{d}} \beta_{X^{\bm{d}},p}(h)\leq\hspace{-.1cm}\sum_{d_1\mid P_{\sf(h)}(z_0)} \sum_{d_2\mid P_{\sf(h)}(z_0)} \sum_{\substack{d_3\mid P_{\sf(h)}(z_0)\\ \gcd(d_2,d_3)>\Delta}} \prod_{p\mid \bm{d}} \beta_{X^{\bm{d}},p}(h)\\
\ll \sum_{d_1\mid P_{\sf(h)}(z_0)} \frac{\widetilde{\omega}(d_1)}{d_1}\sum_{\delta\geq \Delta} \left(\sum_{\substack{d_2\mid P_{\sf(h)}(z_0)\\ \delta\mid d_2}} \frac{\widetilde{\omega}(d_2)}{d_2}\right)^2\ll \prod_{p\in \mathcal{P}_{h,z_0}} \left(1+\frac{\widetilde{\omega}(p)}{p}\right)^3 \sum_{\delta\geq \Delta}\mu(\delta)^2\left(\frac{\widetilde{\omega}(\delta)}{\delta}\right)^2.
\end{multline*}
We then bound, using Lemma \ref{lem:beta} (4), 
\begin{align*}
\sum_{\delta\geq \Delta}\mu(\delta)^2\left(\frac{\widetilde{\omega}(\delta)}{\delta}\right)^2&\leq \sum_{\delta\geq \Delta}\left(\frac{\delta}{\Delta}\right)^{\frac{2}{3}}\mu(\delta)^2 \left(\frac{\widetilde{\omega}(\delta)}{\delta}\right)^2\leq \Delta^{-\frac{2}{3}}\sum_{\delta\geq 1}\mu(\delta)^2\frac{\widetilde{\omega}(\delta)^2}{\delta^{\frac{4}{3}}}\\
&=\Delta^{-\frac{2}{3}}\prod_{p}\left(1+\frac{\widetilde{\omega}(p)^2}{p^{\frac{4}{3}}}\right)\leq \Delta^{-\frac{2}{3}}\prod_{p\mid n}\left(1+\frac{\widetilde{\omega}(p)^2}{p^{\frac{4}{3}}}\right)\prod_{p}\left(1+\frac{4}{p^{\frac{4}{3}} }\right)\\
&\ll \Delta^{-\frac{2}{3}}\prod_{p\mid n}\left(1+\frac{\widetilde{\omega}(p)^2}{p^{\frac{4}{3}}}\right)\ll \Delta^{-\frac{2}{3}}\prod_{p\mid n}\left(1+\frac{1}{(p-1)^{\frac{2}{3}}}\right)\ll \Delta^{-\frac{2}{3}}H(n).
\end{align*}
Note that the last inequality above holds because, for sufficiently-large $p$, 
\[
1+\frac{1}{(p-1)^{\frac{2}{3}}} \ll \left(1+ p^{-\frac{5}{3}}\right)\left(1+\frac{1}{p^{\frac{1}{3}}(p-1)^{\frac{1}{3}}}\right).
\]
\begin{extradetails}
We compute 
\begin{align*}
\frac{1+\frac{1}{(p-1)^{\frac{2}{3}}}}{1+\frac{1}{p^{\frac{1}{3}}(p-1)^{\frac{1}{3}}}} &= \left(1+\frac{1}{(p-1)^{\frac{2}{3}}}\right)\sum_{j=0}^{\infty}\frac{ (-1)^j }{p^{\frac{j}{3}}(p-1)^{\frac{j}{3}}}\\
&\leq  1+\frac{1}{(p-1)^{\frac{2}{3}}}-\frac{1}{p^{\frac{1}{3}}(p-1)^{\frac{1}{3}}}-\frac{1}{p^{\frac{1}{3}}(p-1)} + \frac{1}{p^{\frac{2}{3}}(p-1)^{\frac{2}{3}}} + \frac{1}{p^{\frac{2}{3}}(p-1)^{\frac{4}{3}}}\\
&=1+ \frac{1-\left(1-p^{-1}\right)^{\frac{1}{3}}}{(p-1)^{\frac{2}{3}}} - \frac{1-\left(1-p^{-1}\right)^{\frac{1}{3}}}{p^{\frac{1}{3}}(p-1)} + \frac{1}{p^{\frac{2}{3}}(p-1)^{\frac{4}{3}}}\ll 1+ p^{-\frac{5}{3}},
\end{align*}
so 
\[
\prod_{p\mid n}\left(1+\frac{1}{(p-1)^{\frac{2}{3}}}\right)\ll \prod_{p\mid n} \left(1+ p^{-\frac{5}{3}}\right)\prod_{p\mid n}\left(1+\frac{1}{p^{\frac{1}{3}}(p-1)^{\frac{1}{3}}}\right)\ll H(n).
\]
\end{extradetails}
We then bound (using $\sigma_{-1}(n)=\frac{\sigma(n)}{n}\ll \log(\log(n))$ by \cite[Section 22.9]{HardyWright})
\[
\prod_{\substack{p\mid \sf(h)\\ p>z_0}}\left(1+\frac{2}{p}\right)^3\leq \prod_{p\mid \sf(h)}\left(1+\frac{1}{p}\right)^6=\sigma_{-1}(\sf(h))^6\ll \log(\log(\sf(h)))^6.
\]
and use \cite[(3.30)]{RosserSchoenfeld} to obtain
\[
\prod_{p\in\mathcal{P}_{h,z_0}} \left(1+\frac{\widetilde{\omega}(p)}{p}\right)^3 \sum_{\delta\geq \Delta}\mu(\delta)^2\left(\frac{\widetilde{\omega}(\delta)}{\delta}\right)^2
\ll\Delta^{-\frac{2}{3}} \log(z_0)^6 H(n)^4\log(\log(\sf(h)))^6.
\]
\begin{extradetails}
We have 
\[
 \prod_{p<z_0}\left(1+\frac{2}{p}\right)\leq  \prod_{p<z_0}\left(1+\frac{1}{p}\right)^2=\prod_{p<z_0}\left(1-\frac{1}{p}\right)^{-2}\left(1-p^{-2}\right)^{2}\leq \prod_{p<z_0}\left(1-\frac{1}{p}\right)^{-2}\ll \log(z_0)^2,
\]
where we use \cite[(3.30)]{RosserSchoenfeld} in the last step.
\end{extradetails}
Hence we may rewrite  $M_{h}^+(z_0)$ as 
\[
\hspace{-.25cm}\sum_{d_1\mid P_{\sf(h)}(z_0)}\hspace{-.2cm} \lambda_{d_1}^+\hspace{-.2cm}\sum_{d_2\mid P_{\sf(h)}(z_0)}\hspace{-.2cm} \lambda_{d_2}^+\hspace{-.2cm}\sum_{\substack{d_3\mid P_{\sf(h)}(z_0)\\ \gcd(d_i,d_j)\leq \Delta}} \hspace{-.2cm}\lambda_{d_3}^+ \prod_{p\mid \bm{d}} \beta_{X^{\bm{d}},p}(h)+O\left(\Delta^{-\frac{2}{3}} \log(z_0)^6 H(n)^4\log(\log(\sf(h)))^6\right).
\]
We group together those $\bm{d}$ with $\gcd(d_i,d_j)=u_{i,j}$. Thus 
\begin{multline*}
\sum_{d_1\mid P_{\sf(h)}(z_0)} \lambda_{d_1}^+\sum_{d_2\mid P_{\sf(h)}(z_0)} \lambda_{d_2}^+\sum_{\substack{d_3\mid P_{\sf(h)}(z_0)\\ \gcd(d_i,d_j)\leq \Delta}}\lambda_{d_3}^+ \prod_{p\mid \bm{d}} \beta_{X^{\bm{d}},p}(h)\\
=\sum_{u_{1,2}\leq \Delta}\sum_{u_{1,3}\leq \Delta}\sum_{u_{2,3}\leq \Delta}\sum_{d_1\mid P_{\sf(h)}(z_0)} \lambda_{d_1}^+\sum_{d_2\mid P_{\sf(h)}(z_0)} \lambda_{d_2}^+\sum_{\substack{d_3\mid P_{\sf(h)}(z_0)\\ \gcd(d_i,d_j)= u_{i,j}}}\lambda_{d_3}^+ \prod_{p\mid \bm{d}} \beta_{X^{\bm{d}},p}(h).
\end{multline*}
Since $g(\bm{d})$ only depends on $u_{i,j}$ by Lemma \ref{lem:beta} (5), we may write (abbreviating $g(\bm{d})$ by $g(\bm{u})$ for $\gcd(d_i,d_j)=u_{i,j}$)
\begin{multline*}
\sum_{u_{1,2}\leq \Delta}\sum_{u_{1,3}\leq \Delta}\sum_{u_{2,3}\leq \Delta}\sum_{d_1\mid P_{\sf(h)}(z_0)} \lambda_{d_1}^+\sum_{d_2\mid P_{\sf(h)}(z_0)} \lambda_{d_2}^+\sum_{\substack{d_3\mid P_{\sf(h)}(z_0)\\ \gcd(d_i,d_j)= u_{i,j}}}\lambda_{d_3}^+ \prod_{p\mid \bm{d}} \beta_{X^{\bm{d}},p}(h)\\
=\sum_{u_{1,2}\leq \Delta}\sum_{u_{1,3}\leq \Delta}\sum_{u_{2,3}\leq \Delta}g(\bm{u}) \sum_{d_1\mid P_{\sf(h)}(z_0)} \lambda_{d_1}^+\prod_{p\mid d_1}\beta_{X^{(p,1,1)},p}(h) \sum_{d_2\mid P_{\sf(h)}(z_0)} \lambda_{d_2}^+\prod_{p\mid d_2}\beta_{X^{(1,p,1)},p}(h)\\
\times\sum_{\substack{d_3\mid P_{\sf(h)}(z_0)\\ \gcd(d_i,d_j)= u_{i,j}}}\lambda_{d_3}^+\prod_{p\mid d_3}\beta_{X^{(1,1,p)},p}(h).
\end{multline*}
We rewrite the gcd condition on the inner sums by excluding divisibility by other primes using $\mu$ (or inclusion/exclusion). Letting $\ell_{i,j}$ run through all possible divisors of $P_{\sf(h)}(z_0)$ relatively prime to $u_{i,j}$, this gives  
\begin{multline*}
\sum_{d_1\mid P_{\sf(h)}(z_0)} \lambda_{d_1}^+\prod_{p\mid d_1}\beta_{X^{(p,1,1)},p}(h)\sum_{d_2\mid P_{\sf(h)}(z_0)} \lambda_{d_2}^+\prod_{p\mid d_2}\beta_{X^{(1,p,1)},p}(h)\sum_{\substack{d_3\mid P_{\sf(h)}(z_0)\\ \gcd(d_i,d_j)= u_{i,j}}}\lambda_{d_3}^+\prod_{p\mid d_3}\beta_{X^{(1,1,p)},p}(h)\\
=\sum_{\ell_{1,2}}\mu\left(\ell_{1,2}\right)\sum_{\ell_{1,3}}\mu\left(\ell_{1,3}\right)\sum_{\ell_{2,3}}\mu\left(\ell_{2,3}\right)\sum_{\substack{d_1\mid P_{\sf(h)}(z_0)\\ \lcm(u_{1,2}\ell_{1,2},u_{1,3}\ell_{1,3})\mid d_1}} \lambda_{d_1}^+\prod_{p\mid d_1}\beta_{X^{(p,1,1)},p}(h)\\
\times \sum_{\substack{d_2\mid P_{\sf(h)}(z_0)\\ \lcm(u_{1,2}\ell_{1,2},u_{2,3}\ell_{2,3})\mid d_2}} \lambda_{d_2}^+\prod_{p\mid d_2}\beta_{X^{(1,p,1)},p}(h)\sum_{\substack{d_3\mid P_{\sf(h)}(z_0)\\ \lcm(u_{1,3}\ell_{1,3},u_{2,3}\ell_{2,3})\mid d_3}}\lambda_{d_3}^+\prod_{p\mid d_3}\beta_{X^{(1,1,p)},p}(h).
\end{multline*}
By Lemma \ref{lem:beta} (3), we have 
\begin{multline*}
\sum_{\substack{d_1\mid P_{\sf(h)}(z_0)\\ \delta \mid d_1}} \lambda_{d_1}^+\prod_{p\mid d_1}\beta_{X^{(p,1,1)},p}(h) \leq \prod_{p\mid \delta}\beta_{X^{(\delta,1,1)},p}(h) \prod_{p\mid P_{\sf(h)}(z_0)}\left(1+\frac{2}{p}\right)\\
\leq\prod_{p\mid \delta} \beta_{X^{(\delta,1,1)},p}(h) \prod_{p\mid P_{\sf(h)}(z_0)}\left(1+\frac{1}{p}\right)^2\ll \frac{2^{\omega(\delta)}}{\varphi(\delta)}\log(z_0)^2\prod_{p\mid \sf(h)}\left(1+\frac{1}{p}\right)^2\\
\ll\frac{2^{\omega(\delta)}}{\varphi(\delta)}\log(z_0)^2\log(\log(\sf(h)))^2,
\end{multline*}
using \cite[Section 22.9]{HardyWright} in the last step. We then use
\[
\frac{1}{\delta} \prod_{p\mid \delta}\frac{2}{1-\frac{1}{p}}\leq \frac{1}{\delta}\prod_{p\mid \delta} 4 = \frac{\sigma_0(\delta)^2}{\delta}
\]
to bound 
\[
\frac{2^{\omega(\delta)}}{\varphi(\delta)}\log(z_0)^2\log(\log(\sf(h)))^2\leq \frac{\sigma_0(\delta)^2}{\delta}\log(z_0)^2\log(\log(\sf(h)))^2.
\]
Writing $\lcm(\ell_{i,j},\ell_{i,k})=\frac{\ell_{i,j}\ell_{i,k}}{\gcd(\ell_{i,j},\ell_{i,k})}$ and bounding (using multiplicativity)
\[
\sigma_0\left(\lcm(\ell_{i,j},\ell_{i,k})\right)^2\leq \sigma_0\left(\ell_{i,j}\right)^2\sigma_0\left(\ell_{i,k}\right)^2,
\]
 we thus want to bound 
\[
\sum_{\substack{\ell_{2,3}\leq D_0\\ \ell_{1,3}\leq D_0\\ \Delta^{B} <\ell_{1,2}\leq D_0 }}  \frac{\mu(\ell_{1,2})^2\mu(\ell_{1,3})^2\mu(\ell_{2,3})^2 \sigma_0\left(\ell_{1,2}\right)^4\sigma_0\left(\ell_{1,3}\right)^4\sigma_0\left(\ell_{2,3}\right)^4 \left(\ell_{1,3},\ell_{2,3}\right)\left(\ell_{1,2},\ell_{2,3}\right)\left(\ell_{1,2},\ell_{1,3}\right)}{\ell_{1,2}^2\ell_{1,3}^2\ell_{2,3}^2}.
\]
We then bound 
\begin{multline*}
\sum_{\ell_{2,3}\leq D_0} \frac{\mu(\ell_{2,3})^2\sigma_0\left(\ell_{2,3}\right)^4\left(\ell_{1,3},\ell_{2,3}\right)\left(\ell_{1,2},\ell_{2,3}\right)}{\ell_{2,3}^2}\leq  \prod_{p\mid \ell_{1,3}\ell_{1,2}} \left(2^4+1\right)\prod_{p}\left(1+\frac{2^4}{p^2}\right)\\
\ll \sigma_0\left(\ell_{1,3}\right)^{5}\sigma_0\left(\ell_{1,2}\right)^{5}.
\end{multline*}
Plugging this in, it remains to bound 
\[
\sum_{\substack{\ell_{1,3}\leq D_0\\ \Delta^{B} <\ell_{1,2}\leq D_0 }}  \frac{\mu(\ell_{1,2})^2\mu(\ell_{1,3})^2 \sigma_0\left(\ell_{1,2}\right)^9\sigma_0\left(\ell_{1,3}\right)^9\left(\ell_{1,2},\ell_{1,3}\right)}{\ell_{1,2}^2\ell_{1,3}^2}.
\]
We then similarly bound the sum over $\ell_{1,3}$ as 
\[
\sum_{\ell_{1,3}\leq D_0} \frac{\mu(\ell_{1,3})^2\sigma_0\left(\ell_{1,3}\right)^9\left(\ell_{1,2},\ell_{1,3}\right)}{\ell_{1,3}^2}\leq \prod_{p\mid \ell_{1,2}} \left(2^9+1\right)\prod_{p}\left(1+\frac{2^9}{p^2}\right)\ll \sigma_0\left(\ell_{1,2}\right)^{10}.
\]
The sum over $\ell_{1,2}$ is now bounded by 
\[
\sum_{\ell_{1,2}>\Delta^B} \frac{\mu(\ell_{1,2})^2 \sigma_0\left(\ell_{1,2}\right)^{19}}{\ell_{1,2}^2}\ll_{\varepsilon}\Delta^{\varepsilon-B}.
\]

We therefore obtain that the contribution from those $\ell_{i,j}$ with at least one of them  larger than $\Delta^B$ is 
\[
\ll_{\varepsilon} \Delta^{\varepsilon-B}\log(z_0)^2\log(\log(\sf(h)))^2.
\]
We finally use \cite[Lemma 11]{BrudernFouvry} to finish the proof for the upper bound. Since we have taken the absolute value at each step, the proof for the lower bound follows by the same argument, using Proposition \ref{prop:SahBound} (1) as a starting point.
 \end{proof}
\section{Sieving and the proof of Theorem \ref{thm:main}}\label{sec:final}

We finally employ a second sieve to prove Theorem \ref{thm:main}.

\begin{proof}[Proof of Theorem \ref{thm:main}]
Suppose that $z_0<z$ and set $P_h(z_0,z):=\prod_{\substack{z_0<p\leq z\\ p\nmid \sf(h)}} p$. Then we see from \eqref{eqn:RosserWeightsProduct} that
\begin{align*}
S_{a,h}\left(\mathcal{A}_{\bm{1}},z\right)&= \sum_{\bm{x}\in \mathscr{S}_{a,h}\left(\mathcal{A}_{\bm{1}},z_0\right)}\prod_{j=1}^3\sum_{\ell_j\mid \gcd(x_j,P_h(z_0,z))} \mu\left(\ell_j\right)\\
&= \sum_{\bm{x}\in \mathscr{S}_{a,h}\left(\mathcal{A}_{\bm{1}},z_0\right)}\sum_{\substack{\ell_1\mid P_h(z_0,z)\\ \ell_1\mid x_1}}\mu(\ell_1)\sum_{\substack{\ell_2\mid P_h(z_0,z)\\ \ell_2\mid x_2}}\mu(\ell_2)\sum_{\substack{\ell_3\mid P_h(z_0,z)\\ \ell_3\mid x_3}}\mu(\ell_3)\\
&\overset{\eqref{eqn:RosserWeightsProduct}}\geq  \sum_{\bm{x}\in \mathscr{S}_{a,h}\left(\mathcal{A}_{\bm{1}},z_0\right)}\sum_{\substack{\ell_1\mid P_h(z_0,z)\\ \ell_1\mid x_1}}\Lambda_{\ell_1}^-\sum_{\substack{\ell_2\mid P_h(z_0,z)\\ \ell_2\mid x_2}}\lambda_{\ell_2}^+\sum_{\substack{\ell_3\mid P_h(z_0,z)\\ \ell_3\mid x_3}}\lambda_{\ell_3}^+\\
&=\sum_{\ell_1\mid P_h(z_0,z)}\Lambda_{\ell_1}^-\sum_{\ell_2\mid P_h(z_0,z)}\lambda_{\ell_2}^+\sum_{\ell_3\mid P_h(z_0,z)}\lambda_{\ell_3}^+S_{a,h}\left(\mathcal{A}_{\bm{\ell}},z_0\right).
\end{align*}
 We may then plug in Proposition \ref{prop:SahMainz0} in the inner sum to obtain (note that $S_{a,h}^{\operatorname{ET}}\left(z_0\right)$ is independent of $\bm{\ell}$)
\begin{multline*}
S_{a,h}\left(\mathcal{A}_{\bm{1}},z\right)\geq \sum_{\ell_1\mid P_h(z_0,z)}\Lambda_{\ell_1}^-\sum_{\ell_2\mid P_h(z_0,z)}\lambda_{\ell_2}^+\sum_{\ell_3\mid P_h(z_0,z)}\lambda_{\ell_3}^+
\bigg(W_{a,h}(z_0)+O_{\varepsilon}\Big(\Delta^Ce^{-s_0}\\
+H(n)^{4}\Delta^{-\frac{2}{3}}\log(D_0)^6\log(\log(\sf(h)))^6 +\Delta^{\varepsilon-B}\log(z_0)^2\log(\log(\sf(h)))^2\Big)\bigg)r_{\pgen(X^{\bm{\ell}})}(h)\\
+O\Bigg(S_{a,h}^{\operatorname{ET}}\left(z_0\right)\hspace{-.5em}\sum_{\substack{\bm{\ell}\in \mathcal{S}_{2^a}^3\\ p\mid \ell_j\implies z_0<p\leq z\\ \text{and }p\nmid \sf(h)\\ \ell_j\leq D}}\hspace{-1em}r_{\pgen(X^{\bm{\ell}})}(h)\Bigg)+O\Bigg(\hspace{-.5em}\sum_{\substack{\bm{\ell}\in \mathcal{S}_{2^a}^3\\ p\mid \ell_j\implies z_0<p\leq z\\ \text{and }p\nmid \sf(h)\\ \ell_j\leq D}}\sum_{\substack{\bm{d}\in\mathcal{S}_{2^a}^3\\ p\mid d_j\implies p\leq z_0\\ \text{or }p\mid \sf(h)\\ d_j\leq D_0}}\hspace{-.5em}\left|r_{X^{\bm{d}\bm{\ell}}}(h)-r_{\pspn(X^{\bm{d}\bm{\ell}})}(h)\right|\Bigg).
\end{multline*}
We first bound the last $O$-term. Lemma \ref{lem:OrthogonalBound} implies that 
\[
\sum_{\substack{\bm{\ell}\in \mathcal{S}_{2^a}^3\\ p\mid \ell_j\implies z_0<p\leq z\\ \text{and }p\nmid \sf(h)\\ \ell_j\leq D}}\sum_{\substack{\bm{d}\in\mathcal{S}_{2^a}^3\\ p\mid d_j\implies p\leq z_0\\ \text{or }p\mid \sf(h)\\ d_j\leq D_0}}\left|r_{X^{\bm{d}\bm{\ell}}}(h)- r_{\pspn(X^{\bm{d}\bm{\ell}})}(h)\right|\\
\ll \hspace{-0.5em}\sum_{\substack{\bm{\ell}\in \mathcal{S}_{2^a}^3\\ p\mid \ell_j\implies z_0<p\leq z\\ \text{and }p\nmid \sf(h)\\ \ell_j\leq D}}\sum_{\substack{\bm{d}\in\mathcal{S}_{2^a}^3\\ p\mid d_j\implies p\leq z_0\\ \text{or }p\mid \sf(h)\\ d_j\leq D_0}}\hspace{-1em}
N_{m,\bm{d}\bm{\ell}}^{\frac{41}{16}+\varepsilon} M_{m,\bm{d}\bm{\ell}}^{\frac{5}{2}+\varepsilon} n^{\frac{231}{512}+\varepsilon}.
\]
Since  $\ell_j\leq D$, $d_j\leq D_0$,  $N_{m,\bm{d}\bm{\ell}}\ll_m d_1^2d_2^2d_3^2\ell_1^2\ell_2^2\ell_3^2$, and $M_{m,\bm{d}\bm{\ell}}\ll_m d_1d_2d_3\ell_1\ell_2\ell_3$, we have 
\begin{multline}\label{eqn:cuspidalsumbound}
\sum_{\substack{\bm{\ell}\in \mathcal{S}_{2^a}^3\\ p\mid \ell_j\implies z_0<p\leq z\\ \text{and }p\nmid \sf(h)\\ \ell_j\leq D}}\sum_{\substack{\bm{d}\in\mathcal{S}_{2^a}^3\\ p\mid d_j\implies p\leq z_0\\ \text{or }p\mid \sf(h)\\ d_j\leq D_0}}\hspace{-1em} N_{m,\bm{d}\bm{\ell}}^{\frac{41}{16}+\varepsilon} M_{m,\bm{d}\bm{\ell}}^{\frac{5}{2}+\varepsilon} n^{\frac{231}{512}+\varepsilon}\ll_m (DD_0)^3 (DD_0)^{\frac{41}{16}\cdot 6+\varepsilon} (DD_0)^{\frac{5}{2}\cdot 3+\varepsilon}h^{\frac{231}{512}+\varepsilon}\\
=(DD_0)^{\frac{207}{8}+\varepsilon}h^{\frac{231}{512}+\varepsilon}.
\end{multline}
We next bound the error term containing $S_{a,h}^{\operatorname{ET}}\left(z_0\right)$. By \eqref{eqn:ratiorgendrgen1intoprod}, we have 
\[
r_{\pgen(X^{\bm{\ell}})}(h)=r_{\pgen(X^{\bm{1}})}(h)\prod_{p\mid \bm{\ell}} \beta_{X^{\bm{\ell}},p}(h),
\]
 By choosing $z_0$ sufficiently large, we may assume that every prime larger than $z_0$ does not divide $(m-2)(m-4)$, and hence we may use Lemma \ref{lem:betabound} (3)  to bound 
\[
\sum_{\substack{\bm{\ell}\in \mathcal{S}_{2^a}^3\\ p\mid \ell_j\implies z_0<p\leq z\\ \text{and }p\nmid \sf(h)\\ \ell_j\leq D}}\hspace{-.2cm}\prod_{p\mid\bm{\ell}}\beta_{X^{\bm{\ell}},p}(h)\ll\hspace{-.15cm} \prod_{p\mid P_h(z_0,z)} \sum_{\bm{c}\in \{0,1\}^3} \beta_{X^{\bm{p^c}},p}(h)\ll\hspace{-.15cm} \prod_{p\mid P_h(z_0,z)} \hspace{-.2cm}\left(1+\frac{1}{p}\right)^6\leq \prod_{p\mid P_h(z_0,z)} \left(\frac{p}{p-1}\right)^{6}.
\]
\begin{extradetails}
\begin{multline*}
\sum_{\substack{\bm{\ell}\in \mathcal{S}_{2^a}^3\\ p\mid \ell_j\implies z_0<p\leq z\\ \text{and }p\nmid \sf(h)\\ \ell_j\leq D}}\prod_{p\mid\bm{\ell}}\beta_{X^{\bm{\ell}},p}(h)\ll \prod_{p\mid P_h(z_0,z)} \sum_{\bm{c}\in \{0,1\}^3} \beta_{X^{\bm{p^c}},p}(h)\ll \prod_{p\mid P_h(z_0,z)} \left(1+\frac{1}{p}\right)^6\\
\leq  \prod_{p\mid P_h(z_0,z)} \left(1-\frac{1}{p}\right)^{-6}=\prod_{p\mid P_h(z_0,z)} \left(\frac{p}{p-1}\right)^{6}.
\end{multline*}
Here we used the fact that 
\[
1+\frac{1}{p}\leq \frac{1}{1-\frac{1}{p}}.
\]
\end{extradetails}
We then use \cite[(3.30) and (3.26)]{RosserSchoenfeld} to bound 
\[
\prod_{p\mid P_h(z_0,z)} \left(\frac{p}{p-1}\right)^{6}\leq \frac{\log(z)^6}{\log(z_0)^6} \left(1+\frac{1}{\log^2(z)}\right)^6\left(1+\frac{1}{2\log^2(z_0)}\right)^6.
\]
\begin{extradetails}
\begin{multline*}
\prod_{p\mid P_h(z_0,z)} \left(\frac{p}{p-1}\right)^{6}=\prod_{p\leq z}\left(\frac{p}{p-1}\right)^{6}\prod_{p\leq z_0}\left(1-\frac{1}{p}\right)^{6}\prod_{\substack{p\mid \sf(h)\\ z_0<p\leq z}}\left(\frac{p}{p-1}\right)^{-6}\\
\leq \prod_{p\leq z}\left(\frac{p}{p-1}\right)^{6}\prod_{p\leq z_0}\left(1-\frac{1}{p}\right)^{6}\leq \frac{\log(z)^6}{\log(z_0)^6} \left(1+\frac{1}{\log^2(z)}\right)^6\left(1+\frac{1}{2\log^2(z_0)}\right)^6.
\end{multline*}
\end{extradetails}
Thus 
\begin{equation}\label{eqn:mainerrortermsumfinalbound}
\sum_{\substack{\bm{\ell}\in \mathcal{S}_{2^a}^3\\ p\mid \ell_j\implies z_0<p\leq z\\ \text{and }p\nmid \sf(h)\\ \ell_j\leq D}}\hspace{-1em}r_{\pgen(X^{\bm{\ell}})}(h)\ll r_{\pgen(X^{\bm{1}})}(h)\left(\frac{\log(z)}{\log(z_0)}\right)^6.
\end{equation}
We next bound $S_{a,h}^{\operatorname{ET}}\left(z_0\right)$. Using Lemma \ref{lem:betabound}, we bound (noting that the worst case occurs when $p\mid (m-4)$)
\[
\prod_{\substack{p\leq z_0,\, p\nmid \text{sf}(h)\\ p\neq 2}} \left(\sum_{\bm{c}\in \{0,1\}^3} \beta_{X^{\bm{p^{c}}},p}(h)\right)
\leq \prod_{\substack{p\leq z_0,\, p\nmid \text{sf}(h)\\ p\neq 2}} \left(1+\frac{1}{p}\right)^{13}\leq\prod_{\substack{p\leq z_0,\, p\nmid \text{sf}(h)\\ p\neq 2}} \frac{1}{\left(1-\frac{1}{p}\right)^{13}}.
\]
Hence  \cite[(3.30)]{RosserSchoenfeld} implies that
\begin{equation}\label{eqn:Uppernotsqfreez0}
\hspace{-.1cm}\prod_{\substack{p\leq z_0,\, p\nmid \text{sf}(h)\\ p\neq 2}} \hspace{-.15cm}\left(\sum_{\bm{c}\in \{0,1\}^3} \beta_{X^{\bm{p^{c}}},p}(h)\right)
\leq\hspace{-.1cm}\prod_{2<p\leq z_0} \frac{1}{\left(1-\frac{1}{p}\right)^{13}}\leq \left(\frac{e^{\gamma}}{2}\right)^{13}\hspace{-.1cm}\log(z_0)^{13}\left(1+\frac{1}{\log^2(z_0)}\right)^{13}.
\end{equation}
Using Lemma \ref{lem:betaboundspecialcase}, the other factors in \eqref{eqn:SahETz0} may be bounded by 
\[
\left(\sum_{\bm{c}\in \{0,1\}^3} \beta_{X^{\bm{3^{c}}},3}(h)\right)^{\delta_{3\mid\text{sf}(h)}\delta_{3\mid(m-2)}}
	\left(\sum_{\bm{c}=(1,1,1)} \beta_{X^{\bm{3^{c}}},3}(h)\right)^{\delta_{3\mid\text{sf}(h)}\delta_{3\nmid(m-2)}}\leq \frac{64}{27}
\]
and
\begin{multline*}
\prod_{\substack{p\mid \text{sf}(h)\\ p\neq2,\, p\nmid 3(m-2)(m-4)}} \left(\sum_{\#\{j:c_j=1\}=2} \beta_{X^{\bm{p^{c}}},p}(h)\right)\leq \frac{1}{\operatorname{sf}(h)^2}\prod_{5\leq p\mid \operatorname{sf}(h)} \left(6\frac{1+p^{-1}}{1-p^{-2}}\right)\\
=\frac{6^{\omega(\operatorname{sf}(h))}}{\operatorname{sf}(h)^2}\frac{\sigma_{-1}(\operatorname{sf}(h))}{\prod_{5\leq p\mid\operatorname{sf}(h)} \left(1-p^{-2}\right)}\leq \zeta(2)\frac{6^{\omega(\operatorname{sf}(h))}}{\operatorname{sf}(h)^2}\sigma_{-1}(\operatorname{sf}(h)).
\end{multline*}
We therefore conclude that 
\begin{equation}\label{eqn:SahETbound}
S_{a,h}^{\operatorname{ET}}\left(z_0\right)\ll 2^{-a}\log(z_0)^{13} \frac{6^{\omega(\operatorname{sf}(h))}}{\operatorname{sf}(h)^2}\sigma_{-1}(\operatorname{sf}(h)).
\end{equation}
Combining \eqref{eqn:SahETbound} with \eqref{eqn:mainerrortermsumfinalbound}, we obtain the upper bound
\begin{equation}\label{eqn:SahError}
 S_{a,h}^{\operatorname{ET}}\left(z_0\right)\hspace{-.8em}\sum_{\substack{\bm{\ell}\in \mathcal{S}_{2^a}^3\\ p\mid \ell_j\implies z_0<p\leq z\\ \text{and }p\nmid \sf(h)\\ \ell_j\leq D}}\hspace{-1em}r_{\pgen(X^{\bm{\ell}})}(h)\ll \frac{\log(z_0)^{13}}{2^a} \frac{6^{\omega(\operatorname{sf}(h))}}{\operatorname{sf}(h)^2}\sigma_{-1}(\operatorname{sf}(h))\left(\frac{\log(z)}{\log(z_0)}\right)^6 r_{\pgen(X^{\bm{1}})}(h).
\end{equation}
For the error terms from the main term, we again note that they are independent of $\bm{\ell}$, so, using \eqref{eqn:mainerrortermsumfinalbound},  they may be bounded by 
\begin{multline}\label{eqn:MainErrorFinal}
\hspace{-.25cm}\left(\frac{\Delta^C}{e^{s_0}}+\frac{H(n)^{4}\log(D_0)^6\log(\log(\sf(h)))^6}{\Delta^{\frac{2}{3}}} +\frac{\log(z_0)^2\log(\log(\sf(h)))^2}{\Delta^{B-\varepsilon}}\right)\hspace{-0.6cm} \sum_{\substack{\bm{\ell}\in \mathcal{S}_{2^a}^3\\ p\mid \ell_j\implies z_0<p\leq z\\ \text{and }p\nmid \sf(h)\\ \ell_j\leq D}}\hspace{-1em}r_{\pgen(X^{\bm{\ell}})}(h)\\
\ll  \left(\frac{\Delta^C}{e^{s_0}}+\frac{H(n)^{4}\log(D_0)^6\log(\log(\sf(h)))^6}{\Delta^{\frac{2}{3}}} +\frac{\log(z_0)^2\log(\log(\sf(h)))^2}{\Delta^{B-\varepsilon}}\right)\left(\frac{\log(z)}{\log(z_0)}\right)^6r_{\pgen(X^{\bm{1}})}(h).
\end{multline}
For the main term, we next write 
\begin{multline}\label{eqn:gcdsplit}
\sum_{\ell_1\mid P_h(z_0,z)}\Lambda_{\ell_1}^-\sum_{\ell_2\mid P_h(z_0,z)}\lambda_{\ell_2}^+\sum_{\ell_3\mid P_h(z_0,z)}\lambda_{\ell_3}^+\prod_{p\mid\bm{\ell}} \beta_{X^{\bm{\ell}},p}(h)\\
=\sum_{\ell_1\mid P_h(z_0,z)}\Lambda_{\ell_1}^-\sum_{\substack{\ell_2\mid P_h(z_0,z)\\ \gcd(\ell_1,\ell_2)=1}}\lambda_{\ell_2}^+\sum_{\substack{\ell_3\mid P_h(z_0,z)\\ \gcd(\ell_1,\ell_3)=1\\ \gcd(\ell_2,\ell_3)=1}}\lambda_{\ell_3}^+\prod_{p\mid\bm{\ell}} \beta_{X^{\bm{\ell}},p}(h)+\sum_{\substack{\bm{\ell}\mid P_h(z_0,z)^3\\ \exists (j,k): \gcd(\ell_j,\ell_k)\neq 1}}\Lambda_{\ell_1}^-\lambda_{\ell_2}^+\lambda_{\ell_3}^+\prod_{p\mid\bm{\ell}} \beta_{X^{\bm{\ell}},p}(h)\\
=\sum_{\ell_1\mid P_h(z_0,z)}\Lambda_{\ell_1}^-\sum_{\substack{\ell_2\mid P_h(z_0,z)\\ \gcd(\ell_1,\ell_2)=1}}\lambda_{\ell_2}^+\sum_{\substack{\ell_3\mid P_h(z_0,z)\\ \gcd(\ell_1,\ell_3)=1\\ \gcd(\ell_2,\ell_3)=1}}\lambda_{\ell_3}^+\prod_{p\mid\bm{\ell}} \beta_{X^{\bm{\ell}},p}(h)+\sum_{\substack{\bm{\ell}\mid P_h(z_0,z)^3\\ \exists (j,k): \gcd(\ell_j,\ell_k)\neq 1}}\Lambda_{\ell_1}^-\lambda_{\ell_2}^+\lambda_{\ell_3}^+\prod_{p\mid\bm{\ell}} \beta_{X^{\bm{\ell}},p}(h)\\
=\sum_{\ell_1\mid P_h(z_0,z)}\Lambda_{\ell_1}^-\prod_{p\mid \ell_1}\beta_{X^{(\ell_1,1,1)},p}(h)\sum_{\substack{\ell_2\mid P_h(z_0,z)\\ \gcd(\ell_1,\ell_2)=1}}\lambda_{\ell_2}^+\prod_{p\mid \ell_2}\beta_{X^{(1,\ell_2,1)},p}(h)\sum_{\substack{\ell_3\mid P_h(z_0,z)\\ \gcd(\ell_1,\ell_3)=1\\ \gcd(\ell_2,\ell_3)=1}}\lambda_{\ell_3}^+\prod_{p\mid \ell_3}\beta_{X^{(1,1,\ell_3)},p}(h)\\
+\hspace{-.2cm}\sum_{\substack{\bm{\ell}\mid P_h(z_0,z)^3\\ \exists (j,k): \gcd(\ell_j,\ell_k)\neq 1}}\hspace{-.45cm}\Lambda_{\ell_1}^-\lambda_{\ell_2}^+\lambda_{\ell_3}^+\prod_{p\mid\bm{\ell}} \beta_{X^{\bm{\ell}},p}(h)=\hspace{-.2cm}\sum_{\ell_1\mid P_h(z_0,z)}\hspace{-.25cm}\Lambda_{\ell_1}^-\prod_{p\mid \ell_1}\beta_{X^{(\ell_1,1,1)},p}(h)\left(\sum_{\ell_2\mid P_h(z_0,z)}\hspace{-.25cm}\lambda_{\ell_2}^+\prod_{p\mid \ell_2}\beta_{X^{(1,\ell_2,1)},p}(h)\right)^2\\
+\sum_{\substack{\bm{\ell}\mid P_h(z_0,z)^3\\ \exists (j,k): \gcd(\ell_j,\ell_k)\neq 1}}\Lambda_{\ell_1}^-\lambda_{\ell_2}^+\lambda_{\ell_3}^+\left(\prod_{p\mid\bm{\ell}} \beta_{X^{\bm{\ell}},p}(h)-\prod_{p\mid \ell_1}\beta_{X^{(\ell_1,1,1)},p}(h)\prod_{p\mid \ell_2}\beta_{X^{(1,\ell_2,1)},p}(h)\prod_{p\mid \ell_3}\beta_{X^{(1,1,\ell_3)},p}(h)\right).
\end{multline}

Abbreviating $\beta_{\ell,g}(h):=\prod_{p\mid \gcd(\ell,g)} \beta_{X^{(\ell,1,1)},p}(h)$, we then rewrite the last sum as 
\begin{multline}\label{eqn:gcdsplitlast}
\sum_{\substack{\bm{g}\mid P_h(z_0,z)^3\\ \bm{g}\neq \bm{1}}} \sum_{\substack{\bm{\ell}\mid P_h(z_0,z)^3\\ \gcd(\ell_1,\ell_2)=g_1,\\ \gcd(\ell_1,\ell_3)=g_2\\ \gcd(\ell_2,\ell_3)=g_3}}\hspace{-.5em}\Lambda_{\ell_1}^-\lambda_{\ell_2}^+\lambda_{\ell_3}^+\left(\prod_{p\mid\bm{\ell}} \beta_{X^{\bm{\ell}},p}(h)-\beta_{\ell_1,0}(h)\beta_{\ell_2,0}(h)\beta_{\ell_3,0}(h)\right)\\
\ll\sum_{\substack{\bm{g}\mid P_h(z_0,z)^3\\ \bm{g}\neq \bm{1}}} \sum_{\substack{\bm{\ell}\mid P_h(z_0,z)^3\\ \gcd(\ell_1,\ell_2)=g_1,\\ \gcd(\ell_1,\ell_3)=g_2\\ \gcd(\ell_2,\ell_3)=g_3}}\left|\prod_{p\mid\bm{\ell}} \beta_{X^{\bm{\ell}},p}(h)-\beta_{\ell_1,0}(h)\beta_{\ell_2,0}(h)\beta_{\ell_3,0}(h)\right|\\
=\sum_{\substack{\bm{g}\mid P_h(z_0,z)^3\\ \bm{g}\neq \bm{1}}} \sum_{\substack{\bm{\ell}\mid P_h(z_0,z)^3\\ \gcd(\ell_1,\ell_2)=g_1,\\ \gcd(\ell_1,\ell_3)=g_2\\ \gcd(\ell_2,\ell_3)=g_3}}\prod_{\substack{p\mid \bm{\ell}\\ p\nmid g_1g_2g_3}}\beta_{X^{\bm{\ell}},p}(h) \left|\prod_{p\mid g_1g_2g_3} \beta_{X^{\bm{\ell}},p}(h)-\beta_{\ell_1,g_1g_2g_3}(h)\beta_{\ell_2,g_1g_2g_3}(h)\beta_{\ell_3,g_1g_2g_3}(h)\right|\\
= \sum_{\substack{g\mid P_h(z_0,z)\\ g\neq 1}}\sum_{\substack{\bm{\ell}\mid P_h(z_0,z)^3\\ g_1=\gcd(\ell_1,\ell_2)\mid g\\ g_2=\gcd(\ell_1,\ell_3)\mid g\\ g_3=\gcd(\ell_2,\ell_3)\mid g\\ \lcm(g_1,g_2,g_3)=g}} \prod_{\substack{p\mid \bm{\ell}\\ p\nmid g}}\beta_{X^{\bm{\ell}},p}(h)\left|\prod_{p\mid g} \beta_{X^{\bm{\ell}},p}(h) - \beta_{\ell_1,g}(h)\beta_{\ell_2,g}(h)\beta_{\ell_3,g}(h)\right|.
\end{multline}
Since $p\mid g$ if and only if $p\mid g_j$ for some $j$, we have $\#\{j:p\mid \ell_j\}=1$ if and only if $p\nmid g$. We then note that we have 
\begin{multline}\label{eqn:maxprod}
\left|\prod_{p\mid g} \beta_{X^{\bm{\ell}},p}(h) - \beta_{\ell_1,g}(h)\beta_{\ell_2,g}(h)\beta_{\ell_3,g}(h)\right|\leq \max\left\{\prod_{p\mid g} \beta_{X^{\bm{\ell}},p}(h),\beta_{\ell_1,g}(h)\beta_{\ell_2,g}(h)\beta_{\ell_3,g}(h)\right\}\\
\leq \prod_{p\mid g} \max\left\{\beta_{X^{\bm{\ell}},p}(h), \beta_{X^{(p,1,1)},p}(h)^{\#\{j:p\mid \ell_j\}}\right\},
\end{multline}
where we use the symmetry $\beta_{X^{(p,1,1)},p}(h)=\beta_{X^{(1,p,1)},p}(h)=\beta_{X^{(1,1,p)},p}(h)$ in the last step.

The product after the sum in \eqref{eqn:gcdsplitlast} thus becomes 
\begin{multline}\label{eqn:maxprodsplit}
\prod_{\substack{p\mid \bm{\ell}\\ \#\{j:p\mid \ell_j\}=1}}\beta_{X^{(p,1,1)},p}(h)\prod_{\substack{p\mid g\\ \#\{j:p\mid \ell_j\}=2}}\max\left\{\beta_{X^{(p,p,1)},p}(h),\beta_{X^{(p,1,1)},p}(h)^2\right\}\\
\times \prod_{\substack{p\mid g\\ \#\{j:p\mid \ell_j\}=3}}\max\left\{\beta_{X^{(p,p,p)},p}(h),\beta_{X^{(p,1,1)},p}(h)^3\right\}.
\end{multline}
Without loss of generality, taking $z_0$ sufficiently large, we may assume that $p\nmid 2(m-2)(m-4)$ for every $p\mid \bm{\ell}$. Thus, using \eqref{eqn:beta4bound}, we may bound \eqref{eqn:maxprodsplit} against 
\begin{multline*}
\leq \prod_{\substack{p\mid \bm{\ell}\\ \#\{j:p\mid \ell_j\}=1}}\frac{2}{p}\prod_{\substack{p\mid g\\ \#\{j:p\mid \ell_j\}=2}}\frac{4}{p^2} \prod_{\substack{p\mid g\\ \#\{j:p\mid \ell_j\}=3\\ p\mid n}}\frac{1}{p(p-1)}\prod_{\substack{p\mid g\\ \#\{j:p\mid \ell_j\}=3\\ p\nmid n}}\frac{8}{p^3}\\
=\frac{1}{\ell_1\ell_2\ell_3} \prod_{p\mid \bm{\ell}}2\prod_{\substack{p\mid g\\ \#\{j:p\mid \ell_j\}=2}}2\prod_{\substack{p\mid g\\ \#\{j:p\mid \ell_j\}=3\\ p\mid n}}\frac{p^2}{2(p-1)}\prod_{\substack{p\mid g\\ \#\{j:p\mid \ell_j\}=3\\ p\nmid n}}4\\
\leq  \frac{2^{\omega\left(\ell_1\ell_2\ell_3\right)}4^{\omega(g)}}{\ell_1\ell_2\ell_3}\gcd\left(g,n,\ell_1,\ell_2,\ell_3\right)\leq \frac{2^{\omega(\ell_1)}}{\ell_1}\frac{2^{\omega(\ell_2)}}{\ell_2}\frac{2^{\omega(\ell_3)}}{\ell_3} 4^{\omega(g)}\gcd\left(g,n,\ell_1,\ell_2,\ell_3\right)
\end{multline*}
where in the second-to-last step we used
\[
\prod_{\substack{p\mid g\\ \#\{j:p\mid \ell_j\}=3\\ p\mid n}} p = \gcd\left(g,n,\ell_1,\ell_2,\ell_3\right)\qquad \text{ and }\qquad  \frac{p}{2(p-1)}\leq 2\leq 4.
\]
Hence we obtain that \eqref{eqn:gcdsplitlast} may be bounded 
\begin{equation}\label{eqn:gellsum}
\ll \sum_{\substack{g\mid P_{h}(z_0,z)\\ g\neq 1}}\sum_{\substack{\bm{\ell}\mid P_h(z_0,z)^3\\ g_{1,2}=\gcd(\ell_1,\ell_2)\mid g\\ g_{1,3}=\gcd(\ell_1,\ell_3)\mid g\\ g_{2,3}=\gcd(\ell_2,\ell_3)\mid g\\ \lcm(g_{1,2},g_{1,3},g_{2,3})=g}} \frac{2^{\omega(\ell_1)}}{\ell_1}\frac{2^{\omega(\ell_2)}}{\ell_2}\frac{2^{\omega(\ell_3)}}{\ell_3} 4^{\omega(g)}\gcd\left(g,n,\ell_1,\ell_2,\ell_3\right).
\end{equation}
Since $g_j:=\lcm(g_{j,k},g_{j,k'})\mid \ell_j$, we may make a change of variables $\ell_j\mapsto g_j\ell_j$ to obtain that \eqref{eqn:gellsum} equals
\begin{multline*}
\sum_{\substack{g\mid P_{h}(z_0,z)\\ g\neq 1\\ g_1,g_2,g_3\mid g\\ \lcm(\gcd(g_1,g_2),\gcd(g_1,g_3),\gcd(g_2,g_3))=g}} \hspace{-1.2cm}\frac{2^{\omega\left(g_1\right)+\omega\left(g_2\right)+\omega\left(g_3\right)}4^{\omega(g)}}{g_1g_2g_3}\gcd\left(n,g_1,g_2,g_3\right) 
\sum_{\substack{\bm{\ell}\mid P_h(z_0,z)^3\\ g_j\nmid \bm{\ell}\\ \gcd(\ell_j,\ell_k)=1}} \frac{2^{\omega(\ell_1)}}{\ell_1}\frac{2^{\omega(\ell_2)}}{\ell_2}\frac{2^{\omega(\ell_3)}}{\ell_3}\\
\leq \sum_{\substack{g\mid P_{h}(z_0,z)\\ g\neq 1\\ g_1,g_2,g_3\mid g\\ \lcm(\gcd(g_1,g_2),\gcd(g_1,g_3),\gcd(g_2,g_3))=g}} \hspace{-1.2cm}\prod_{p\mid g} \frac{2^{3}4p^{\min\{\ord_p(g_1),\ord_p(g_2),\ord_p(g_3)\}}}{p^{\ord_p(g_1)+\ord_p(g_2)+\ord_p(g_3)}} \sum_{\bm{\ell}\mid P_h(z_0,z)^3} \frac{2^{\omega(\ell_1)}}{\ell_1}\frac{2^{\omega(\ell_2)}}{\ell_2}\frac{2^{\omega(\ell_3)}}{\ell_3}.
\end{multline*}
We then note that 
\[
\ord_p(g_1)+\ord_p(g_2)+\ord_p(g_3)-\min\left\{\ord_p(g_1),\ord_p(g_2),\ord_p(g_3)\right\}= 2
\]
because $\lcm(\gcd(g_1,g_2),\gcd(g_1,g_3),\gcd(g_2,g_3))=g$. Writing 
\[
\sum_{\bm{\ell}\mid P_h(z_0,z)^3} \frac{2^{\omega(\ell_1)}}{\ell_1}\frac{2^{\omega(\ell_2)}}{\ell_2}\frac{2^{\omega(\ell_3)}}{\ell_3}=\prod_{p\mid P_h(z_0,z)}\left(1+\frac{2}{p}\right)^3,
\]
we obtain 
\begin{multline}\label{eqn:gsumpprod}
 \leq \sum_{\substack{g\mid P_{h}(z_0,z)\\ g\neq 1\\ g_1,g_2,g_3\mid g}} \prod_{p\mid g} \frac{2^{3}4}{p^2} \prod_{p\mid P_h(z_0,z)}\left(1+\frac{2}{p}\right)^3 =\sum_{\substack{g\mid P_{h}(z_0,z)\\ g\neq 1}}\sigma_0(g)^3 \prod_{p\mid g} \frac{2^{3}4}{p^2} \prod_{p\mid P_h(z_0,z)}\left(1+\frac{2}{p}\right)^3\\
\ll_{\varepsilon} \sum_{\substack{g\mid P_{h}(z_0,z)\\ g\neq 1}} g^{\varepsilon-2}\prod_{p\mid P_h(z_0,z)}\left(1+\frac{2}{p}\right)^3\leq \sum_{g\geq z_0} g^{\varepsilon-2}\prod_{p\mid P_h(z_0,z)}\left(1+\frac{2}{p}\right)^3,
\end{multline}
We then bound 
\begin{equation}\label{eqn:gsumbound}
\sum_{g\geq z_0} g^{\varepsilon-2}\ll \int_{z_0}^{\infty} x^{\varepsilon-2}dx\ll z_0^{\varepsilon-1}.
\end{equation}
Using \cite[(3.30) and (3.26)]{RosserSchoenfeld}, we then bound 
\begin{equation}\label{eqn:pprodbound}
\prod_{p\mid P_h(z_0,z)}\left(1+\frac{2}{p}\right)\leq \prod_{p\mid P_h(z_0,z)}\left(1+\frac{1}{p}\right)^2\leq \prod_{z_0< p\leq z}\frac{1}{\left(1-\frac{1}{p}\right)^2}\ll \frac{\log(z)^2}{\log(z_0)^2}.
\end{equation}
\begin{extradetails}
\begin{multline*}
\prod_{p\mid P_h(z_0,z)}\left(1+\frac{2}{p}\right)\leq \prod_{p\mid P_h(z_0,z)}\left(1+\frac{1}{p}\right)^2\leq \prod_{z_0< p\leq z}\frac{1}{\left(1-\frac{1}{p}\right)^2}\leq \frac{\prod_{p\leq  z_0} \left(1-\frac{1}{p}\right)^2}{\prod_{p\leq  z} \left(1-\frac{1}{p}\right)^2}\\
\leq \frac{\log(z)^2}{\log(z_0)^2}\left(1+\frac{1}{2\log^2(z_0)}\right)^2\left(1+\frac{1}{\log^2(z)}\right)^2\ll \frac{\log(z)^2}{\log(z_0)^2}.
\end{multline*}
\end{extradetails}
Plugging \eqref{eqn:gsumbound} and \eqref{eqn:pprodbound} into \eqref{eqn:gsumpprod} yields 
\[
\sum_{\substack{\bm{\ell}\mid P_h(z_0,z)^3\\ \exists (j,k): \gcd(\ell_j,\ell_k)\neq 1}}\Lambda_{\ell_1}^-\lambda_{\ell_2}^+\lambda_{\ell_3}^+\left(\prod_{p\mid\bm{\ell}} \beta_{X^{\bm{\ell}},p}(h)-\beta_{\ell_1,0}(h)\beta_{\ell_2,0}(h)\beta_{\ell_3,0}(h)\right)\ll_{\varepsilon} z_0^{\varepsilon-1}\frac{\log(z)^6}{\log(z_0)^6}.
\]
Plugging back into \eqref{eqn:gcdsplit}, we conclude that 
\begin{multline}\label{eqn:nogcd}
\sum_{\ell_1\mid P_h(z_0,z)}\Lambda_{\ell_1}^-\sum_{\ell_2\mid P_h(z_0,z)}\lambda_{\ell_2}^+\sum_{\ell_3\mid P_h(z_0,z)}\lambda_{\ell_3}^+\prod_{p\mid\bm{\ell}} \beta_{X^{\bm{\ell}},p}(h)\\
=\sum_{\ell_1\mid P_h(z_0,z)}\Lambda_{\ell_1}^-\prod_{p\mid \ell_1}\beta_{X^{(\ell_1,1,1)},p}(h)\left(\sum_{\ell_2\mid P_h(z_0,z)}\lambda_{\ell_2}^+\prod_{p\mid \ell_2}\beta_{X^{(1,\ell_2,1)},p}(h)\right)^2+O_{\varepsilon}\left(z_0^{\varepsilon-1}\log(z)^6\right).
\end{multline}
By \cite[(6.32)]{IwaniecKowalski}, we have 
\begin{equation}\label{eqn:lambda+bound}
\sum_{\ell_2\mid P_h(z_0,z)}\lambda_{\ell_2}^+\prod_{p\mid \ell_2}\beta_{X^{(1,\ell_2,1)},p}(h)\geq \prod_{\substack{z_0< p\leq z\\ p\nmid \sf(h)}} \left(1-\beta_{X^{(p,1,1)},p}(h)\right).
\end{equation}
We then use \eqref{eqn:beta4bound} to bound (for $z_0\geq 3$)
\[
\prod_{\substack{z_0< p\leq z\\ p\nmid \sf(h)}} \left(1-\beta_{X^{(p,1,1)},p}(h)\right)\geq \prod_{z_0< p\leq z} \left(1-\frac{2}{p}\right)\geq \frac{1}{\zeta(2)}\prod_{z_0< p\leq z} \left(1-\frac{1}{p}\right)^2.
\]
\begin{extradetails}
\begin{multline*}
\prod_{\substack{z_0< p\leq z\\ p\nmid \sf(h)}} \left(1-\beta_{X^{(p,1,1)},p}(h)\right)\geq \prod_{z_0< p\leq z} \left(1-\frac{2}{p}\right)=\prod_{z_0< p\leq z} \left(1-\frac{1}{p}\right)^2 \left(1-\frac{1}{(p-1)^2}\right)\\
\geq \frac{1}{\zeta(2)}\prod_{z_0< p\leq z} \left(1-\frac{1}{p}\right)^2.
\end{multline*}
\end{extradetails}
Using \cite[(3.30) and (3.26)]{RosserSchoenfeld}, for $z_0$ sufficiently large we obtain from \eqref{eqn:lambda+bound} that 
\begin{equation}\label{eqn:lambda+boundfinal}
\sum_{\ell_2\mid P_h(z_0,z)}\lambda_{\ell_2}^+\prod_{p\mid \ell_2}\beta_{X^{(1,\ell_2,1)},p}(h)\gg  \left(\frac{\log(z_0)}{\log(z)}\right)^2
\end{equation}
\begin{extradetails}
We first use \cite[(3.30) and (3.26)]{RosserSchoenfeld} to bound
\[
\prod_{z_0\leq p\leq z} \left(1-\frac{1}{p}\right)^2\geq\left(\frac{\log(z_0)}{\log(z)}\right)^2\left(1+\frac{1}{2\log^2(z)}\right)^{-2}\left(1+\frac{1}{\log^2(z_0)}\right)^{-2}.
\]
We thus obtain from \eqref{eqn:lambda+bound} that 
\[
\sum_{\ell_2\mid P_h(z_0,z)}\lambda_{\ell_2}^+\prod_{p\mid \ell_2}\beta_{X^{(1,\ell_2,1)},p}(h)\gg  \left(\frac{\log(z_0)}{\log(z)}\right)^2
\]
\end{extradetails}
Taking $\kappa=2$ and 
\[
K:=\zeta(2)\left(1+\frac{1}{2\log^2(z)}\right)^{2}\left(1+\frac{1}{\log^2(z_0)}\right)^{2},
\]
we conclude that for $z_0$ sufficiently large
\begin{equation}\label{eqn:Betaprodbound}
\prod_{\substack{z_0< p\leq z\\ p\nmid \sf(h)}} \left(1-\beta_{X^{(p,1,1)},p}(h)\right)\geq K^{-1}\left(\frac{\log(z_0)}{\log(z)}\right)^{\kappa}.
\end{equation}
By choosing $z_0$ sufficiently large, we can take $K$ arbitrarily close to $\zeta(2)$ from above, say $K<1.645$. By \cite[Theorem 6.1]{IwaniecKowalski}, we hence obtain
\begin{equation}\label{eqn:Lambda-boundfinal}
\sum_{\ell_1\mid P_h(z_0,z)}\Lambda_{\ell_1}^-\prod_{p\mid \ell_1}\beta_{X^{(\ell_1,1,1)},p}(h)>\left(1-e^{19-s}(1.645)^{10}\right)\prod_{\substack{z_0< p\leq z\\ p\nmid \sf(h)}} \left(1-\beta_{X^{(p,1,1)},p}(h)\right).
\end{equation}
One verifies that for $s\geq 24$ we have 
\[
1-e^{19-s}(1.645)^{10}>0.
\]
Plugging \eqref{eqn:lambda+boundfinal} and \eqref{eqn:Lambda-boundfinal} into \eqref{eqn:nogcd} and then using \eqref{eqn:Betaprodbound} yields
\begin{multline*}
\sum_{\ell_1\mid P_h(z_0,z)}\Lambda_{\ell_1}^-\sum_{\ell_2\mid P_h(z_0,z)}\lambda_{\ell_2}^+\sum_{\ell_3\mid P_h(z_0,z)}\lambda_{\ell_3}^+\prod_{p\mid\bm{\ell}} \beta_{X^{\bm{\ell}},p}(h)\\
\geq (1.645)^{-3}\left(1-e^{19-s}(1.645)^{10}\right)\left(\frac{\log(z_0)}{\log(z)}\right)^6+ O_{\varepsilon}\left(z_0^{\varepsilon-1}\log(z)^6\right).
\end{multline*}
Taking $z_0=\log(z)^{13}$, we have 
\begin{equation}\label{eqn:maintermsumfinalbound}
\sum_{\ell_1\mid P_h(z_0,z)}\Lambda_{\ell_1}^-\sum_{\ell_2\mid P_h(z_0,z)}\lambda_{\ell_2}^+\sum_{\ell_3\mid P_h(z_0,z)}\lambda_{\ell_3}^+\prod_{p\mid\bm{\ell}} \beta_{X^{\bm{\ell}},p}(h)
\gg \left(\frac{\log(z_0)}{\log(z)}\right)^{6}.
\end{equation}
Choosing $a$ so that $1-\frac{3}{2^a}+\frac{3}{2^{2a}}-\frac{1}{2^{2a}}>\frac{1}{3}$ and again using \cite[(3.30)]{RosserSchoenfeld}, we conclude that 
\[
W_{a,h}(z_0)\gg \frac{1}{\log(z_0)^3}.
\]
Thus the main term has a lower bound 
\begin{equation}\label{eqn:mainwithWah}
\gg \frac{\log(z_0)^3}{\log(z)^6} r_{\pgen(X^{\bm{1}})}(h).
\end{equation}
Combining \eqref{eqn:mainwithWah} with \eqref{eqn:cuspidalsumbound}, \eqref{eqn:SahETbound}, and \eqref{eqn:MainErrorFinal}, we conclude that 
\begin{multline}\label{eqn:almostdone}
S_{a,h}\left(\mathcal{A}_{\bm{1}},z\right)\gg \Bigg( \frac{\log(z_0)^3}{\log(z)^6} +O\bigg( \frac{\log(z_0)^{13}}{2^a} \frac{6^{\omega(\operatorname{sf}(h))}}{\operatorname{sf}(h)^2}\sigma_{-1}(\operatorname{sf}(h))\left(\frac{\log(z)}{\log(z_0)}\right)^6+\frac{\Delta^C}{e^{s_0}}\\
+\frac{H(n)^{4}\log(D_0)^6\log(\log(\sf(h)))^6}{\Delta^{\frac{2}{3}}} +\frac{\log(z_0)^2\log(\log(\sf(h)))^2}{\Delta^{B-\varepsilon}}\left(\frac{\log(z)}{\log(z_0)}\right)^6\bigg)\Bigg)r_{\pgen(X^{\bm{1}})}(h)\\
+O_{\varepsilon}\left((DD_0)^{\frac{207}{8}+\varepsilon}h^{\frac{231}{512}+\varepsilon}\right).
\end{multline}
Writing $z=h^{\theta}$, we choose $s=24$ and $D_0=h^{\varepsilon}$, so $DD_0\ll_{\varepsilon} h^{24\theta+\varepsilon}$ and  $s_0\gg_{\varepsilon} \frac{\log(h)}{\log(\log(h))}$. We also choose 
\[
\Delta=H(n)^6\log(h)^{\max\{\frac{13}{B},\frac{39}{2}\}}.
\]
We further assume that
\begin{equation}\label{eqn:sflowerboundassume}
2^{a}\operatorname{sf}(h)^{2-\varepsilon}>\log(z)^{13},
\end{equation}
Then, since (noting that $6^{\omega(\operatorname{sf}(h))}\sigma_{-1}(\operatorname{sf}(h))\ll \operatorname{sf}(h)^{\varepsilon}$)
\[
\frac{\log(z_0)^{13}}{2^a} \frac{6^{\omega(\operatorname{sf}(h))}}{\operatorname{sf}(h)^2}\sigma_{-1}(\operatorname{sf}(h))\left(\frac{\log(z)}{\log(z_0)}\right)^6\ll_{\varepsilon} \log(z_0)^{7}\log(z)^{-7+\varepsilon}\ll_{\varepsilon} \log(z)^{-7+\varepsilon},
\]
 \eqref{eqn:almostdone} becomes
\begin{multline}\label{eqn:cuspdone}
S_{a,h}\left(\mathcal{A}_{\bm{1}},z\right)\gg \Bigg( \frac{\log(z_0)^3}{\log(z)^6} +O\bigg( \frac{\Delta^C}{e^{s_0}}+\frac{H(n)^{4}\log(D_0)^6\log(\log(\sf(h)))^6}{\Delta^{\frac{2}{3}}}\\
 \hspace{1.5cm} +\frac{\log(z_0)^2\log(\log(\sf(h)))^2}{\Delta^{B-\varepsilon}}\left(\frac{\log(z)}{\log(z_0)}\right)^6\bigg)\Bigg)r_{\pgen(X^{\bm{1}})}(h)+O_{\varepsilon}\left(h^{621\theta+ \frac{231}{512}+\varepsilon}\right).
\end{multline}
We then note that $\log(\log(\sf(h)))\ll_{\varepsilon}\log(h)^{\varepsilon}$ and $\Delta\gg H(n)^6\log(h)^{\frac{39}{2}}$ together with $\log(D_0)\ll \log(h)$ implies that
\[
\frac{H(n)^{4}\log(D_0)^6\log(\log(\sf(h)))^6}{\Delta^{\frac{2}{3}}}\ll_{\varepsilon} \log(h)^{-7+\varepsilon},
\]
while $\Delta\gg \log(h)^{\frac{13}{B}}$ implies that 
\[
\frac{\log(z_0)^2\log(\log(\sf(h)))^2}{\Delta^{B-\varepsilon}}\left(\frac{\log(z)}{\log(z_0)}\right)^6\ll \log(h)^{-7+\varepsilon}\log(z_0)^{-4}.
\]
Since $s_0\gg_{\varepsilon} \frac{\log(h)}{\log(\log(h))}$ and (as in \cite[p. 95]{BrudernFouvry}, noting that $H(n)\ll \prod_{p|n}\left(1+p^{-\frac{2}{3}}\right)\ll\prod_{p|n}\left(1+p^{-\frac{1}{2}}\right)$)
\[
\log(\Delta)\ll\log(h)^{\frac{1}{2}},
\]
we conclude that \eqref{eqn:cuspdone} may be written as 
\begin{equation}\label{eqn:SahFinal}
S_{a,h}\left(\mathcal{A}_{\bm{1}},z\right)\gg \frac{\log(z_0)^3}{\log(z)^6}r_{\pgen(X^{\bm{1}})}(h)+O_{\varepsilon}\left(h^{621\theta+ \frac{231}{512}+\varepsilon}\right).
\end{equation}
We now use Lemma \ref{lem:X1lower} to obtain that this is 
\[
\gg_{m,\theta,\varepsilon} h^{\frac{1}{2}-\varepsilon} +O\left(h^{621\theta+ \frac{231}{512}+\varepsilon}\right).
\]
Taking 
\[
\theta= \frac{1}{12719}<\frac{25}{512\cdot 621},
\]
we have 
\[
621\theta+\frac{231}{512}<\frac{1}{2},
\]
and hence for $h$ sufficiently large we have
\[
S_{a,h}\left(\mathcal{A}_{\bm{1}},z\right)>0. 
\]
Hence in \eqref{eqn:sum3mgonal} we may take that odd $p\mid x$ implies that 
\[
p>h^{\theta}=\left(8(m-2)n+3(m-4)^2\right)^{\theta}\gg_m n^{\theta}=n^{\frac{1}{12719}}
\]
and $\ord_2(x)< a$ (with $a\geq 2$), and likewise the same property for $y$ and $z$. If 
\begin{equation}\label{eqn:ppowers}
\sum_{p>n^{\frac{1}{12719}}}\ord_p(x) = \alpha>0,
\end{equation}
then 
\[
p_m(x)\geq \frac{(m-2)x^2}{2}-\frac{|m-4|\cdot |x|}{2}\geq \frac{m-\frac{5}{2}}{2} x^2,
\]
where the last inequality holds because $|x|>n^{\frac{\alpha}{12719}}>2m$ for $n$ sufficiently large (depending on $m$) and $\alpha>0$. Then, letting $\beta$ be the maximum of the respective choices of $\alpha$ for $x$, $y$, and $z$ from \eqref{eqn:ppowers}, we have 
\[
n=p_m(x)+p_m(y)+p_m(z)\gg_{m} n^{\frac{2\beta}{12719}},
\]
and we conclude that 
\[
\beta\leq \frac{12719}{2},
\]
which implies that $\beta\leq 6359$. Overall, we obtain that $x$, $y$, and $z$ are divisible by at most $6359+a$ primes. In particular, if $\sf(h)$ is sufficiently large (i.e., $\sf(h)\gg \log(h)^{7}$), then by \eqref{eqn:sflowerboundassume} we may take $a=2$ and have at most $6361$ primes. 

Note further that 
\[
\#\{h<X: \sf(h)<\log(h)^7\}\leq \log(X)^7\sqrt{X}.
\]
Hence for a set of $h$ of density $1$ we have at most $6361$ primes.
\end{proof}

\end{document}